\documentclass[12pt,a4paper]{article}
\usepackage[a4paper,vmargin=3cm,hmargin=2cm]{geometry}
\usepackage[dvipsnames]{xcolor}
\usepackage[font={bf,it,footnotesize}]{caption}
\usepackage[OT2,OT1]{fontenc}
\usepackage[all]{xy}
\usepackage{algpseudocode}
\usepackage{algorithm}
\usepackage{amscd}
\usepackage{amsmath}
\usepackage{amssymb}
\usepackage{amsthm}
\usepackage{dsfont}
\usepackage{enumitem}
\usepackage{extarrows}
\usepackage{float}
\usepackage{hyperref}
\usepackage[latin1]{inputenc}
\usepackage{marvosym}
\usepackage{mathrsfs}
\usepackage{multicol}
\usepackage[new]{old-arrows}
\usepackage{pb-diagram, pb-xy}
\usepackage{tikz}
\usepackage{pgfplots}
\usepackage{tikz-cd}

\definecolor{redish}{HTML}{D62728}
\definecolor{blueish}{HTML}{1F77B4}
\definecolor{greenish}{HTML}{2CA02C}
\definecolor{orangeish}{HTML}{FF7F0E}
\definecolor{purpleish}{HTML}{9467DB}

\setlist[enumerate, 1]{
  leftmargin = \parindent+0.5em\relax, 
  align = left,
  labelwidth=\parindent+0.5em\relax,
  labelsep = 0pt
}

\newcommand{\sqed}{\hfill{$\square$}}

\newcommand\cyr{%
\renewcommand\rmdefault{wncyr}%
\renewcommand\sfdefault{wncyss}%
\renewcommand\encodingdefault{OT2}%
\normalfont\selectfont}
\DeclareTextFontCommand{\textcyr}{\cyr}

\DeclareMathOperator*{\esssup}{ess\,sup}

\newcommand{\loc}{{\mathrm{loc}}}
\newcommand{\IC}{\mathbb{C}}

\newcommand{\IN}{\mathbb{N}}

\newcommand{\IR}{\mathbb{R}}

\newcommand{\IZ}{\mathbb{Z}}

\newcommand{\cK}{\mathcal K}
\newcommand{\cL}{\mathcal L}

\newcommand{\cR}{\mathcal R}
\newcommand{\cT}{\mathcal T}

\newcommand{\df}{\mathrm{d}}

\renewcommand{\Re}{{\mathrm{Re}}}
\renewcommand{\Im}{{\mathrm{Im}}}
 
\newtheorem{thm}{Theorem}[section]

\newtheorem{cor}[thm]{Corollary}

\newtheorem{lem}[thm]{Lemma}
\newtheorem{prop}[thm]{Proposition}

\theoremstyle{definition}
\newtheorem{defn}[thm]{Definition}
\newtheorem{exa}[thm]{Example}

\newtheorem{rem}[thm]{Remark}

\hypersetup{
    colorlinks=false,
    pdfpagemode=FullScreen,
}

\newcommand{\appref}[1]{Appendix \ref{#1}}

\newcommand{\corref}[1]{Corollary \ref{#1}}
\newcommand{\defnref}[1]{Definition \ref{#1}}

\renewcommand{\eqref}[1]{$($\ref{#1}$)$}

\newcommand{\lemref}[1]{Lemma \ref{#1}}
\newcommand{\cosref}[2]{\hyperlink{#1}{#2}}
\newcommand{\propref}[1]{Proposition \ref{#1}}

\newcommand{\secref}[1]{Section \ref{#1}}
\newcommand{\thmref}[1]{Theorem \ref{#1}}

\renewcommand{\to}{\xrightarrow{\quad}}
\newcommand{\apc}{\rightarrow}
\renewcommand{\mapsto}{\longmapsto}

\usepackage{thmtools} 
\usepackage{thm-restate}
\usepackage{stackengine}
\newcommand\xrowht[2][0]{\addstackgap[.5\dimexpr#2\relax]{\vphantom{#1}}}
\usepackage[numbers,square]{natbib}
\usepackage{subfigure}

\date{\today}

\title{\bf Positivity and long-term behaviour of a diffusion model with measure-valued nonlocal reaction term: applications in bioscience and engineering}
\author{Xiao Yang$^{1^*}$, Qiyao Peng$^{2, 1}$, Sander C. Hille$^1$}
\date{{\it\footnotesize $^{1^*}$ Mathematical Institute, Faculty of Science, Leiden University, The Netherlands.\\
Correspondence: x.yang@math.leidenuniv.nl\\
$^{2}$ Mathematics for AI in Real-world Systems, School of Mathematical Sciences, Lancaster University, UK.\\}
\vskip 1em
\today
}

\begin{document}
\numberwithin{equation}{section}
\numberwithin{figure}{section}

\maketitle
\begin{abstract}
    \noindent The behaviour is investigated of solutions to a diffusion equation on the real line with nonlocal and singular reaction term, i.e., given by a Dirac source or sink at the origin. It gives a simplified representation of for example a control system that senses concentration at a distance, but ``intervenes" at the origin. Positivity of solutions (for positive initial conditions) cannot be guaranteed for all parameter settings in the model. We determine a parameter regime and conditions on the positive initial condition in terms of monotonicity and symmetry, that do allow us to conclude the positivity of the solution for all time. In addition, we provide conditions that ensure convergence of the system to a constant steady state (pointwise), outside the region of observation. Technically, we extensively use Laplace transform arguments to achieve these results. 

\end{abstract}

\section{Introduction}\label{cha:1}
In various settings in (bio)science and engineering a modeller is confronted with the question how to represent a system in which ``{\it entities}" (e.g., a chemical compound, type of biological cells) diffuse in the environment of a structure that can release and take up these entities, controlled by ``observations" of their abundance nearby. For instance, cell-to-cell communication through signalling molecules follows this general pattern, for which the mathematical work uses mostly continuum modelling, or so-called macroscopic models, where cells are modelled by cell density in the partial differential equation (PDE). These models consider the uptake of the compounds by describing ligand-receptor binding kinetics on the cell ``boundary", i.e., the cell membrane. The governing equations include the dynamics of the cells, receptors on the cell membrane and the concentration of signalling molecules (e.g., ligand); see \citep{Klotz1984,Finlay2020}. 

One may also think for example of the vascular system in animals that exchanges various compounds (e.g., oxygen, carbon dioxide, drugs) between blood and surrounding tissue, regulating the permeability of the blood vessel wall; or how immune cells spread to regions where they are most needed, signalled by so-called cytokines, e.g., immune response in wound healing known as the inflammation phase. In this phase of wound healing, immune cells such as macrophages and lymphocytes enter the blood vessels from other parts of body and leave the blood circulation from the broken vessels at the wound edge to clear the bacteria \citep{pena2024cellular}. Another example related to ``exchanging behaviour" between blood vessels and their direct environment is provided by cancer cell metastasis. In every metastatic journey, intravasation and extravasation are the fundamental steps. That is, the cancer cell enters the blood circulation (intravasation) from one part of body, then it exits at the other (extravasation) to develop the secondary tumour via the junction between the epithelial cells of the blood vessels \citep{Sznurkowska2022-ts,Di_Russo2024-zr}; see Fig. \ref{fig:1D_takeup} for an illustration. 

In chemical engineering, one may wish to control the concentration of a compound in a solution (or the temperature) as measured by a few localized sensors in the bath, by adjusting the concentration (or temperature) in a liquid flowing through a permeable tube system in the bath; see \cite{DUA01}.

A mathematical model of such a system would describe the dynamics of the density of entities (e.g., concentration of the compound) in the environment of the exchanging structure most naturally as a diffusion equation with suitable (dynamic) boundary conditions at the interface between structure and environment. If the structure has complicated or changing geometry in time, theoretical tractability easily becomes cumbersome, followed by various challenges in numerical simulations, such as computational efficiency and different numerical schemes to handle the changing and moving geometry (e.g., cutFEM \citep{Burman2014} and weighted shifted boundary method \citep{Xu2024}).

\begin{figure}[h!]
    \centering
    \includegraphics[width=\linewidth]{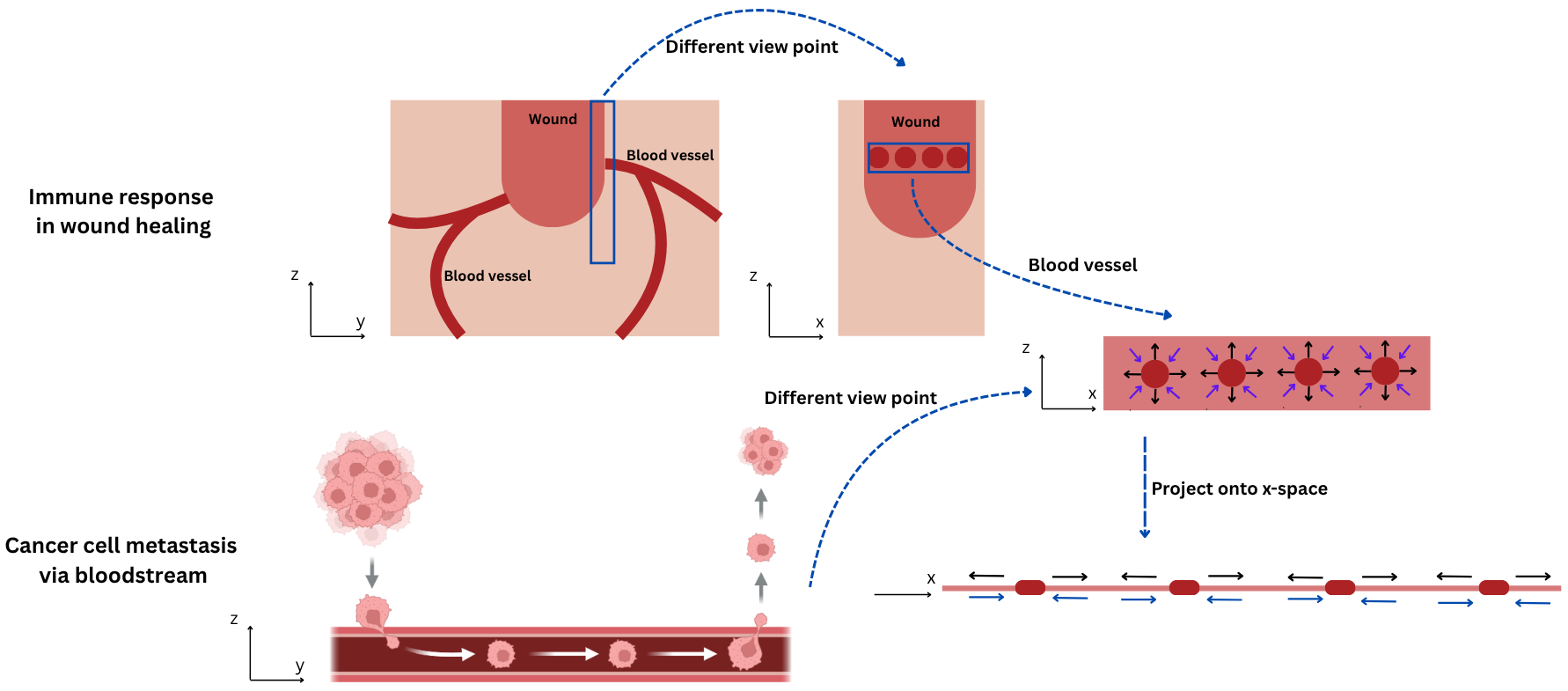}
    \caption{Schematic of immune cells migrating to wound region and cancer cell metastasis via the bloodstream. The cells enter the blood vessels from one part of the body and exit from the other part. Projecting onto x-space, the blood vessels and the trajectory of cells entering and exiting the blood vessels can be simplified to a one-dimensional model, as in System \eqref{eq: system considered u}--\eqref{def:Psi new}, where $u$ then stands for the cell density in these two biological examples.}
    \label{fig:1D_takeup}
\end{figure}

In previous work \citep{Yang2025, Peng2023, Peng2024-fv} we studied reduction of a model for cell-cell communication via diffusive compounds in a two-dimensional setting with ``spatially extended" cells by comparing its solution to that of a model on a cell-cell communication via diffusive compounds spatial domain with point sources. While in \citet{Peng2023, Peng2024-fv} cells are only secreting compounds, we introduced a nonlocal term in the work \citet{Yang2025} to model also the uptake of compounds by cells.

If the exchanging structure has symmetries, reduction of the dimensionality of the computational domain may also be made, resulting in a substantial increase in computational efficiency -- and often an improvement in theoretical tractability. For example, if the above mentioned vasculature or tubing in the bath consists of parallel tubes over a sufficiently long stretch, a three-dimensional description may be replaced well by a diffusive model for the density in a two-dimensional cross section orthogonal to the tubing. If the tubing is in a plane, a further reduction to a one-dimensional description for the average concentration perpendicular to this plane is reasonable (see Fig. \ref{fig:1D_takeup} for a schematic that presents two aforementioned biological examples and shows how we reduce to a one-dimensional model). In the latter reduction, the connectedness of the spatial domain must be maintained to retain a reasonable representation of reality, since diffusive movement from one side of a tube to the other without moving through, must be possible. As a result, in the one-dimensional case, the spatial exclusion model that we have been working on in two dimensions \citep{Yang2025,Peng2023,Peng2024-fv}, will not be valid and in this study, therefore, we only focus on the point source model.

These considerations of model reduction, motivated by the aforementioned examples from biology and engineering, led us to study the following ``prototypical" initial value problem, which is central in the current paper. It concerns an isotropic diffusion equation on $\IR$ with a nonlocal and non-autonomous reaction term that is singular, in the sense that it is Dirac-measure-valued:
\begin{equation}\label{eq:system considered}
    \left\{\begin{alignedat}{2}
        \partial_t v &=  D\partial_x^2 v + \widetilde{\Psi}_t[v]\delta_0, \quad&& \mbox{for } (t,x)\in \IR^+\times\IR,\\
        v(0,x) &= v_0(x), && \mbox{for } x\in\IR,
    \end{alignedat}\right.
\end{equation}
with $D>0$. The intensity $\widetilde{\Psi}_t[v]\in\IR$ of the Dirac point source (or sink, if the intensity is negative) is allowed to depend nonlocally on the solution to allow for control by remote sensors or to reflect that the exchanging structure that ``disappeared" in the reduction did have a spatial extent. Moreover, the requirement of maintaining connectivity of the spatial domain and the possibility of diffusing ``around" the exchanging structure in the one-dimensional reduced model is met.  

Well-posedness of equations with measure-valued reaction terms more general than System \eqref{eq:system considered} in various types of Sobolev spaces has been studied by Amann and Quitner (see \citep{Amann1995,Amann2001,AMQ:04}). Specification of their results to our case -- with an appropriate regularity assumption on the map $(t,v)\mapsto \widetilde{\Psi}_t[v]$ -- yields existence of so-called {\it mild solutions} to semilinear nonlocal Equation \eqref{eq:system considered} in the Sobolev space $H^1(\IR)$. The numerical simulation of diffusion equations with a reaction term that is measure-valued and concentrated on a lower-dimensional manifold requires care. First results in these directions can be found in \citet{Peng2024-fv,Gjerde2019,Asghar2024-ii, Ashyraliyev2008, Koeppl_2014}. Also, one has started to study inverse problems (parameter estimation) for such models \citep{Qiu2025}. Equations with nonlocal reaction terms can be found in e.g., \citep{Banerjee2020-bt,Banerjee2022}

A key point of attention for the modeller is, what functional to pick for $(t,v)\mapsto \widetilde{\Psi}_t[v]$ such that solutions (in a suitable sense) to Equation \eqref{eq:system considered} exist and depend continuously on the data (well-posedness) while assuring close correspondence of this solution $v$ to the average concentration over appropriate sections in the original three-dimensional system. One important criterium for correspondence is, that positivity of the solution is maintained, for positive initial condition $v_0$, because only positive densities are physically relevant, of course. Establishing positivity of the solution cannot be approached by means of classical techniques employed for (local, measure-valued) reaction-diffusion equations (see e.g., \cite{Sikic:1994,Canizo_ea:2012,Hille_ea:2025}) because of the nonlocality of $\widetilde{\Psi}_t$. 

This paper expands on work that considered (in two-dimensional space) the diffusion equation with forced point sources or sinks -- i.e., independent of the solution (cf. \cite{Yang2025}). Here, we do consider feedback from the solution. The main mathematical challenges that are faced in this novel setting and for which appropriate results are provided are:
\begin{itemize}[leftmargin=*]
    \item Exhibiting an approach towards positivity of solutions to this particular type of nonlocal equations that allows providing conditions on the model data that guarantee positivity of the solution -- generally not on all of $\IR$ unfortunately -- but on the part of $\IR$ that corresponds to the spatial domain outside the exchanging structure. It involves Laplace transform and a detailed analysis of positivity of specific special functions through this transform (see Section \ref{Sec:Laplace}).

    Positivity can be converted into a result on having control on the solution, keeping it below a desired threshold; see Section \ref{subsec: Positivity and control problem} below.
    
    \item Providing conditions on $\widetilde{\Psi}_t[v]$ such that the mild solution will converge to a steady state -- in an asymptotic sense (and appropriate sense of convergence).
\end{itemize}

Our approach to both employs a Laplace transform argument. As we also view the positivity as ``controlling the solution below a desired threshold", it brings feedback control systems in mind, in which the variable being controlled is measured and compared with a preset target. 

In a linear feedback system, by Laplace transform the study of its behaviour boils down to studying the so-called {\it transfer function}. Consider for example the system depicted in Fig. \ref{fig:feedback}. 
\begin{figure}[h!]
    \centering
    \includegraphics[scale=0.3]{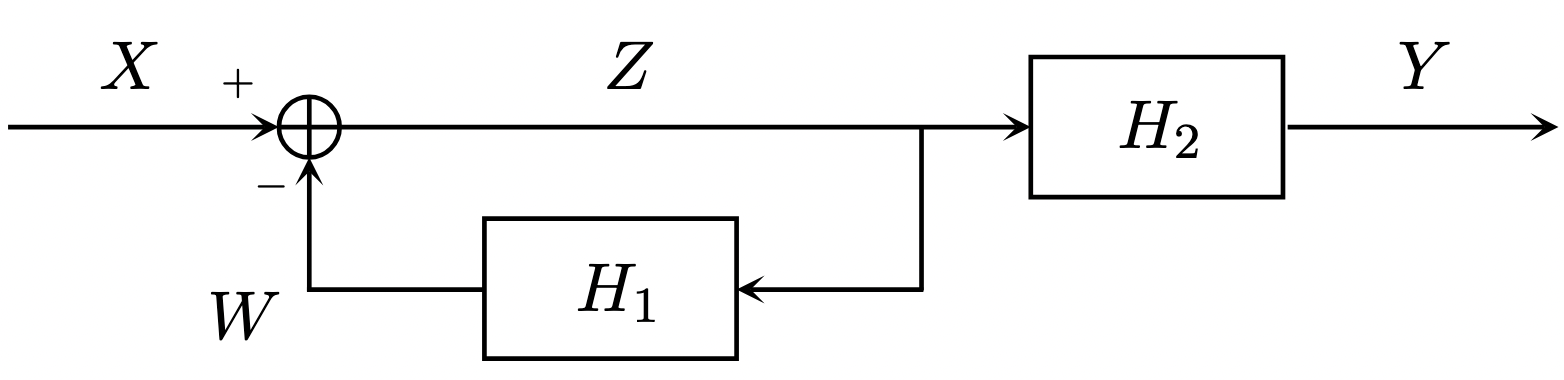}
    \caption{Feedback block diagram that represents the dynamics of $u_+$ in s-domain in Section \ref{sec:expression of u+ s-domain}. For the purpose of expressing the control system in a general setting, we use the symbols that are not relevant to our study.}
    \label{fig:feedback}
\end{figure}

The Laplace transformed system variables, denoted $X, Y, Z, W$, are related by the algebraic equations (for holomormic functions):
\begin{equation*}
\left\{
    \begin{aligned}
    Z&=X-W,\\
    W&=H_1Z,\\
    Y&=H_2Z.
\end{aligned}
\right.
\end{equation*}
The first two equations result in 
\[
    Z=\frac{X}{1+H_1},\qquad \mbox{hence}\quad  X=(1+H_1)Z.
\]
The transfer function for this feedback system between input $X$ and output $Y$ is then 
\[
    H:=\frac{Y}{X}=\frac{H_2}{1+H_1}.
\]
It will be similar transfer functions of which the study is essential for obtaining the positivity results for our model; see Section \ref{sec:expression of u+ s-domain}. In particular, the key idea is, that one can study the behaviour of the solution by considering a convenient system with the same transfer function. 

\subsection{Mathematical assumptions}
\label{sec:mathass}
We make the following fundamental assumptions and choices for the intensity functional $\widetilde{\Psi}_t[v]$. First, we assume that it depends on $v$ only through the solution at time $t$: $v_t:=v(t,\cdot)$. Thus, no time delay is considered. Next, the nonlocal dependence on $v_t$ is limited to its value (``measurement") at two distinct points, taken symmetrically around the Dirac point $x=0$ for convenience, say at $x=\pm R$. Note that point evaluation is continuous in $H^1(\IR)$, because of the Sobolev embedding of $H^1(\IR)$ into $C_b^0(\IR)$ (see \cite{Adams:2003}, Theorem 4.12, p.85). Hence, we assume:
\begin{quote}{\it 
\begin{enumerate}
    \item[(R1)] \;There is a function $(t,x_1,x_2)\mapsto f_t(x_1,x_2)$ such that
\[
    \widetilde{\Psi}_t[v] = f_t\bigl( v_t(-R), v_t(R)\bigr).
\]
Since the modelled system is thought to be symmetric around the Dirac point $x=0$, $f_t$ is symmetric: $f_t(x_1,x_2)=f_t(x_2,x_1)$ for all $x_1,x_2\in\IR$ and all $t\geqslant0$.
\end{enumerate}}
\end{quote}
Moreover, we assume the existence of a (positive) ``{\it asymptotically steady state}': 
\begin{quote}{\it 
\begin{enumerate}
    \item[(R2)] \;There exists $\bar{v}\geqslant 0$ such that 
\begin{equation}
    f_t(\bar{v},\bar{v})\apc 0,\quad\mbox{as } t\apc+\infty.
\end{equation}
\end{enumerate}}
\end{quote}
This will ensure that the solution $v_t$ to System \eqref{eq:system considered} will converge (in appropriate sense) to the constant function $\bar{v}$ as $t\to+\infty$. Note that the constant functions are not in the state space $H^1(\IR)$. Moreover, generally, System \eqref{eq:system considered} is not autonomous. Hence, it will not have a steady state in the usual sense.
Since we wish to focus on establishing an approach for proving positivity of solutions to nonlocal equations of the class studied, we want to take $f_t$ as simple as possible: 
\begin{quote}{\it
\begin{enumerate}
    \item[(R3)] \;$f_t\in C^2(\IR^2)$ and $\nabla f_t(x,x)$ is independent of $t$.
\end{enumerate}}
\end{quote}
Note that the symmetry of $f_t$ from (R1) implies (with the first part of (R3)) that $\partial_{x_1} f_t=
\partial_{x_2}f_t$ on the diagonal $\{(x,x)\mid x\in \IR\}$. Then, for any $c\in\IR$, and writing $u:=v-c$ for the deviation of $v$ from $c$, one has
\begin{align}
    \widetilde{\Psi}_t[v] &=f_t(c,c) + \nabla f_t(c,c) \left[ \begin{array}{c} {u_t}(-R) \\ u_t(R)\end{array}\right]+\mathrm{h.o.t.} \nonumber\\
    & \approx\Phi_c(t) + a_c\bigl[ u_t(-R) + u_t(R) \bigr]\label{eq:specific intensity this paper}
\end{align}
for small $u_t(\pm R)$, for some $a_c\in \IR$ that is independent of time, because of the second part of Assumption (R3). 

\subsection{Question of positivity of solution in relation to a control problem}
\label{subsec: Positivity and control problem}

Consider System \eqref{eq:system considered} with $\widetilde{\Psi}_t[v]$ satisfying (R1) and (R3). Assume that it has the linearized form in Expression \eqref{eq:specific intensity this paper}. The constant $c\in\IR$ may represent a particular threshold value for the state variable $v$. In an application, the equations can represent a system in which the value of the state variable at position $x\in\IR$ are controlled by change at $x=0$ only, based on deviation of the state variable from a desired overall value $c$ (i.e., ``$u$") ``measured" at location $x=-R$ and $x=R$.  A related control problem then is:

\begin{quote}{\it
    For given $\Phi_c(t)$ and $v_0(x)\leqslant c$, is it possible to find a parameter value $a_c\in\IR$ such that $v$ remains below the threshold $c$ for all time on (at least) the spatial domain $\IR\smallsetminus(-R,R)$?}
\end{quote}

In the context of this problem it is then convenient to define $u:=c-v$, as above, such that the condition $v\leqslant c$ is satisfied if we show positivity of $u$ on $\IR\smallsetminus(-R,R)$. Then, $u$ satisfies
\begin{equation}
\label{eq: system considered u}
    \left\{\begin{alignedat}{2} 
        \partial_t u &=  D\partial_x^2 u + \Psi_t[u]\delta_0, \quad&& \mbox{for } (t,x)\in \IR^+\times\IR,\\
        u(0,x) &= u_0(x),  && \mbox{for } x\in\IR,
    \end{alignedat}\right.
\end{equation}
with $u_0 := c-v_0\geqslant 0$ and
\begin{equation}\label{def:Psi new}
    \Psi_t[u] := \Phi(t) - a\bigl[ u(t,-R) + u(t,R)].
\end{equation}
Natural conditions are that $\Phi(t)\geqslant 0$, representing that the system tends to increase $u$, i.e., to move $v$ to a lower value than $c$. A parameter value $a\leqslant 0$ represents yet another tendency for the system to move $v$ below $c$. It is clear intuitively, that then $u$ will remain positive for all time when $u_0\geqslant 0$. The interesting case is, whether positivity of $u$ can be guaranteed for some $a>0$. This represents a situation in which the system tends to increase $v$, based upon the measurements at $x=\pm R$, bringing it ``close" to $c$, while at the same time ``steadily" decreasing $v$ through the forcing $\Phi(t)$. Thus, in the setting $a>0$ there is  negative feedback.

\subsection{Non-dimensionalization of the model}

We shall be concerned with System \eqref{eq: system considered u} -- \eqref{def:Psi new} from now on. It is mathematically convenient to work with a dimensionless model with a minimal number of parameters.

The state variable $u$ represents compound density, hence it has dimension $[u] = NL^{-1}$. The intensity $\Psi[u]$ of the Dirac source therefore must have dimension of a number flux, i.e., $NT^{-1}$, and $\Phi(t)$ then has the same dimension. We conclude that the parameter $a$ in Equation \eqref{def:Psi new} that characterizes the uptake from the environment has the dimension of a permeability (or speed): $[a] = LT^{-1}$.

We non-dimensionalize System \eqref{eq: system considered u} - \eqref{def:Psi new} on the diffusive time-scale corresponding to the length-scale of the structure's radius $R$. That is, new non-dimensional space and time variables become
\[
    \xi := \frac{x}{R},\qquad \tau:= \frac{t}{t^*}\quad \mbox{with } t^*:= \frac{R^2}{D}.
\]
Let $u^*$ be a reference density, fixed later. Define
\[
    \hat{u}(\tau,\xi) := \frac{u(t,x)}{u^*},\qquad \hat{u}_0(\xi):= \frac{u_0(x)}{u^*},\qquad \hat{\Psi}[\hat{u}](\tau) := \frac{t^*}{u^*R} \Psi[u](t).
\]
Then System \eqref{eq: system considered u} - \eqref{def:Psi new} becomes the dimensionless system
\begin{equation}
    \label{eq: dimensionless system}
    \left\{\begin{alignedat}{2}
    \partial_\tau \hat{u} &= \partial^2_\xi \hat{u} +\hat{\Psi}[\hat{u}](\tau)\delta_0, \quad && (\tau,\xi)\in\IR^+\times\IR, \\
    \hat{u}(0,\xi) &= \hat{u}_{0}(\xi), && \text{at } \tau=0,  \xi\in\IR,
    \end{alignedat}\right.
\end{equation}
where
\begin{equation}
\label{eq:dimensionless Psi}
    \hat{\Psi}[\hat{u}](\tau):=\hat{\Phi}(\tau)-\hat{a}\bigl[\hat{u}(\tau,1)+\hat{u}(\tau,-1)\bigr]
\end{equation}
with
\[
    \hat{\Phi}(\tau) := \frac{t^*}{u^*R} \Phi(t),\qquad \hat{a} := \frac{t^* a}{R} = \frac{aR}{D}.
\]
We shall assume later, that $\Phi_\infty:=\lim_{t\apc+\infty} \Phi(t)\geqslant 0$ exists. One may then choose the reference density $u^*$ such that $\lim_{\tau\apc+\infty} \hat{\Phi}(\tau)=1$, although this is not strictly required. 

From this point we shall work on the dimensionless model. For ease of presentation, we shall drop the accents and return to write $t$ and $x$ for the time and spatial variable. Thus, we work with System \eqref{eq: system considered u} -- \eqref{def:Psi new} as if we fixed $D=1$ and $R=1$. The non-dimensional variables $\xi$ and $\tau$ were introduced in this section only for the discussion of non-dimensionalization. Elsewhere in the paper they may appear with a different meaning.

\subsection{Structure of the manuscript}
The paper is structured as follows. Section \ref{sec:overview of results} displays an overview of the key results of this study. The positivity of the solution to System \eqref{eq: system considered u} -- \eqref{def:Psi new} is proved by utilising the Laplace transform as detailed in Section \ref{Sec:Laplace}. Section \ref{Sec:Pos} then theoretically defines the regions of parameters that guarantee the positivity of the solution. Furthermore, we discuss the existence of the steady state of System \eqref{eq: system considered u} -- \eqref{def:Psi new} in Section \ref{Sec_steady_state}. We verify our theoretical results with numerical simulations in Section \ref{Sec_numerical_results}. Conclusions are delivered in Section \ref{Sec_conclusions}. \cosref{app}{Appendices} contains some preliminary results, detailed computations, and clarifications of potentially ambiguous notations.

\section{Overview of main results}
\label{sec:overview of results}
In \cite{Yang2025}, we have discussed the well-posedness of a system analogous to System \eqref{eq: system considered u} -- \eqref{def:Psi new} in two spatial dimensions. In that case and here, this question can be addressed completely by the results obtained by Amann and Quitner \cite{Amann1995,Amann2001,AMQ:04}. The concept of solution used there is that of distributional solution. It is, however, shown that for the particular type of reaction term in our equation, this distributional solution is a so-called {\it mild solution} (see \cite{Amann2001}, \cite{AMQ:04}, Theorem 3.4, p. 1061). That is, $u$ is a continuous temporal trajectory $u:\IR_0^+\to X$, where $X$ is a suitable Sobolev space on $\Omega$, that solves the abstract Cauchy problem in $X$:
\begin{equation}\label{eq:cauchyproblem}
\dot u(t)= \mathfrak A u(t) + \mathfrak F(u(t),t),\quad u(0)=u_0\in X,
\end{equation}
where $\dot u$ is the derivative with respect to time and $(\mathfrak A,\mathcal{D}(\mathfrak A))$ is the generator of a strongly continuous semigroup $(e^{\mathfrak At})_{t\geqslant 0}$ on $X$. The solution is called mild, because it solves the formally integrated version of \eqref{eq:cauchyproblem}, that is the Variation of Constants Formula
\begin{equation}\label{eq:semigroupmild}
u(t)=e^{\mathfrak A t}u_0+\int_0^te^{\mathfrak A (t-\tau)}\mathfrak F(u(\tau),\tau)\df\tau
\end{equation}
(cf. also Section 2.1 of \cite{Yang2025}). The integral must be interpreted as a Bochner integral in $X$. For our system $\mathfrak{F}$ will map into a larger space that includes the Dirac measure $\delta_0$, but the analytic semigroup defined by the Laplace operator (on a suitable domain) will regularize it ``instantaneously" to have the integrant in $X$ for all $\tau\neq t$.

We have the following well-posedness result that was not covered by the previous paper \cite{Yang2025}. comparing to Theorem 1 in it, we are able to obtain a result with improved regularity for $u$. This is because the domain $\IR$ has dimension one, and has an empty boundary.

\begin{thm}[Well-posedness]\label{thm:wellposereduced}
 Let $a\in\IR$. System \eqref{eq: system considered u} -- \eqref{def:Psi new} has a unique mild solution $u\in C^0(\IR_0^+,H^1(\IR))$, provided $u_{0}\in H^1(\IR)$ and $\Phi\in L^\infty_{\mathrm{loc}}(\IR_0^+)$. Moreover, $u$ depends continuously on the initial condition $u_{0}$, for the topology of uniform convergence on compact sets.
\end{thm}
\begin{proof} We provide only a sketch, as the result derives from \cite{Amann2001,AMQ:04} following the same idea and principle as the proof of Theorem 2 of \cite{Yang2025}. It consists of appropriately picking the exponents of the Sobolev spaces in the following map
\[
    \mathfrak {F}:C^0([0,T],W^{1,p}(\IR))\to L^r([0,T],W^{\sigma-2,p}(\IR))),
    \quad u \mapsto \bigl(\;t\mapsto \Psi[u](t)\delta_0\bigr).
\]
First of all, we have $\delta_0\in H^{\sigma-2}(\IR)$ for any $\sigma-2<-\frac{1}{2}$, namely, $\sigma<\frac{3}{2}$. Second, the boundary of the domain $\IR$ is empty. Having noticed these two points, we can follow Chapter 3 in \cite{AMQ:04} to conclude that it is possible to take $p=2$. Hence, the solution is in $W^{1,2}(\IR)=:H^1(\IR)$.
\end{proof}

We shall work with this notion of mild solution in the remainder of the paper. A mild solution is a weak solution, in the usual sense. In this study, the fundamental conditions on $u_0$ and $\Phi$ are $u_{0}\in H^1(\IR)$ and $\Phi\in L^\infty_{\mathrm{loc}}(\IR_0^+)$ is of at most exponential order (to have Laplace transform exist), respectively. Hence, these conditions will be assumed for the rest of the theorems in this manuscript, which may not always be stated explicitly in every statement.

Our focus here is on establishing positivity and long-term behaviour of solutions of the model. Due to the singularity introduced by the Dirac delta distribution and the nonlocality of the reaction term, classical theories unfortunately cannot be applied. Regarding the positivity of the solution, it is evident that the solution remains positive everywhere for positive $u_0$, whenever $\Phi(t)\geqslant 0$ and $a\leqslant 0$, because the perturbation $\mathfrak F$ is then positive on positive functions (see \cite{Sikic:1994, Hille_ea:2025}). 

The interesting case is when the negative feedback occurs, i.e., $a\geqslant0$ . Our main result on this case is:

\begin{thm}[Positivity under monotonicty condition]\label{thm:positivity}
    Suppose in System \eqref{eq: system considered u} -- \eqref{def:Psi new} that initial condition $0\leqslant u_{0}\in H^1(\IR)$ and $0\leqslant\Phi(t)\in L^\infty_\loc(\IR_0^+)$ is of at most exponential order. If $\hat{a}= \frac{aR}{D}\in[0, e^{-1}]$ and if $u_{0}$ is increasing on $(-\infty,0)$ while decreasing on $\IR^+$, then the mild solution $u(t)\in H^1(\IR)$ to the system is nonnegative for $t\geqslant 0$ and a.e. $x\in\IR\smallsetminus(-R,R)$.
\end{thm}

Our other main result on refined behaviour of solutions is the existence of an asymptotic steady state:

\begin{thm}\label{thm:steadystate}
    Suppose in the System \eqref{eq: system considered u} -- \eqref{def:Psi new} that $\Phi(t)\geqslant 0$ and initial condition $u_{0}\geqslant0$ as in \thmref{thm:positivity}. Assume moreover, that $\lim_{t\apc+\infty}\Phi(t)=\Phi_{\infty}\in\IR$ exists and that $\Phi-\Phi_{\infty}\in L^1(\IR_0^+)$. If $\hat{a}=\frac{aR}{D}\in(0,e^{-1}]$, then the solution $u$ to the system converges pointwise to the value $\displaystyle\frac{\Phi_\infty}{2\hat{a}}$ everywhere.
\end{thm}

The main result on positivity, Theorem \ref{thm:positivity}, is complemented by the following, which will turn out to be less involved to prove. However, it requires a symmetry condition on the initial condition.
\begin{prop}[Positivity under symmetry condition]\label{prop:Positivity with symmetry}
    Suppose in System \eqref{eq: system considered u} -- \eqref{def:Psi new} that $\Phi(t)\geqslant 0$ and initial condition $u_{0}\geqslant0$ as in \thmref{thm:positivity}. If $\hat{a}=\frac{aR}{D}\in(0,e^{-1}]$ and $u_0$ is symmetrical on $\IR\smallsetminus (-\frac{1}{2}R,\frac{1}{2}R)$ around $x=0$, then the mild solution $u(t)\in H^1(\IR)$ to the system is nonnegative for $t\geqslant 0$ and a.e. $x\in\IR\smallsetminus(-R,R)$.
\end{prop}

\begin{rem}{\it
  Within the range of $a$ for which positive solutions on the domain $\IR\smallsetminus(-1,1)$ can be expected for suitable initial conditions according to our main results, it is still possible to construct examples of initial conditions that yield solutions that eventually become negative on that domain. Without getting into details, let us sketch such a construction. Take an initial condition $u_0$ that is highly asymmetrical on $\IR\smallsetminus(-1,1)$: e.g., $u_0$ vanishes identically on $(-\infty,\frac{1}{2})$, while on $[1,+\infty)$ it takes very large values, but going quickly to 0 on $[\frac{1}{2},1]$. Thus, it violates the monotonicity condition on $\IR^+$ of \thmref{thm:positivity}. As a consequence of this strong asymmetry and violation of the mentioned monotonicity, during the diffusion process, the point source drives the solution $u$ below zero.}
\end{rem}

\section{Employing Laplace transform for proving positivity}
\label{Sec:Laplace}
Due to the nonlocal term in the reaction-diffusion equation, classical methods to obtain positivity of solutions cannot be applied to System \eqref{eq: system considered u} -- \eqref{def:Psi new}. Moreover, it can be anticipated that the nonlocal character inhibits the {\it global} positivity of the solution. However, since that system is intended as approximation of a model with a spatially extended object that secretes and takes-up compounds from its environment, only the values and behaviour of the solution in our model on the part of space that is ``external" to the object is of interest. As the overview of main results shows, we can guarantee positivity of the solution on this external space in a particular regime for the uptake parameter $a$.

To do so, in this work, we utilize extensively the Laplace transform, enabled by the essentially linear character of the uptake. In this section we shall briefly recall the results from Laplace Transform Theory that are key to our approach. Our main reference for these is \citet{Arendt-Batty}.

Let $f\in L_{\loc}^1(\IR^+_0)$. One says that $f$ has Laplace transform $F:=\cL\{f\}$ if for some $s\in\IC$ the limit
\begin{equation}\label{eq:Laplace integral def}
    F(s) = \cL\{f\}(s) := \lim_{T\apc+\infty}\int_0^{T}e^{-st}f(t)\df t
\end{equation}
exists. The {\it abscissa of convergence} of $f$ is
\[
    \mathrm{abs}(f) := \inf \bigl\{ \Re(s)\mid s\in\IC, F(s) \mbox{ exists}\bigr\}.
\]
With the convention that $\inf\emptyset = +\infty$, $f$ has a Laplace transform if and only if $\mathrm{abs}(f)<+\infty$. In that case, the Laplace integral \eqref{eq:Laplace integral def} converges for any $s\in\IC$ with $\Re(s)>\mathrm{abs}(f)$ (see \cite{Arendt-Batty}, Proposition 1.4.1). If $f$ is of {\it exponential order}, i.e., there exists $\omega\in\IR$ and $M\geqslant 0$ such that
\[
    |f(t)| \leqslant Me^{\omega t}\quad \mbox{for a.e. } t\geqslant 0,
\]
then $\mathrm{abs}(f)\leqslant \omega$.   $F(s)$ is a holomorphic function on $\{s\in\IC\mid\Re(s)>\mathrm{abs}(f)\}$ (see \cite{Arendt-Batty}, Theorem 1.5.1). 

The function essentially uniquely determines its Laplace transform, and vice versa:
\begin{thm}[Uniqueness Theorem]\label{thm:uniquness Laplace}
Let $f,g\in L^1_\loc(\IR_0^+)$ with $\mathrm{abs}(f)<+\infty$ and $\mathrm{abs}(g)<+\infty$. Let $r_0>\max\bigl(\mathrm{abs}(f),\mathrm{abs}(g)\bigr)$. If $\mathcal{L}\{f\}(r)=\mathcal{L}\{g\}(r)$ for all $r\in (r_0,+\infty)$, then $f=g$.
\end{thm}
A slight rephrasing is given by
\begin{cor}\label{thm:lapuniq}
    Let $f\neq g\in L^1_{\mathrm{loc}}(\IR_0^+)$.  If $\cL\{f\}$ and $\cL\{g\}$ exist, then $\cL\{f\}, \cL\{g\}$ do not coincide for some $s\in\IC$.
\end{cor}
\noindent Thus, the Laplace transform is invertible on the set 
\[
    \mathcal{H}_{hs}^{\mathcal{L}} := \mathcal{L}\bigl\{ \{f\in L^1_\loc(\IR_0^+)\mid \mathrm{abs}(f)<+\infty\}\bigr\},
\]
which consists of functions $\IC\to\IC$ that are holomorphic on some halve space. One says that $\cL^{-1}\{F\}:=f$ is the \textit{inverse Laplace transform} of $F$, and that $f$ and $F$ are a {\it Laplace transform pair}. 
The range $\mathcal{H}_{hs}^{\mathcal{L}}$ is hard to describe precisely, but  a class of holomorphic functions can be identified that is obtained through Laplace transform: 
\begin{thm}[Paley-Wiener]\label{thm:lapinvexist}
    Let $\omega\in\IR$ and $F:\{s\in\IC\mid\Re(s)>
    \omega\}\to\IC$ be a holomorphic function. If
    $$\sup_{\sigma>\omega}\int_{\sigma-i\infty }^{\sigma+i\infty}|F(s)|^2\df s<+\infty,$$
    then there exists $f\in L_{\mathrm{loc}}^1(\IR_0^+)$ and $C>0$ such that $\cL\{f\}=F$ and $|f(t)|\leqslant Ce^{\omega t}$.
\end{thm}
Proofs of Theorems \ref{thm:lapuniq} and \ref{thm:lapinvexist} can be reduced immediately from the results in \citet{Arendt-Batty}, Theorem 1.7.3 and Theorem 1.8.3, respectively.

Thus, Laplace transform converts functions from a ``time domain" (often named as $t-$domain) to a complex frequency domain (often named as $s-$domain). For any function $F(s)$ in the $s-$domain, if it is the Laplace transform of a function $f(t)$ in $t$-domain, then $F(s)$ is a holomorphic function on some complex right half-plane $\{s \in \IC \mid \Re(s) > \omega\}$. Properties of the holomorphic function $F(s)$ translate into associated properties of $f(t)$ in the time domain; Figure \ref{fig:laplaceflowchart} sketches a typical flowchart of applying Laplace transform to answer a question.
\begin{figure}[h!]
    \centering
    \includegraphics[scale=0.3]{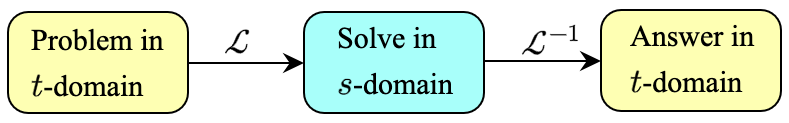}
    \caption{Flowchart of applying Laplace transform, where $\mathcal{L}$ represents Laplace transform and $\mathcal{L}^{-1}$ is inverse Laplace transform.}
    \label{fig:laplaceflowchart}
\end{figure}

Particularly, for linear ordinary differential equations as found for example in Electrical Engineering, Laplace transform helps simplify the problem by converting the differential equations to algebraic equations.  In this paper, the use of the Laplace transform derives similarly from its property to turn convolution of functions into ordinary multiplication of (holomorphic) functions:
\[
    \mathcal{L}\{f*g\}(s) = \mathcal{L}\{f\}(s)\cdot \mathcal{L}\{g\}(s).
\]
The solution to the diffusion equation on $\IR$ is given by a convolution integral equation in the time domain. Applying Laplace transform turns this into an algebraic equation of holomorphic functions in the $s-$domain.

Positivity of a function in the $t$-domain can be characterised by the property of complete monotonicity of its Laplace transform. A function $g(x)$ is called {\it completely monotonic} on the interval $(x_0,+\infty)$ in $\IR$ if
\begin{equation}\label{def:complete monotonicity}
    (-1)^n g^{(n)}(x)\geqslant 0\qquad\mbox{for all } x>x_0, n\in\IN.
\end{equation}
See also \cite{Arendt-Batty}, Section 2.7.
The necessity of this property is exhibited as follows
\begin{prop}\label{prop:poscompmono}
    Let $f\in L_{\mathrm{loc}}^1(\IR_0^+)$ with $\cL\{f\}$ exists for some $s_0\in\IC$. If $f\geqslant0$ almost everywhere, then $F:=\cL\{f\}$ is completely monotonic on $(\Re(s_0),+\infty)$.
\end{prop}
\begin{proof}
By the above, $F$ is analytic on $(\Re(s_0),+\infty)$. Therefore, for any $x>\Re(s_0)$,
\[
    \frac{\df^n F}{\df s^n}(x) =\left.\int_0^{+\infty}\frac{\partial^n}{\partial s^n}(e^{-st}f(t))\df t\,\right|_{s=x} = 
 (-1)^n\int_0^{+\infty}x^ne^{-xt}f(t)\df t.
\]
The integrand $x^ne^{-xt}f(t)$ is nonnegative a.e.
\end{proof}

The former result can be strenghtened to a full characterization of positivity of $f$, by means of the Post-Widder Inversion Theorem.

\begin{prop}[Characterization of positivity by Laplace transform]\label{prop:char positivity by Laplace}
    Let $f\in L^1_\loc(\IR_0^+)$ such that $\mathrm{abs}(f)<+\infty$. Then, $f\geqslant0$ if and only if $\cL\{f\}$ is completely monotonic on $(\mathrm{abs}(f),+\infty)$.
\end{prop}
\begin{proof}
    The necessity has been shown already in \propref{prop:poscompmono}. For sufficiency, note that almost every $t\geqslant 0$ is a Lebesgue point of $f$. Take such a point $t>0$. For sufficiently large $n\in\IN$, $n/t>\mathrm{abs}(f)$. The Post-Widder Inversion Formula (see \thmref{thm:postwidder}) yields
    \[
        f(t)=\lim _{n\apc+\infty }{\frac {(-1)^{n}}{n!}}\left({\frac {n}{t}}\right)^{n+1}F^{(n)}\left({\frac {n}{t}}\right)\geqslant 0,
    \]
    by complete monotonicity. Thus, $f\geqslant 0$ in $L^1_\loc(\IR_0^+)$.
\end{proof}

The following result is crucial for properly rephrasing later on our mathematical problem at hand. 
\begin{lem}[\citet{Arendt-Batty} Prop. 1.6.8, p.40]\label{lem:lapsqrtsub}
    Assume $f\in L_{\mathrm{loc}}^1(\IR_0^+)$ of exponential order. Let 
    \[
        g(t):=\int_0^{+\infty}\frac{\tau e^{-\frac{\tau^2}{4t}}}{\sqrt{4\pi t^3}}f(\tau)\df\tau.
    \]
    Then $g\in L_{\mathrm{loc}}^1(\IR_0^+)$ is well-defined, $G:=\cL\{g\}$ also exists, and $G(s)=F(\sqrt{s})$.
\end{lem}
\begin{proof}
By \eqref{eq:lapsqrtkern} in \lemref{lem:lapkerncomputation},
\begin{align*}
G(s) &= \int_0^{+\infty}\left(\int_0^{+\infty}\frac{\tau e^{-\frac{\tau^2}{4t}}}{\sqrt{4\pi t^3}}f(\tau)\df\tau\right)e^{-st}\df t\\
 &= \int_0^{+\infty}\left(\int_0^{+\infty}\frac{\tau e^{-\frac{\tau^2}{4t}}}{\sqrt{4\pi t^3}}e^{-st}\df t\right)f(\tau)\df\tau\\
&= \int_0^{+\infty} e^{-\sqrt s\tau}f(\tau)\df\tau 
= F(\sqrt s).
\end{align*}
Note that we apply Fubini-Tonelli theorem in order to swap the integrals.
\end{proof}
\begin{cor}\label{cor:lapsqrtsub}
    Given $f,\xi(\cdot,\tau)\in L_{\mathrm{loc}}^1(\IR_0^+)$  that are of exponential order and $\cL\{\xi(\cdot,\tau)\}(s)=e^{-\eta(s)\tau}$, then
$$\cL^{-1}\{F\circ \eta\}(t)=\int_0^{+\infty} \xi(t,\tau)f(\tau)\df\tau.$$
\end{cor}

\section{Positivity of solutions on part of the domain}\label{Sec:Pos}

For a nonlocal model like System \eqref{eq: system considered u} -- \eqref{def:Psi new} one cannot expect to find that solutions for positive initial conditions, with positive forcing $\Phi(t)$ to remain positive on the whole domain $\IR$, in general. See e.g., the counter-example at the end of \secref{sec:overview of results}. However, numerical simulations suggest that positivity can be ensured in many cases {\it on a subdomain of $\IR$} that contains ``the environment" of the original ``exchanging structure", i.e., $\IR\smallsetminus(-1,1)$. This will be satisfactory in applications, as the system is intended to approximate a model that is defined on this subdomain (only).

The mathematical techniques that we are aware of that are employed to prove positivity of solutions for PDEs with local reaction terms, like those e.g., in \cite{Sikic:1994,Canizo_ea:2012,Hille_ea:2025}, are not suited to prove positivity only on a subdomain, nor are able to deal with nonlocal reaction terms. Therefore, in this work we invented an approach suitable for the nonlocal reaction term in System \eqref{eq: system considered u} -- \eqref{def:Psi new} and applicable to subdomains.

\subsection{Pointwise positivity on the boundary suffices}
\label{sec:pointwise positivity}
Recall that we work with the non-dimensional equations, which -- improperly speaking -- effectively amounts to taking $R=1$ and $D=1$. Let $I$ be an open interval in $\IR$ whose closure does not contain $0$, the support of the Dirac measure in the reaction term. We approach the question of positivity of a solution of System \eqref{eq: system considered u} -- \eqref{def:Psi new} on the subdomain $I$ by means of the observation:
\begin{lem}\label{lem:central pos principle boundary}
    Let $u_0\geqslant 0$ on $I$. The mild solution $u$ to System \eqref{eq: system considered u} -- \eqref{def:Psi new} is positive on $I$ for all $t\geqslant 0$ if and only if $u(t,p)\geqslant 0$ for all boundary points $p$ of $I$ for all $t\geqslant 0$.
\end{lem}
\begin{proof}
    Note that the mild solution $u(t)\in H^1(\IR)$ and that by a Sobolev Embedding (see \cite{Adams:2003} Theorem 4.12, p.85) $u$ is continuous. Hence, point evaluation of $u(t)$ at $p\in \partial I$ is defined.
    Since $0\not\in I$, $u$ on $I$ is the solution to $\partial_t w = \Delta w$ with initial condition $w(0)=u(0)|_I\geqslant 0$ and Dirichlet boundary conditions $w(t,p)= u(t,p)$ at $p\in\partial I$ for all $t\geqslant 0$. According to the comparison principle (see Corollary 2.5 of \cite{Salsa2022partial}), if $u(t,p)\geqslant 0$ for all $t\geqslant 0$ and $p\in\partial I$, then $w(t)\geqslant 0$ for all $t\geqslant 0$. 
\end{proof}

Define
\begin{equation}\label{eq:def u pm}
    u_+(t):= u(t,1) + u(t,-1),\qquad u_-(t):= u(t,1) - u(t,-1).
\end{equation}
As an immediate consequence of Lemma \ref{lem:central pos principle boundary} one obtains that ``$u_+(t)\geqslant0$ for all $t\geqslant 0$" is a necessary condition for positivity of the solution $u$ on the subdomain $\IR\smallsetminus(-1,1)$. The following corollary formulates a slightly stronger result.
\begin{cor}
    The following are equivalent:
    \begin{enumerate}
        \item[({\it i})] $u(t,x)\geqslant 0$ for all $t\geqslant0$ and $x\in\IR\smallsetminus(-1,1)$.
        \item[({\it ii})] $u(t,-1)\geqslant0$ and $u(t,1)\geqslant 0$ for all $t\geqslant 0$.
        \item[({\it iii})] $u_+(t)\geqslant |u_-(t)|$ for all $t\geqslant 0$.
    \end{enumerate}
\end{cor}
\begin{proof}
``({\it i}\,) $\Leftrightarrow$ ({\it ii}\,)" is Lemma \ref{lem:central pos principle boundary}. For ``({\it ii}\,) $\Rightarrow$ ({\it iii}\,)", observe that $u(t,1)+2u(t,-1)\geqslant u(t,1)$, because $u(t,-1)\geqslant 0$. Then, by subtracting $u(t,-1)$ on both sides, one gets $u_+(t)\geqslant u_-(t)$. Similarly, one can start with $2u(t,1)+u(t,-1)\geqslant u(t,-1)$ and arrive at $u_+(t)\geqslant -u_-(t)$. The condition in ({\it iii}\,) implies that both $u(t,1)+u(t,-1)\geqslant u(t,1)-u(t,-1)$ and $u(t,1)+u(t,-1)\geqslant u(t,-1)-u(t,1)$ must hold. This gives ({\it ii}\,).
\end{proof}

We shall therefore start by finding conditions that ensure the positivity of $u_+$, which are necessary for $u$ to be positive on $\IR\smallsetminus (-1,1)$. Then we proceed to study the positivity on the boundary, ``pointwise" at $x=-1$ and $x=1$. This requires a careful analysis of the integral kernels in the expressions of $u(t,1)$ and $u(t,-1)$.

\subsection{Positivity of the sum: \texorpdfstring{$u_+$}{Lg}} \label{sec:positivity_u+}

Let $(e^{\Delta t})_{t\geqslant 0}$ be the diffusion semigroup in $H^1(\IR)$, generated by the Laplacian $\Delta$ on the domain $D(\Delta):= H^2(\IR)$. It is given explicitly in this setting as convolution with the Gauss-Weierstrass or heat kernel on $\IR$ given by
$$\cK_1(t,\beta):=\frac{e^{-\frac{\beta^2}{4t}}}{\sqrt{4\pi t}},\qquad(t>0, \beta\in\IR),$$
as
\[
    e^{\Delta t}f(x) := \int_\IR \cK_1(t,x-\tilde{x})f(\tilde{x})\df\tilde{x}.
\]
The Sobolev space $H^1(\IR)$ embeds continuously into $C_b^0(\IR)$, because of spatial dimension 1 (see Sobolev Embedding Theorem, e.g., \cite{Adams:2003}, Theorem 4.12, p.85). Thus, point evaluation at $x=-1$ and $x=1$ are defined on $H^1(\IR)$ and yield continuous functionals.
In view of the Variation of Constants Formula \eqref{eq:semigroupmild} and substituting 
\[
    \mathfrak F(u(\tau),\tau)=(\Phi(\tau)-a(u(\tau,1)+u(\tau,-1)))\delta_0,
\]
we obtain for $t>0$ and $x\in\IR$:
\begin{equation}
    u(t,x)=\int_0^t\cK_1(t-\tau,x)(\Phi(\tau)-au(\tau,1)-au(\tau,-1))\df\tau+\int_{\IR}\cK_1(t,x-\tilde{x})u_{0}(\tilde{x})\df \tilde{x}. \label{eq:integral expr u}
\end{equation}
From the latter equation, we can obtain a renewal equation for $u_+$.

\subsubsection{A Renewal Equation for $u_+$}\label{sec:DDE_u+}

Since the mild solution $t\mapsto u(t):\IR_0^+\to H^1(\IR)$ is continuous, by the observation made above that point evaluation is continuous, we obtain that $t\mapsto u_+(t)$ is a continuous function. From Equation \eqref{eq:integral expr u}, we derive that it solves the integral equation
\begin{equation}\label{eq:volterraeq}
u_+(t) = f_0(t)+\int_0^t h(t-\tau)u_+(\tau)\df s = v(t) + h*u_+(t), 
\end{equation}
where 
\begin{equation}
\label{eq_kernel_h}
    h(t-\tau) := -2a\cK_1(t-\tau,1),
\end{equation}
and
\begin{equation}\label{eq:gt}
f_0(t):=2\int_0^t\cK_1(t-\tau,1)\Phi(\tau)\df \tau+\int_{\IR}(\cK_1(t,1-\tilde{x})+\cK_1(t,1+\tilde{x}))u_{0}(\tilde{x})\df \tilde{x}.
\end{equation}

We used here the symmetry $\cK_1(t,-x)=\cK_1(t,x)$. Note that $v\geqslant0$ under our assumptions (i.e., $\Phi(t)\geqslant 0$ and $u_0\geqslant 0$). It is a \textit{Volterra integral equation of the second kind}, for which a comprehensive solution theory exists; see \cite{Volterra-Naito}. Our case is a more particular instance of the \textit{convolution type}, also known as a (linear) \textit{renewal equation}. The kernel $h$ defined in \eqref{eq_kernel_h} is in $L^1_{\loc}(\IR_0^+)$ as well as $\Phi\in L^1_{\loc}(\IR_0^+)$. Hence, its $n$-fold convolution product, $h^{(*n)}:=h*\dots*h$, is defined. Consequently, we can solve $u_+$ via
\begin{equation}\label{eq:resolvent expr uplus}
    u_+ = f_0+ \mathcal{R}*f_0, \qquad{ \text{with }} \mathcal{R} := \sum_{n=1}^\infty h^{(*n)},
\end{equation}
if a suitable interpretation can be given to the series defining the so-called {\it resolvent kernel} $\mathcal{R}$ of the renewal equation (see e.g., \cite{Diekmann_ea:1995}, Section I.2).

First of all, we need the following results for later use:

\begin{lem}\label{lem:v exp order}
$f_0$ is of exponential order.
\end{lem}
\begin{proof}For all $t>0$, the first integral in \eqref{eq:gt} is bounded by $C_1e^{\gamma_1 t}$ for some $C_1,\gamma_1$, because $\cK_1(t,1)$ is bounded as a function of $t$ and $\Phi\in L_{\mathrm{loc}}^\infty(\IR_0^+)$ of some exponential order $\gamma_1$. Therefore, we just need to show that the second integral satisfies the same property for all $t>0$. In fact, we can show that the integral of each term is bounded. Because there is an embedding $u_0\in H^1(\IR)\subseteq C_b^0(\IR)\subseteq L^{\infty}(\IR)$ (see the discussion at the beginning of \secref{sec:mathass}), we can apply H\"older's inequality,
$$\left|\int_{\IR}\cK_1(t,1\pm\tilde{x})u_0\df\tilde{x}\right|\leqslant\Vert u_0\Vert_{L^{\infty}}\Vert\cK_1(t,\cdot)\Vert_{L^1}.$$
Due to Gauss integral, it is well known that $\Vert\cK_1(t,\cdot)\Vert_{L^1}=1$ independent of $t$. Thus, $f_0$ is of exponential order, and hence belongs to $L^\infty_{\loc}(\IR_0^+)$.
\end{proof}

\begin{lem}
\label{lem:cR exp order}
$\cR$ is of exponential order.
\end{lem}
\begin{proof}Inductively, we have
$$|h^{(*n)}(t)|\leqslant\frac{\Vert h\Vert_{\infty}^nt^{n-1}}{(n-1)!}.$$
Then
$$\label{eq:resolvestimation}
|\cR(t)|\leqslant\sum_{n=1}^{\infty}\frac{\Vert h\Vert_{\infty}^nt^{n-1}}{(n-1)!}=\Vert h\Vert_{\infty}{e^{\Vert h\Vert_\infty t}}.
$$
In fact, $\cR$ is even a smooth integral kernel and also in $L^\infty_{\loc}(\IR_0^+)$.
\end{proof}

\begin{cor}\label{cor:uplusgrowthrate}
$u_+$ is of exponential order.
\end{cor}
\begin{proof}
As we explained above, Equation \eqref{eq:volterraeq} can be solved by \eqref{eq:resolvent expr uplus}. Then it follows from \lemref{lem:v exp order} and \lemref{lem:cR exp order}. In particular, $u_+\in L^\infty_{\loc}(\IR_0^+)$ as well.
\end{proof}

If $a\leqslant 0$, then $h\geqslant 0$ and consequently, $u_+\geqslant 0$, as can be immediately seen from \eqref{eq:resolvent expr uplus}. However, the negative feedback ($a>0$) really complicates the analysis of positivity in the system, because now the series defining $\mathcal{R}$ is alternating. 

\subsubsection{An expression for $u_+$ in the $s$-domain}\label{sec:expression of u+ s-domain}

Corollary \ref{cor:uplusgrowthrate} guarantees that $u_+$ is of exponential order. Thus, applying the Laplace transform to study integral equation  \eqref{eq:volterraeq} is now a most natural step to take. Necessary fundamental definitions and properties concerning Laplace Transform can be found in Section \ref{Sec:Laplace} and \appref{app:A}. So, let 
\[
U_+(s):=\cL\{u_+\}(s),\quad 
\Theta(s):=\cL\{\Phi\}(s),\quad
F_0:=\cL\{f_0\}(s),\quad \mbox{and}\quad Q(s,\beta):=\cL\{\cK_1(\cdot,\beta)\}(s),
\]
where all Laplace transforms exist because the functions that are transformed are of exponential order (see \corref{cor:uplusgrowthrate}, \lemref{lem:v exp order}, and \eqref{eq:lapheatkernone}). The convolution equation \eqref{eq:volterraeq} then becomes an algebraic relation for the Laplace transforms:
\begin{equation}\label{eq:relation Laplace u-plus}
U_+(s)=F_0(s)-2aQ(s,1)U_+(s),
\end{equation}
where 
\begin{equation}
\label{eq:Q}
Q(s,\beta) = \frac{e^{-\beta\sqrt{s}}}{2\sqrt{s}}
\end{equation}

(computed in Appendix \ref{app:C}, in particular see \eqref{eq:lapheatkernone}) and
\begin{equation}\label{eq:def G}
F_0(s)=2Q(s,1)\Theta(s)+\int_{\IR}(Q(s,|1-\tilde{x}|)+Q(s,|1+\tilde{x}|))u_0(\tilde{x})\df \tilde{x}.
\end{equation}
Comparison of \eqref{eq:relation Laplace u-plus} and \eqref{eq:def G} indicates that defining
\begin{equation}\label{eq:lappspell}
    P_a(s,\beta) := \frac{Q(s,\beta)}{1+2aQ(s,1)}
\end{equation}
yields a convenient expression for $U_+(s)$:
\begin{equation}\label{eq:lapuplusinpa}
U_+(s)=2P_a(s,1)\Theta(s)+\int_{\IR}(P_a(s,|1-\tilde{x}|)+P_a(s,|1+\tilde{x}|))u_0(\tilde{x})\df \tilde{x}.
\end{equation}

\subsubsection{Outline of the approach to proving positivity of $u_+$}
\label{sec:outlineposuplus}

If $p_a(t,\beta):=\cL^{-1}\{P_a(.\cdot\beta)\}(t)$ exists, apply inverse Laplace transform formally to \eqref{eq:lapuplusinpa}. The desired solution in the $t$-domain becomes
\begin{equation}\label{eq:uplusinpa} 
u_+(t)=\underbrace{\int_0^t2p_a(\tau,1)\Phi(t-\tau)\df\tau}_{u_+^{\mathrm{tp}}}+\underbrace{\int_{\IR}(p_a(t,|1-\tilde{x}|)+p_a(t,|1+\tilde{x}|))u_0(\tilde{x})\df \tilde{x}}_{u_+^{\mathrm{sp}}}.
\end{equation}

Thus, the question about positivity of $u_+$ can be related to properties of its holomorphic Laplace transform $U_+$, which essentially boils down to similar properties of the functions $P_a(s,\beta)$, through Expression \eqref{eq:lapuplusinpa}. We observe the sufficient condition:
\begin{quote}
    {\it If the inverse Laplace transform of $P_a(s,\beta)$ exists and is nonnegative for all $\beta\geqslant 0$, then the inverse Laplace transform of $U_+$ exists, equals $u_+$ and is nonnegative.}
\end{quote} 
Analysing positivity through the Laplace transform is not straightforward, though. \propref{prop:char positivity by Laplace} provides an entrance, but it is difficult to have control over the signs of all derivatives, required for complete monotonicity. Therefore, we take a different approach.

For $P_a(s,\beta)$ we have an explicit expression, although complicated. In fact, 
\begin{equation}
P_a(s,\beta) = \frac{\frac{e^{-\beta\sqrt{s}}}{2\sqrt{s}}}{1+2a\frac{e^{-\beta\sqrt{s}}}{2\sqrt{s}}}
 = \frac{e^{-\beta\sqrt{s}}}{2\sqrt{s}+2ae^{-\sqrt{s}}}.
\end{equation}
A crucial observation, formulated in \lemref{lem:lapsqrtsub} or \corref{cor:lapsqrtsub}, allows us to change from the complicated form of $P_a(s,\beta)$, to the better analyzable {\it meromorphic} function 
\begin{equation}\label{eq:tilde P}
    \tilde P_a(s,\beta) := \frac{e^{-\beta s}}{2s + 2ae^{-s}}.
\end{equation}
Accordingly (cf. \eqref{eq:Q}), we also define 
\begin{equation}\label{eq:tilde Q}
    \tilde Q(s,\beta) := \frac{e^{-\beta s}}{2s}.
\end{equation}
It allows to conclude that 
\begin{quote}
    {\it If the inverse Laplace transform of $\tilde P_a(s,\beta)$ exists and is positive, then so is that of $P_a(s,\beta)$,}
\end{quote} 
according to \lemref{lem:lapsqrtsub}. Thus, the question is now reduced to studying $\tilde P_a(a,\beta)$. In summary of the discussion above, we formulate
\begin{prop}\label{prop:patoupluspos}
    Assume that $0\leqslant u_0\in H^1(\IR)$ and $0\leqslant\Phi\in L^\infty_\loc(\IR^+_0)$. Fix $a\in\IR$ and assume that the inverse Laplace transforms $p_a(t,\beta):=\mathcal{L}^{-1}\{P_a(s,\beta)\}$ and $\tilde{p}_a(t,\beta):=\mathcal{L}^{-1}\{\tilde{P}_a(s,\beta)\}$ exist in $L^1_\loc(\IR_0^+)$ for all $\beta\geqslant 0$. Consider the following statements:
    \begin{enumerate}
        \item[({\it i})]  $u_+(t)\geqslant 0$ for all $t\geqslant 0$;
        \item[({\it ii})]  $p_a(t,\beta)\geqslant0$ for all $t\geqslant 0$ and $\beta\geqslant 0$;
       \item[({\it iii})]  $\tilde p_a(t,\beta)\geqslant0$ for all $t\geqslant 0$ and $\beta\geqslant 0$.
       \item[({\it iv})]  $\tilde p_a(t,0)\geqslant0$ for all $t\geqslant 0$
    \end{enumerate}

    Then ({\it iv}) $\Leftrightarrow$ ({\it iii}) $\Rightarrow$ ({\it ii}) $\Rightarrow$ ({\it i}).
\end{prop}
\begin{proof} All implications have been discussed above, or are obvious, except for ``({\it iv\,}) $\Rightarrow$ ({\it iii\,})". Consider Expression \eqref{eq:tilde P}. Multiplication in the $s$-domain by $e^{-s\beta}$ simply shifts the function by $\beta$ to the right, due to Table \cosref{tab:laplace}{A.1}. Hence, 
\begin{equation}\label{eq:tilde p beta as p 0}
    \tilde{p}_a(t,\beta)=\begin{cases} \tilde{p}_a(t-\beta,0), & \text{for } t\geqslant\beta\\
    0, & \mbox{otherwise}.
    \end{cases}
\end{equation}
Thus, the positivity of $\tilde{p}_a(t,\beta)$ is not affected by a change in $\beta$.
\end{proof}

Now, a second advantage of Laplace transform comes into play: it allows us to change perspectives between different but equivalent problems. We view $\tilde{P}_a(s,\beta)$ as the transfer function of a feedback system, by Laplace transform. In fact, we will relate \eqref{eq:tilde P} to a delayed differential equation (DDE) in Section \ref{subsec:relate to DDE}. Analyzing the resulting DDE using literature (and -- independently -- in \appref{app:B}), in turn, finally yields the following nontrivial positivity result for $\tilde{p}_a(t,0)$:
\begin{prop}\label{prop:positivity tilde pa}
If $a\leqslant e^{-1}$, then $\tilde{p}_a(t,\beta)\geqslant0$ and $p_a(t,\beta)\geqslant0$ for all $t,\beta\geqslant 0$.
\end{prop}
This will be proven in \secref{subsec:relate to DDE}, as a consequence of -- ultimately -- \corref{clry:suff pos condition specific}. Finally we obtain the similarly nontrivial positivity result for $u_+$ as an immediate consequence: 
\begin{thm}\label{thm:upluspos}
Assume that $0\leqslant u_0\in H^1(\IR)$ and $0\leqslant\Phi\in L^\infty_\loc(\IR^+_0)$, of at most exponential order. If $a\in[0,e^{-1}]$, then the function $u_+$ is nonnegative.
\end{thm}
\begin{proof}
By \propref{prop:patoupluspos}, we just need to guarantee $\tilde p_a(t,0)$ is nonnegative for all $\beta\geqslant0$. This is provided by \propref{prop:positivity tilde pa}.   
\end{proof}

The condition of at most exponential order has been included to assure that the Laplace transform of $\Phi$ exists. This is necessary for th eapplication of our techniques. We shall assume this condition from this point onwards.

Anticipating the results that follow in subsequent sections that establish the existence of the inverse Laplace transforms $p_a(t,\beta)$ and $\tilde{p}_a(t,\beta)$ in $L^1_\loc(\IR_0^+)$, in view of \propref{prop:patoupluspos} we need to show that $p_a(t,\beta)\geqslant0$ or $\tilde p_a(t,\beta)\geqslant0$. According to \propref{prop:poscompmono} we expect $P_a(s,\beta)$ to be completely monotonic (recall \eqref{def:complete monotonicity}). This yields insight in the shape of the region in the $a\beta$-plane where $p_a(t,\beta)$ cannot be nonnegative for all $t$. This provides useful information beyond that provided in \thmref{thm:upluspos}, derived from the study of the DDE.

The easiest case is the first derivative
\begin{equation}
\label{Eq_P_a_derivative}
P_a'(s,\beta)=
-\frac{e^{(1- \beta) \sqrt{s}}}{4  \sqrt{s}\left( a + e^{\sqrt{s}} \sqrt{s} \right)^2}\left( a\beta-a + e^{\sqrt{s}} + \beta\sqrt{s}e^{\sqrt{s}} \right).
\end{equation}
$P_a'(s,\beta)$ is a strictly increasing function on $[0,+\infty)$ under the condition that $\beta\geqslant0$. Thus, $P_a'(s,\beta)$ has a (simple) positive real root exactly when $P_a'(0,\beta)=a\beta-a+1<0$, which defines the \textit{critical curve}. In this region, indicated in Figure \ref{fig:possible-reject} in red, we therefore reject that $p_a(t,\beta)$ is nonnegative.

In theory, we can also determine the region on the $a\beta$-plane where $P''_a(s,\beta)$ has a positive root, and analogously for higher derivatives. However, identifying these regions for second and higher derivatives is extremely involved. Instead, the analysis of the associated DDE will identify a region in $a$-space that certainly works: $a\leqslant e^{-1}\approx 0.3679$. Note that there seems to be a gap between $e^{-1}$ and 1, the smallest value of $a$ in the rejected region. It is not clear whether there exist admissible values of $a$ in this gap. We will investigate this further numerically, in Section \ref{Sec_numerical_results}.

\begin{figure}
    \centering
    \includegraphics[scale=0.45]{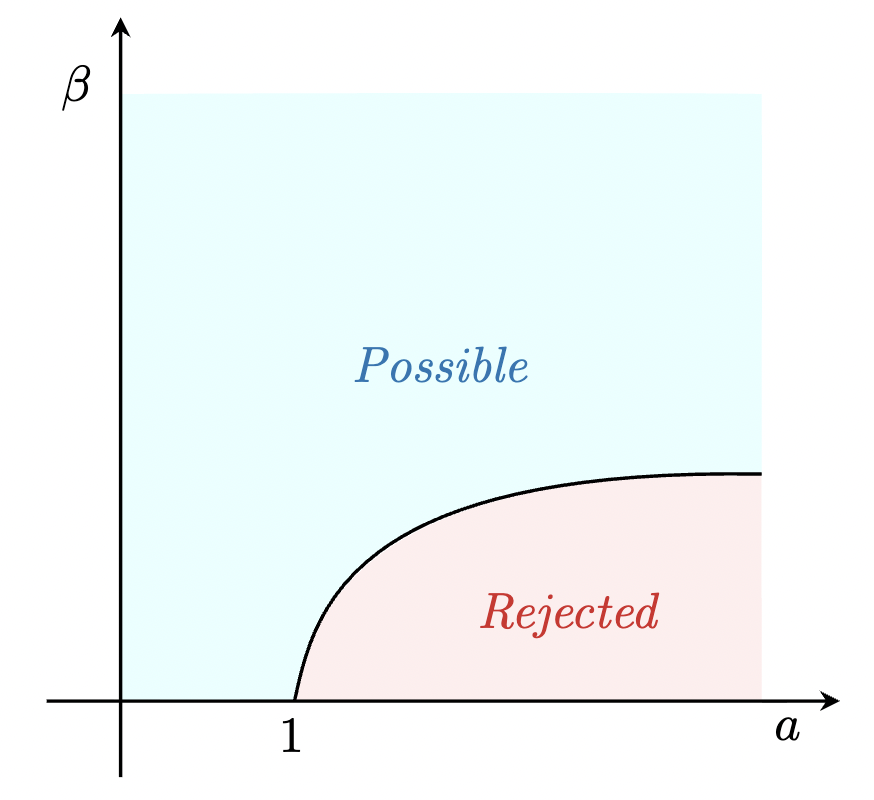}
    \caption{Regions of $a$ and $\beta$ regarding the positivity of $p_a(t,\beta)$ determined by the existence of the positive real root in Equation \eqref{Eq_P_a_derivative} is shown. Thus, the positivity of $u_+$ is rejected or possible accordingly. The regions are separated by the critical curve $a\beta - a + 1 =0$.}
    \label{fig:possible-reject}
\end{figure}

Thus, we shall now focus on establishing properties of the transfer functions $P_a(s,\beta)$ and their simplified forms $\tilde{P}_a(s,\beta)$, in particular the existence and properties of their inverse Laplace transforms $p_a(t,\beta)$ and $\tilde{p}_a(t,\beta)$.

\subsubsection{ Analysis of $\tilde{p}_a(s,\beta)$ by related delay equation}
\label{subsec:relate to DDE}

Rewrite \eqref{eq:tilde P} for $\beta=0$ as
\begin{equation}\label{eq:padde}
s\tilde P_a(s,0)-\frac{1}{2}=-ae^{-s}\tilde P_a(s,0).
\end{equation}
Applying \cosref{tab:laplace}{A.1} to \eqref{eq:padde}, the inverse transform $\tilde p_a(t,0)$ -- assuming for the moment that it exists in $L^1_\loc(\IR_0^+)$ -- equivalently satisfies the following delayed differential equation (cf. \appref{app:B}),
\begin{equation}\label{eq:ddeonlyb}
\begin{cases}
    \tilde{p}_a'(t,0)=-a\tilde{p}_a(t-1,0), &  t\in(1,+\infty)\\
    \tilde{p}_a(t,0) = \frac{1}{2}, & t\in[0,1]
\end{cases}.
\end{equation}

So, we are faced with the question of determining all $A\in\IR$ for which the solution $y(t)$ to the delay differential equation with specific history,
\begin{equation}
\label{eq:dimensionless scalar dde}
        \begin{cases}
            y'(t) = Ay(t-1), \quad &t\in(1, +\infty),\\
            y(t) = 1, \quad & t\in[0,1].
        \end{cases}
\end{equation}
is nonnegative for all time $t$. This can be answered partially by invoking \cite{gyori1991oscillation} Theorem 3.3.1, p.70, which is in the field a well-known result on positivity of solutions of a much broader class of linear delay differential equations (DDEs). We shall quote this result for the convenience of the reader, because \cite{gyori1991oscillation} seems not easily accessible. 

The theorem considers non-autonomous scalar DDEs with variable coefficients and delays, of the form
\begin{equation}\label{eq:DDE with variable delay and coefficients}
    x'(t) + \sum_{i=1}^n p_i(t)x(t-\tau_i(t))=0,\qquad \mbox{for } t\in[t_0,T),
\end{equation}
with $t_0\leqslant T\leqslant \infty$ and continuous coefficient functions $p_i:[t_0,T)\apc\IR$ and delay functions $\tau_i:[t_0,T)\apc\IR_+$. Put
\[
    t_{-1} := \min_{1\leqslant i\leqslant n} \bigl[ \inf_{t_0\leqslant t < T} t-\tau_i(t)\bigr],\qquad g(t) :=  \min_{1\leqslant i\leqslant n} \bigl[ t_0\vee (t-\tau_i(t)) ]\quad\mbox{for } t\in[t_0,T)
\]
and define
\[
    \Phi^+ := \bigl\{ \phi\in C^0([t_{-1}, t_0],\IR)\mid \phi(t_0)>0, \phi(t)\geqslant\phi(t_0) \mbox{ for all } t\in[t{-1}, t_0] \bigr\}.
\]
Then, one has
\begin{thm}[Gy\"ori \& Ladas, \cite{gyori1991oscillation}, Theorem 3.3.1, p.70]
    Assume that $p_i$ and $\tau_i$ are continuous and that
    \begin{equation}
        \sum_{i=1}^n \int_{g(t)}^t p_i(s)^+ \df s \leqslant \frac{1}{e}\qquad\mbox{for all } t\in[t_0,T),
    \end{equation}
    then for any $\phi\in\Phi^+$ the solution to \eqref{eq:DDE with variable delay and coefficients} satisfying the history condition $x(t)=\phi(t)$ for $t\in[t_{-1},t_0]$ is positive: $x(t)>0$ for $t\in[t_0,T)$.
\end{thm}

This implies the following sufficient condition for positivity of the solution $y(t)$ to \eqref{eq:dimensionless scalar dde}:
\begin{cor}\label{clry:suff pos condition specific}
    Let $A\geqslant -e^{-1}$. Then $y(t)>0$ for all $t\geqslant 0$.
\end{cor}
We can now obtain results on $\tilde{p}_a(t,\beta)$ from the DDE analysis.
\begin{prop}\label{prop:tildepexpdecay}
    $y(t)$ is of exponential order. One has $\tilde{p}_a(t,0)=\frac{1}{2}y(t)$. Therefore, $\tilde{p}_a(t,\beta)$ is of exponential order, in particular 
    \begin{equation}
        \bigl| \tilde{p}_a(t,\beta)\bigr|\leqslant \frac{1}{2}\theta(t-\beta) e^{|a|(t-\beta)}\quad \mbox{for all } t\geqslant 0.
    \end{equation}
    Moreover, one has the series expression (which is a finite sum for each $t\geqslant0$):
    \begin{equation}\label{eq:lapinvpsub}
        \tilde p_a(t,\beta)=\frac{1}{2}\sum_{m=0}^{\infty}\frac{(-a)^m(t-m-\beta)^m}{m!}\theta(t-m-\beta).
\end{equation}
\end{prop}
\begin{proof}
    According to \lemref{lem:bound y}, $y(t)$ is continuous and exponentially bounded. Thus, it has a Laplace transform, $Y(s)$. Taking Laplace transform of the DDE \eqref{eq:dimensionless scalar dde} with $A=-a$, then $X(s):=\frac{1}{2}Y(s)$ satisfies 
    \[
        sX(s) - \frac{1}{2} = -ae^{-s}X(s).
    \]
    Hence, $X(s)=\tilde{P}_a(s,0)$. By the Uniqueness Theorem, \thmref{thm:uniquness Laplace}, $\frac{1}{2}y(t)$ is the unique function in $L^1_\loc(\IR_0^+)$ for which the Laplace transform exists and equals $\tilde{P}_a(s,0)$. Thus $\tilde{p}_a(t,0)=\frac{1}{2}y(t)$. One has
    \[
        |y(t)|\leqslant e^{|a|t}\quad \mbox{for all } t\geqslant 0,
    \]
    according to \lemref{lem:bound y}. Relation \eqref{eq:tilde p beta as p 0} between $\tilde{p}_a(t,\beta)$ and $\tilde{p}_a(t,0)$ now yields the indicated estimate for $\tilde{p}_a(t,\beta)$.

    By integration by steps (cf. \eqref{eq:etaA}) it can be readily verified that 
    \[
        y(t) = \sum_{m=0}^\infty \frac{A^m (t-m)^m}{m!} \theta(t-m).
    \]
    Again applying Relation \eqref{eq:tilde p beta as p 0} yields expression \eqref{eq:lapinvpsub}.
\end{proof}

\begin{proof}[Proof of \propref{prop:positivity tilde pa}]
Thanks to \propref{prop:patoupluspos}, it suffices to prove $\tilde p_a(t,0)\geqslant0$. From the analysis made above, we got that $\tilde{p}_a(t,0)=\frac{1}{2}y(t)$, with $y(t)$ the solution to \eqref{eq:dimensionless scalar dde} with $A=-a$. Hence, if $a\leqslant e^{-1}$, $\tilde{p}_a(t,0)>0$ for all $t\geqslant 0$.
\end{proof}

As this result is of importance in application of our results in this paper, we shall provide in Appendix ref{app:B} a proof of Corollary \ref{clry:suff pos condition specific} based on Laplace transform that is independent of \cite{gyori1991oscillation}, Theorem 3.3.1. However, this approach can only be applied to the situation of constant coefficients and single constant delay. It was the original line of argumentation before we became aware of \cite{gyori1991oscillation}.

The condition $A\geqslant -e^{-1}$ is sufficient for strict positivity of $y(t)$. Whether positivity fails for $A<-e^{-1}$ close to $-e^{-1}$ is not clear. Moreover, various research has considered the question whether all solutions are oscillatory or that there exist positive solutions (see e.g., \cite{gyori1991oscillation,DIBLIK1998200}), in the generality of variable coefficients. For our investigation of positivity of solutions to System \eqref{eq: system considered u} - \eqref{def:Psi new} we want the specific solution to \eqref{eq:dimensionless scalar dde} to be positive.

\subsubsection{Further analysis of $P_a(s,\beta)$ and $\tilde{P}_a(s,\beta$) and their inverse Laplace transforms}
\label{sec:furtheranapatildepa}
One must always justify the existence of the inverse Laplace transform of a holomorphic function, especially when the function is not obtained as the Laplace transform of some given function in $L^1_{\mathrm{loc}}(\IR_0^+)$. So, in this case we need to prove that each term in \eqref{eq:lapuplusinpa} has inverse Laplace transform. It is apparent that $P_a(s,\beta)$ is the most crucial factor in the expression. 

The inverse transform $\tilde{p}_a(t,\beta):=\cL^{-1}\{\tilde{P}_a(\cdot,\beta)\}$ exists in $L^1_\loc(\IR_0^+)$, according to \propref{prop:tildepexpdecay}. Then 
\begin{equation}\label{eq:papatilderelation}
    p_a(t,\beta)=\int_0^{+\infty}\frac{\tau e^{-\frac{\tau^2}{4t}}}{\sqrt{4\pi t^3}}\tilde p_a(\tau,\beta)\df\tau
\end{equation} by \lemref{lem:lapsqrtsub} or \corref{cor:lapsqrtsub}. 
Taking this form of $p_a(t,\beta)$, we can justify the existence of the inverse Laplace transform of each term in $U_+$.

\begin{prop}\label{prop:paexists}
$p_a(t,\beta)=\cL^{-1}\{P_a(\cdot,\beta)\}(t)$ exists and is of exponential order.
\end{prop}
\begin{proof}
When $\beta>0$, take a big $\omega\in\IR$. $P_a$ is holomorphic on $\{s\in\IC\mid \Re(s)>\omega\}$, because $\sqrt{s}$ will dominate in the denominator $\sqrt{s}+ae^{-\sqrt s}$. Along each vertical line eventually $e^{-\sqrt{s}\beta}$ will dominate. Therefore, it satisfies the conditions in \thmref{thm:lapinvexist}.

When $\beta=0$, \propref{prop:tildepexpdecay} shows that $\tilde p_a(t,0)$ is of exponential order: in fact, $|\tilde p_a(t,0)|\leqslant\frac{1}{2}e^{at}$. Then by \eqref{eq:papatilderelation} 
$$
p_a(t,0)=\int_0^{+\infty}\frac{\tau e^{-\frac{\tau^2}{4t}}}{\sqrt{4\pi t^3}}\tilde p_a(\tau,0)\df\tau.
$$
We just need to prove this function is well-defined in $L^1_{\mathrm{loc}}(\IR_0^+)$ and of exponential order. By \eqref{eq:paexpbound}, 
\begin{equation}
|p_a(t,0)|\leqslant\frac{5e^{20a^2t}}{8\sqrt{\pi t}},
\end{equation}
which certifies that $p_a(t,0)\in L^1_{\mathrm{loc}}(\IR_0^+)$ and is of exponential order.
\end{proof}

\begin{prop}\label{prop:uplustermexist}
Each term in \eqref{eq:uplusinpa} is well-defined and possesses a Laplace transform.
\end{prop}
\begin{proof}
By \propref{prop:paexists} and $\Phi\in L^{\infty}_\loc(\IR_0^+)$ is of exponential order, we can see that $u_+^{\mathrm{tp}}$ in \eqref{eq:uplusinpa} exists and is of exponential order. By \corref{cor:uplusgrowthrate} $u_+$ is also of exponential order. Thus, $u_+^{\mathrm{sp}}$ in \eqref{eq:uplusinpa} also has the same properties. However, we prove a stronger statement that 
$$q_a(t):=\int_{\IR}p_a(t,|1\pm \tilde{x}|)u_0(\tilde{x})\df \tilde{x}$$
is in $L^1_{\mathrm{loc}}(\IR_0^+)$ and is of exponential order as we will analyze each term later. From \eqref{eq:lapinvpsub} note that $\tilde p_a(t,\beta)$ is $\tilde p_a(t,0)$ shifted by $\beta$ to the right. Then $|\tilde p_a(t,\beta)|\leqslant\frac{1}{2}\theta(t-\beta)e^{at}$. Again by \eqref{eq:paexpbound}, 
\begin{equation}\label{eq:exp bound pa}
    |p_a(t,\beta)|\leqslant\frac{5e^{-\frac{\beta^2}{5t}+20a^2t}}{8\sqrt{\pi t}}.
\end{equation}
Then use H\"older's inequality,
\begin{align*}
    |q_a(t)|&\leqslant\left(\int_{\IR}|p_a(t,|1\pm \tilde{x}|)|^2\df \tilde{x}\cdot\int_{\IR}|u_0(\tilde{x})|^2\df \tilde{x}\right)^{\frac{1}{2}}\\
    &\leqslant\Vert u_0\Vert_{L^2}\cdot\frac{5\sqrt[4]{5}e^{20a^2t}}{8\sqrt[4]{\pi t}}.
\end{align*}
As $u_0\in H^1(\IR)$, $q_a\in L^1_{\mathrm{loc}}(\IR_0^+)$ is well-defined and of exponential order.
\end{proof}

The behaviour of the solution is determined by $p_a(t,\beta)$, which is then tightly related to the simpler function $\tilde p_a(t,\beta)$ through \eqref{eq:papatilderelation}. $\tilde p_a(t,\beta)$ has an exact formula given by \eqref{eq:lapinvpsub}. Moreover, $\tilde{p}_a(t,\beta)$ is related to $\tilde{p}_a(t,0)$ through \eqref{eq:tilde p beta as p 0}.  For different $a$, the behaviour of $\tilde p_a(t,0)$ varies significantly, which has been illustrated in Fig. \ref{fig:tildepa} To provide some insight into this dependence, we begin by plotting \eqref{eq:lapinvpsub} for various choices of $a$ in Fig. \ref{fig:tildepapos}.

When $a=0$, $\tilde p_a(t,0)=\frac{1}{2}$ is a constant function; when $a<0$, it seems that $\tilde p_a(t,0)$ always increases exponentially. This phenomenon is presented in Fig. \ref{fig:tildepapos}, left panel, for $a=0$ and $-1$.
When $a>0$, it becomes more complicated. For $a$ close to $0$, $\tilde p_a(t,0)$ is exponentially decaying. For $a$ a bit further from $0$, $\tilde p_a(t,0)$ behaves like a damped oscillation. For $a$ much further from $0$, $\tilde p_a(t,0)$ behaves like unstable oscillation. This can be seen in Fig. \ref{fig:tildepapos}, right panel.

\begin{rem}{\it
The behaviour of $\tilde p_a(t,0)$ is determined by the poles of $\tilde{P}_a(s,0)$, which are the solutions to the transcendental equation 
\begin{equation}
    s+ae^{-s}=0 \;\;\Leftrightarrow\;\; se^s=-a.
\end{equation}
The solutions of this equation are characterized by the Lambert $W$-function (see \cite{CGH96}), $W_k(-a)$. For example, when $a = 1$, the principal poles of $\tilde P_a(s,0)$ are not purely real: $W_0(-1) \approx -0.318 + 1.337i$, $W_{-1}(-1) \approx -0.318 - 1.337i$. Ignoring all other poles, $\tilde p_a(t,0)$ can be approximated as
$$\tilde p_a(t,0)\approx C_1 e^{-0.318t} \cos(1.337t + \varphi_1),$$
for some $C_1$ and $\varphi_1$. This explains the damped oscillatory behavior of the green curve in Fig. \ref{fig:tildepa}, right panel. Similarly, when $a=2$, the principal poles of $P_a(s,0)$ are $W_0(-2) \approx 0.173 + 1.674i$, $W_{-1}(-2) \approx 0.173 - 1.674i$. This leads to the approximation
$$\tilde p_a(t,0)\approx C_2e^{0.173t} \cos(1.674t + \varphi_2),$$
which explains the unstable oscillatory behavior of the blue curve in Fig. \ref{fig:tildepa}, right panel. Finally, when $a = 0.25$, the only principal pole of $\tilde P_a(t,0)$ is $W_0(-0.25) \approx -0.357$. Thus, the solution simplifies to
$$\tilde p_a(t,0) \approx C_3 e^{-0.357t},$$
which explains the exponential decay behavior of the red curve in Fig. \ref{fig:tildepa}, right panel. We are particularly interested in the last case as the red curve, because this is the case when $a>0$ and $\tilde p_a(t,0)$ is positive and decaying.}

\begin{figure}[H]
    \centering
    \begin{tabular}{cc}
        \includegraphics[scale=0.5]{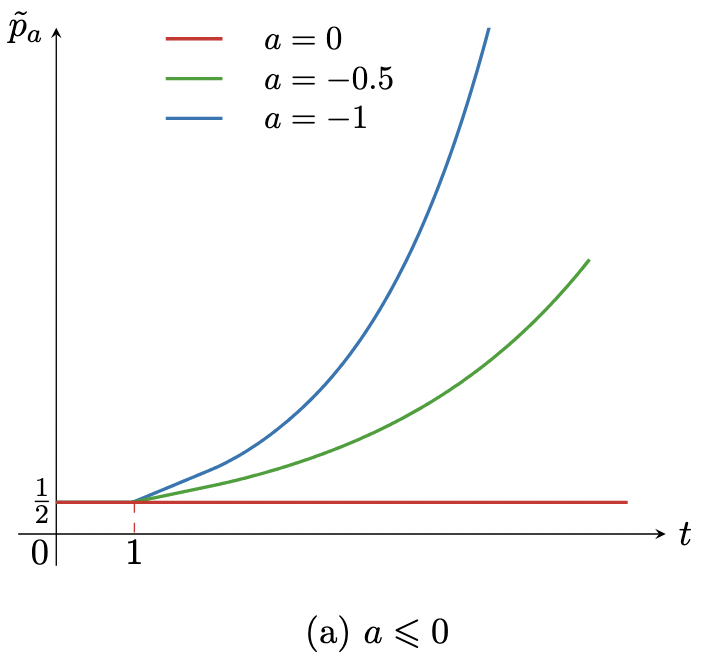}\qquad\includegraphics[scale=0.5]{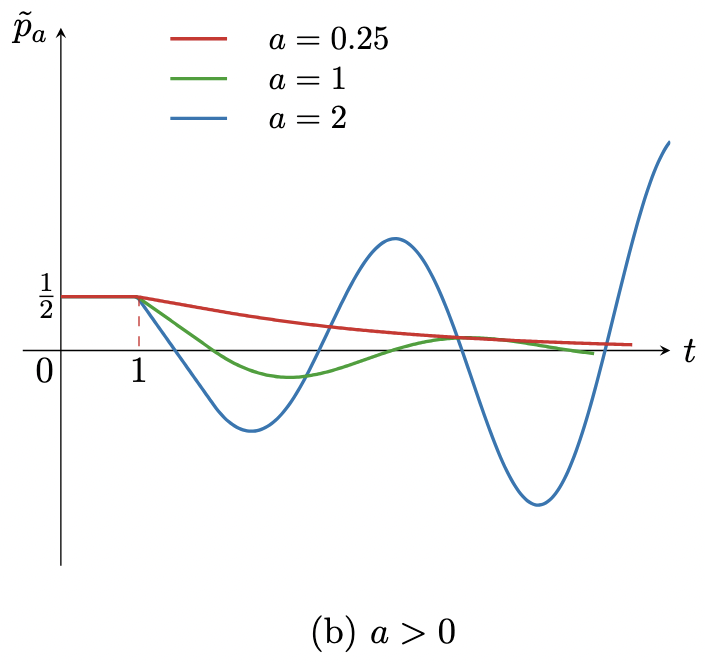}
    \end{tabular}
    \caption{ Change of shape of $\tilde p_a(t,0)$ with varying $a$. The function has been plotted using the exact expression \eqref{eq:lapinvpsub}. Left: when $a\leqslant0$, particularly, $a\in\{0,0.5, -1\}$. Right: when $a>0$, in particular $a \in \{0.25, 1, 2\}$. Note that for $a\in\{1,2\}$, $\tilde{p}_a$ is oscillatory and can become negative. }
    \label{fig:tildepa}
\end{figure}
\end{rem}

\begin{prop}\label{prop:paw11}
The following regularity properties hold:
\begin{enumerate}
    \item[(i)] $p_a(t,\beta)\in L^1(\IR_0^+)$ for all $\beta\geqslant 0$;
    \item[({ii})] $p_a(t,\beta)\in W^{1,1}(\IR_0^+)$ for all $\beta>0$.
\end{enumerate}
\end{prop}
\begin{proof}
({\it i}\,): According to \propref{prop:paexists}, $p_a(t,\beta)$ is exponentially decaying. One derives from \eqref{eq:exp bound pa} that
\begin{equation}\label{eq:paboundinf}
    |p_a(t,\beta)|\leqslant\frac{Ce^{-\frac{\beta^2}{4t}}}{\sqrt{\pi t}}.
\end{equation}
whose integral is bounded on $(0,1)$ as $e^{-\frac{\beta^2}{4t}}\leqslant1$ for $t>0$. For the converges of the integral of $|p_a(t,\beta)|$ on $(1,+\infty)$, it is enough to show that $$\int_1^{+\infty}\int_\beta^{+\infty}\frac{\tau e^{-\frac{\tau^2}{4t}}}{\sqrt{4\pi t^3}}Ce^{\omega\tau}\df\tau\df t$$
is finite. This is a direct consequence of \lemref{lem:xnt3n1inft}, applied with $n=0$ and $f(\tau)=Ce^{\omega\tau}\tau$.\sqed

\noindent ({\it ii}\,): By \propref{prop:tildepexpdecay}, we just need to show the following integral converges
\begin{equation}\label{eq:paprime}
\int_0^{+\infty}\left|\int_\beta^{+\infty}\frac{\partial}{\partial t}\left(\frac{\tau e^{-\frac{\tau^2}{4t}}}{\sqrt{4\pi t^3}}\tilde p_a(\tau,\beta)\right)\df\tau\right|\df t,
\end{equation}
which is equal to 
$$\int_0^{+\infty}\left|\int_\beta^{+\infty}\left(-\frac{3\tau e^{-\frac{\tau^2}{4t}}}{2\sqrt{4\pi t^5}}+\frac{\tau^3 e^{-\frac{\tau^2}{4t}}}{4\sqrt{4\pi t^7}}\right)\tilde p_a(\tau,\beta)\df\tau\right|\df t.$$
On $t\in(1,+\infty)$, the integral is finite by \lemref{lem:xnt3n1inft}. For $t\in(0,1)$, find $C'\geqslant0$ such that $C<C'\tau^2$ for all $\tau\in[\beta,+\infty)$ which is possible for $\beta>0$. Then the integral is bounded above by
$$\int_0^{1}\int_\beta^{+\infty}\left(\frac{3\tau^3 e^{-\frac{\tau^2}{4t}}}{2\sqrt{4\pi t^5}}+\frac{\tau^5e^{-\frac{\tau^2}{4t}}}{4\sqrt{4\pi t^7}}\right)C'e^{\omega\tau}\df\tau\df t,$$
which is finite due to \lemref{lem:xnt2n01}.
\end{proof}

\subsection{Pointwise positivity for suitably symmetric initial conditions}

Our goal is to deal with positivity of the solution at each point. However, we are still not there yet, because we only find some condition (\thmref{thm:upluspos}) where $u_+(t)=u(t,1)+u(t,-1)$ is guaranteed to be nonnegative. Using substitution principle we just need to prove that $u(t,1),u(t,-1)\geqslant0$ for all $t\geqslant0$. Throughout we assume $a\in[0,e^{-1}]$ as in \thmref{thm:upluspos}. Now, our goal is to find some conditions on $u_0$ to ensure $u(t,1),u(t,-1)\geqslant0$. The very first and easy condition that we can impose is:
\begin{prop}\label{prop:sympos}
    Assume $0\leqslant u_0\in H^1(\IR)$, $0\leqslant \Phi\in L^1_\loc(\IR_0^+)$ and $a\in[0,e^{-1}]$. Further, assume that $u_0$ is an even function. Then $u(t,1)\geqslant 0$ and $u(t,-1)\geqslant 0$ for all $t\geqslant 0$. Consequently,
$u(t,x)\geqslant 0$ for all $t\geqslant 0$ and $x\in\IR\smallsetminus (-1,1)$.
\end{prop}
\begin{proof}
Due to symmetry, $u(t,1)=u(t,-1)$ for all $t\geqslant 0$. Their sum, $u_+(t)$, is nonnegative, according to \thmref{thm:upluspos}. Then they must be both nonnegative. 
\end{proof}
The symmetry condition on $u_0$ can be weakened. Recall the definition of $u_-(t)$ in \eqref{eq:def u pm}. Equation \eqref{eq:integral expr u} yields, by substracting the corresponding expressions for $x=1$ and $x=-1$,

\begin{equation}\label{eq:uminus}
u_-(t)=\int_{\IR}\left(\frac{e^{-\frac{|1-\tilde{x}|^2}{4t}}}{\sqrt{4\pi t}}-\frac{e^{-\frac{|1+\tilde{x}|^2}{4t}}}{\sqrt{4\pi t}}\right)u_0(\tilde{x})\df \tilde{x}.
\end{equation}
Then $u(t,1)=\frac{1}{2}(u_+(t)+u_-(t))=$
\begin{equation}\label{eq:up1pm}
\int_0^t2p_a(\tau,1)\Phi(t-\tau)\df\tau+\int_{\IR}(p_a(t,|1-\tilde{x}|)+p_a(t,|1+\tilde{x}|)+\cK_1(t,|1-\tilde{x}|)-\cK_1(t,|1+\tilde{x}|))u_0(\tilde{x})\df \tilde{x},
\end{equation}
and $u(t,-1)=\frac{1}{2}(u_+(t)-u_-(t))=$
\begin{equation}\label{eq:up-1pm}
\int_0^t2p_a(\tau,1)\Phi(t-\tau)\df\tau+\int_{\IR}(p_a(t,|1-\tilde{x}|)+p_a(t,|1+\tilde{x}|)+\cK_1(t,|1+\tilde{x}|)-\cK_1(t,|1-\tilde{x}|))u_0(\tilde{x})\df \tilde{x}.
\end{equation} 
As we only allow $a\in[0,e^{-1}]$, one of the main results from \secref{sec:outlineposuplus}, \propref{prop:positivity tilde pa}, shows that $p_a(t,\beta)\geqslant0$ for all $\beta\geqslant0$. The only terms in \eqref{eq:up1pm} and \eqref{eq:up-1pm} that may potentially tarnish the positivity are the ones in the spatial integrals.
Again, recalling the notations in \eqref{eq:tilde P} and \eqref{eq:tilde Q}, if we can prove that the integrals of the inverse transforms of
\begin{equation}
\tilde R_{a,+}(s,\tilde{x}):=\tilde P_a(s,|1-\tilde{x}|)+\tilde P_a(s,|1+\tilde{x}|)+\tilde Q(s,|1-\tilde{x}|)-\tilde Q(s,|1+\tilde{x}|)
\end{equation}
and
\begin{equation}
\tilde R_{a,-}(s,\tilde{x}):=\tilde P_a(s,|1-\tilde{x}|)+\tilde P_a(s,|1+\tilde{x}|)+\tilde Q(s,|1+\tilde{x}|)-\tilde Q(s,|1-\tilde{x}|)
\end{equation}
against $u_0(\tilde{x})$ are nonnegative, the desired result follows from \lemref{lem:lapsqrtsub} or \corref{cor:lapsqrtsub}.

Denote $\beta_+:=|1-\tilde{x}|$ and $\beta_-:=|1+\tilde{x}|$ -- interpreting them as distances from $\tilde{x}$ to $1$ and $-1$ respectively. Spell out 
\begin{equation}\label{eq:raspell}
\tilde R_{a,-}(s,\tilde{x})=\frac{e^{-\beta_+s}}{2s+2ae^{-s}}+\frac{e^{-\beta_-s}}{2s+2ae^{-s}}+\frac{e^{-\beta_-s}}{2s}-\frac{e^{-\beta_+s}}{2s}.
\end{equation}
The last two terms are transformed to $\frac{1}{2}\theta(t-\beta_-)-\frac{1}{2}\theta(t-\beta_+)$. If $\beta_-\leqslant\beta_+$, equivalently $\tilde{x}\leqslant0$,their sum is guaranteed to be nonnegative. In this case, it leads to imposing extra conditions on $u_0$ only for $\tilde{x}>0$, equivalently $\beta_->\beta_+$. Similarly, for $\tilde R_{a,+}$, it leads to imposing extra conditions on $u_0$ only for $\tilde{x}<0$. Thus, due to symmetry it is sufficient to just study $\tilde R_{a,-}$ and $\tilde{x}>0$.

We further distinguish two cases: $\tilde x<\frac{1}{2}$, equivalently $\beta_--\beta_+<1$, and $\tilde x\geqslant\frac{1}{2}$, equivalently, $\beta_--\beta_+\geqslant1$. For each case, we plot each term of the inverse transform $\tilde r_{a,-}:=\cL^{-1}\{\tilde R_{a,-}\}$ in \eqref{eq:raspell} as a function of $t$ for some fixed $\tilde{x}>0$; see Figure \ref{fig:ra2}. The curves of these four terms are colored in red, blue, orange, and green, respectively. 

\begin{figure}[h!]
    \centering
    \includegraphics[scale=0.58]{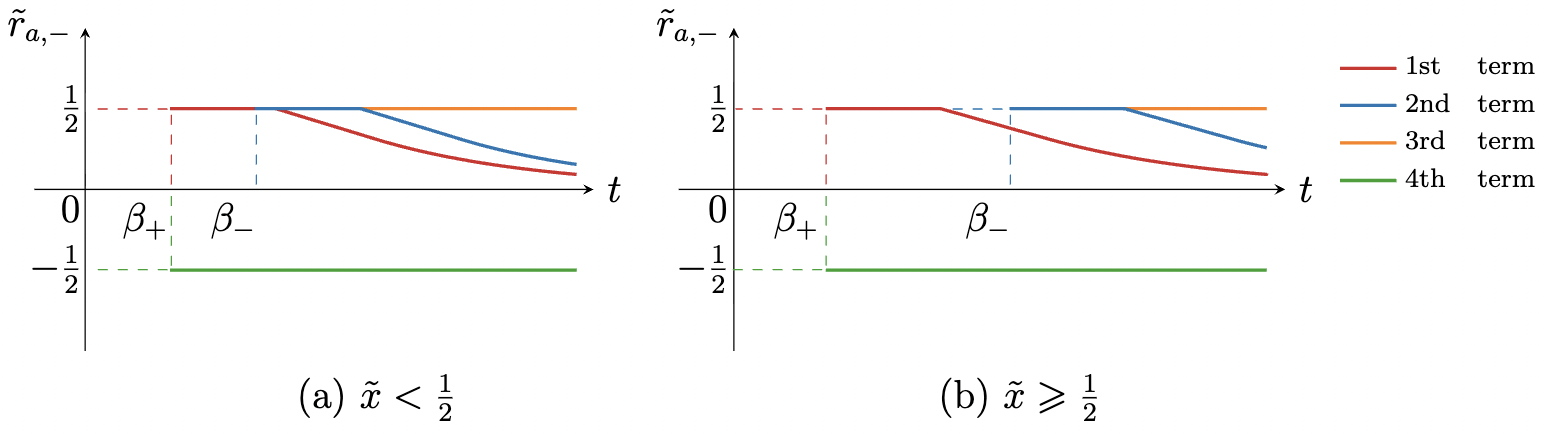}
    \caption{Shape of the inverse Laplace transform $\tilde r_{a,-}$ of each of the four terms in \eqref{eq:raspell}, respectively. (Left panel: $\displaystyle\tilde x <\frac{1}{2}$; Right panel: $\displaystyle\tilde x\geqslant\frac{1}{2}$.)} 
    \label{fig:ra2}
\end{figure}

Note that the absolute value of the horizontal segment of the first three terms is the same as the absolute value of the fourth term. Hence, the nonnegativity of $\tilde{r}_{a,-}$ is guaranteed in $t\in(0,\beta_++1)\cup(\beta_-, +\infty)$. So, the only \textit{problematic interval} is $(\beta_++1,\beta_-)$ where $\tilde r_{a,-}$ is negative. We are assuming $\tilde{x}>0$. At the beginning, $\beta_+,\beta_-$ overlap at $1$. When $\tilde{x}$ increases, $\beta_-$ always shift to the right, but $\beta_+$ first shifts to the left till $0$ when $\tilde{x}=1$, and then shifts to the right. 

\begin{prop}\label{prop:halfsympos}
Assume $0\leqslant u_0\in H^1(\IR)$, $0\leqslant \Phi\in L^1_\loc(\IR_0^+)$ and $a\in[0,e^{-1}]$. If $u_0$ is even on $\IR\smallsetminus \bigl[-\frac{1}{2},\frac{1}{2}\bigr]$, then $u(t,x)\geqslant 0$ for all $t\geqslant 0$ and $x\in\IR\smallsetminus(-1,1)$.
\end{prop}
\begin{proof}
We decompose $u_0$ into the sum of two nonnegative functions $u_{01}$ and $u_{02}$. $u_{01}$ is $0$ on $(-\frac{1}{2},\frac{1}{2})$ and equal to $u_0$ outside $(-\frac{1}{2},\frac{1}{2})$; $u_{02}$ is the other way around. If $u_{01}$ serves as the initial condition, it will generate nonnegative solutions by \propref{prop:sympos}. If $u_{02}$ serves as the initial condition, it will generate nonnegative solutions as well. This is because when $\tilde{x}\in(-\frac{1}{2},\frac{1}{2})$, $|\beta_+-\beta_-|\leqslant1$. Then the \textit{problematic interval} $(\beta_++1,\beta_-)$ of $\tilde r_{a,-}$ is empty. In this case, $\tilde r_{a,-}(t,\tilde{x})$ is nonnegative for all $t$ and $\tilde{x}\in(-\frac{1}{2},\frac{1}{2})$. While $u_{02}$ is only nonzero inside $(-\frac{1}{2},\frac{1}{2})$. Then the spatial integral is nonnegative.
\end{proof}

\subsection{Proof of main positivity result}

We shall now prove our main result on positivity,  \thmref{thm:positivity}.

\begin{proof}[Proof of \thmref{thm:positivity}]
As before, we just consider $\tilde r_{a,-}(t,\tilde{x})$ and $u_0(\tilde x)$ for $\tilde{x}>0$ and $t\geqslant0$. We need to carefully switch between the spatial and temporal coordinates back and forth.
\begin{itemize}[leftmargin=*]
\item Case 1: $t\geqslant2$. When $\tilde{x}\in(0,1)$, $\beta_-=\tilde{x}+1<2\leqslant t$. $t$ is outside the \textit{problematic interval} $(\beta_++1,\beta_-)$, so $\tilde r_{a,+}$ is always nonnegative. In this case, the function is only negative on $\tilde{x}\in(1,+\infty)$ and $t\in(\beta_++1,\beta_-)$. When $\tilde{x}\in(1,+\infty)$, $\beta_-=\tilde{x}+1,\beta_+=\tilde{x}-1$. Thus, $t$ is in the \textit{problematic interval} $(\beta_++1,\beta_-)=(\tilde{x},\tilde{x}+1)$ exactly when $\tilde{x}\in(t-1,t)$. Then  $\tilde r_{a,-}$ is negative, which is a line with slope $\frac{a}{2}$ as a function of $\tilde{x}$ (cf. \eqref{eq:lapinvpsub}). However, on the interval $\tilde{x}\in(t-2,t-1)$, $\beta_-=\tilde{x}+1<t<\tilde{x}+2=\beta_-+1$. Then $\tilde r_{a,-}\geqslant\frac{1}{2}$ as shown in Fig. \ref{fig:raspatial2}. 

\begin{minipage}{\linewidth}
    \centering
    \includegraphics[width = 0.9\textwidth]{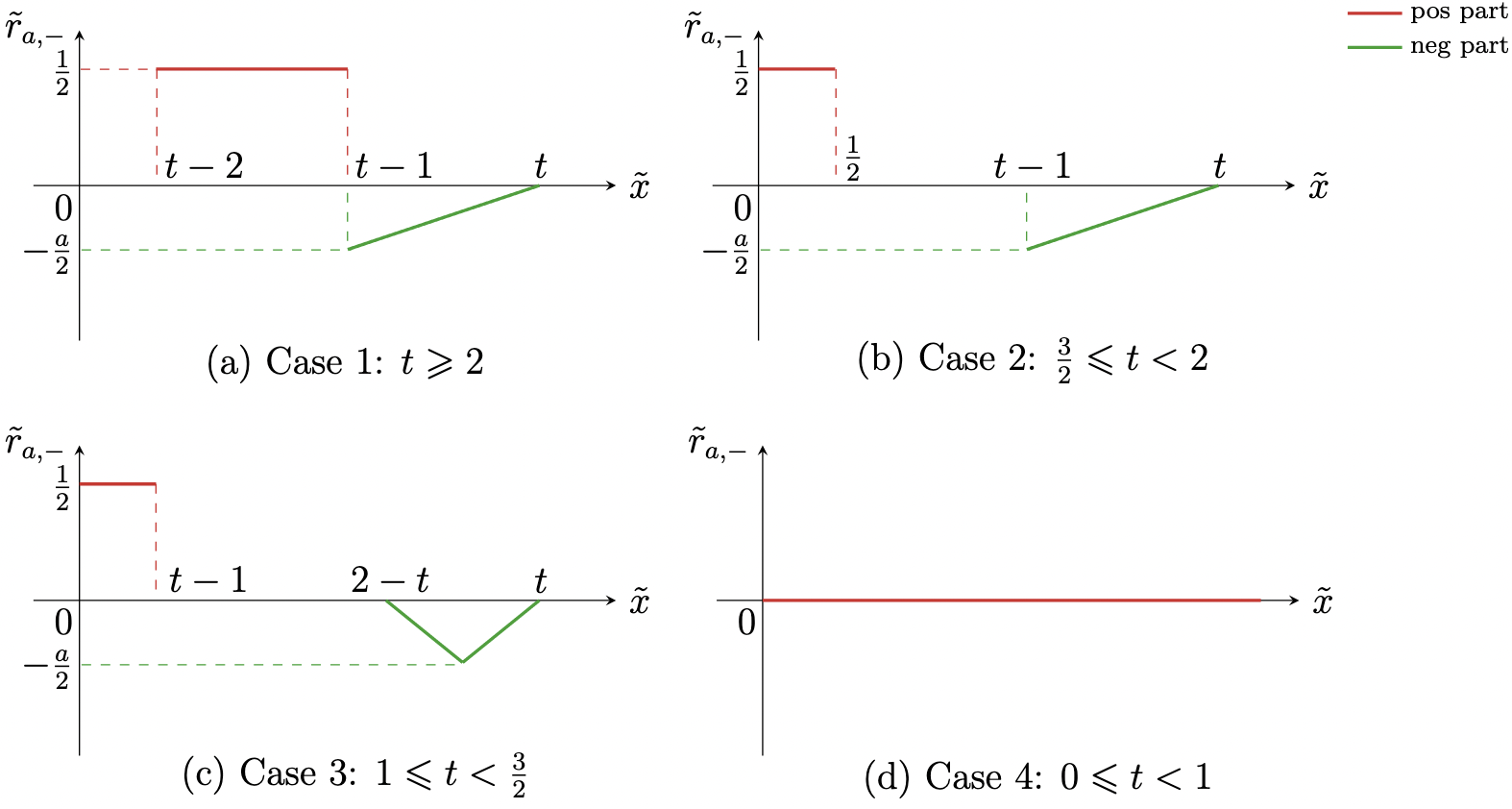}
    \captionof{figure}{The inverse Laplace transform $\tilde{r}_{a,-}$ regarding various values of $t$. The value of $\tilde{r}_{a,-}$ is dominated by the first term (positive part in red lines) and the fourth term (negative part in green lines). However, since the negative part only appears in the historical time domain, the positivity of $\tilde{r}_{a,-}$ in $(t, +\infty)$ is guaranteed.}
    \label{fig:raspatial2}
\end{minipage}

Thus, on this interval $(t-2,t-1)$, we have $\tilde r_{a,-}(t,\tilde{x})>|\tilde r_{a,-}(t,\tilde{x}+1)|=-\tilde r_{a,-}(t,\tilde{x}+1)$. As $u_0$ is monotonically decreasing, $u_0(\tilde{x})\geqslant u_0(\tilde{x}+1)$. Now, every term is nonnegative, so $\tilde r_{a,-}(t,\tilde{x})u_0(\tilde{x})\geqslant-\tilde r_{a,-}(t,\tilde{x}+1)u_0(\tilde{x}+1)$. Then we will have
$$\int_{t-2}^{t-1}\tilde r_{a,-}(t,\tilde{x})u_0(\tilde{x})\df \tilde{x}+\int_{t-1}^{t}\tilde r_{a,-}(t,\tilde{x})u_0(\tilde{x})\df \tilde{x}\geqslant0.$$
As the only negative of $\tilde r_{a,-}$ appears on $(t-1,t)$. Then we can conclude that the whole integral is nonnegative.

\item Case 2: $\frac{3}{2}\leqslant t<2$. The idea is very similar to Case 1, but we need to exploit the conditions a bit more. For $\tilde{x}\in (0,\frac{1}{2})$, $\beta_-<\frac{3}{2}\leqslant t$ and $t<2<\beta_-+1$. When $t\in(\beta_-,\beta_-+1)$, $\tilde r_{a,-}\geqslant\frac{1}{2}$ as shown in Fig. \ref{fig:raspatial2}. Thus, $\tilde r_a\geqslant\frac{1}{2}$ on $\tilde{x}\in(0,\frac{1}{2})$ during this period of time. The negative part of $\tilde r_{a,-}$ only appears outside $(0,\frac{1}{2})$ for at most one unit interval with absolute value below $\frac{a}{2}$; see \propref{prop:halfsympos}. Then the negative part of the integral is bounded from below by $$-\int_{x_0}^{x_0+1}\frac{a}{2}u_0(\tilde{x})\df \tilde{x}$$
for some $x_0\geqslant\frac{1}{2}$.
Because $u_0(\tilde{x})$ is monotonically decreasing, it is further bounded from below by
$$-2\int_{x_0}^{x_0+\frac{1}{2}}\frac{a}{2}u_0(\tilde{x})\df \tilde{x}=-\int_{x_0}^{x_0+\frac{1}{2}} au_0(\tilde{x})\df \tilde{x}.$$
However, 
$$\int_{0}^{\frac{1}{2}}au_0(\tilde{x})\df \tilde{x}-\int_{x_0}^{x_0+\frac{1}{2}}au_0(\tilde{x})\df \tilde{x}\geqslant0,$$
because $\tilde r_{a,-}\geqslant\frac{1}{2}$ on $(0,\frac{1}{2})$, $\tilde r_{a,-}u_0\geqslant\frac{1}{2}u_0\geqslant au_0$ on $(0,\frac{1}{2})$ where we use the range for $a\in[0,e^{-1}]$ and $e=2.718\cdots>2$. We can conclude that the whole integral is nonnegative.

\item Case 3: $1<t<\frac{3}{2}$. $\tilde{x}\in(0,t-1)$, $\tilde r_{a,-}\geqslant\frac{1}{2}$. The negative part is on $\tilde{x}\in (2-t,t)$ with absolute value bounded by $\frac{a}{2}$. The length of the first interval has length is half of the second one. Using the same trick as in Case 2, we can conclude that the whole integral is nonnegative.

\item Case 4: $0\leqslant t\leqslant1$. We can see that the red and green curves in Fig. \ref{fig:ra2} always cancel each other out on $[\beta_+,\beta_++1]$. Thus, $\tilde r_{a,-}=0$ for $t\in[0,1]\subseteq[0,\beta_++1]$ and all $\tilde{x}>0$.
\end{itemize}
This completes the proof.
\end{proof}

\section{Convergence to a Steady State}\label{Sec_steady_state}

First of all, in order to expect a nonnegative steady state of the system that makes sense for real life application, we shall assume $a\in(0,e^{-1}]$. Moreover, a necessary condition for the existence of a steady state is that the temporal limits of the boundary values $u(t,\pm1)$ at the ``cell boundary" should exist.  However, we shall see in the following that to ensure the existence of steady states stronger regularity conditions on $u_+$ and $u_-$ are needed -- specifically, particular regularity of their derivatives.

Note that $u_+$ and $u_-$ do possess derivatives in $L_{\loc}^1$. Because we have already computed their Laplace transforms in \secref{Sec:Laplace} and \secref{Sec:Pos}, and differentiation in the $t$-domain corresponds to multiplication by $s$ in the  $s$-domain, it follows that their derivatives also belong to $L_{\loc}^1$. So do $u'(t,\pm1)\in L^1_\loc(\IR_0^+)$.

\begin{prop}\label{prop:steadyderive}
Assume that $u_+',u_-'\in L^1(\IR_0^+)$. If $\lim_{t\apc+\infty}u_+(t)$ and $\lim_{t\apc+\infty}u_-(t)$ exist, then the mild solution $u(t)$ to System \eqref{eq: dimensionless system} converges pointwise on $[1,+\infty)$ to the steady state value $\lim_{t\apc+\infty} u(t,1)$, respectively to $\lim_{t\apc+\infty} u(t,-1)$ on $(-\infty,-1]$.
\end{prop}
\begin{proof}
Due to symmetry, it is enough to show that if $\lim_{t\apc+\infty}u(t,1)$ exists, then $u$ has a steady state on $[1,+\infty)$. By the comparison principle (see Corollary 2.5 of \cite{Salsa2022partial}), we can equivalently consider the diffusive system of $u_R$ on the semi-infinite real line $\IR_0^+$ whose boundary condition at $x=0$ is forced by the given solution of $u(t,1)$,

$$\left\{\begin{alignedat}{2}
    \frac{\partial u_R}{\partial t}&=\frac{\partial^2 u_R}{\partial x^2} && \text{for } (t,x)\in\IR_0^+\times\IR_0^+\\
    u_R(t,0)&=u(t,1) && \text{for } t\in\IR_0^+\\
    u_R(0,x)&=u_0(x+1)\quad && \text{for } x\in \IR_0^+
\end{alignedat}\right..
$$
Substitute $u_R$ by $\tilde u_R+u(t,1)$ to obtain a system with homogeneous boundary condition
\begin{equation}\label{eq:vrhom}
\left\{\begin{alignedat}{2}
    \frac{\partial\tilde u_R}{\partial t}&=\frac{\partial^2 \tilde u_R}{\partial x^2} + u'(t,1) && \text{for } (t,x)\in\IR_0^+\times\IR_0^+\\
    \tilde u_R(t,0)&=0 && \text{for } t\in\IR_0^+\\
    \tilde u_R(0,x)&= u_0(x+1)- u(0,1)\quad && \text{for } x\in \IR_0^+
\end{alignedat}\right..
\end{equation}
Evaluating the Variation of Constants Formula \eqref{eq:semigroupmild} in the setting of System \ref{eq:vrhom}, we get \begin{equation}\label{eq:vrkern}
    \tilde{v}_R(t,x)=\int_0^t \int_{\IR_0^+}\cK_{1,\mathrm{semi}}(t-\tau,x,\tilde x) u'(\tau,1)\df \tilde x\df\tau+\int_{\IR_0^+}\cK_{1,\mathrm{semi}}(t,x,\tilde x)(u_0(\tilde x+1)-u(0,1))\df\tilde x,\end{equation}
where the heat kernel for $\IR_0^+$ with Dirichlet boundary condition at $x=0$ is given by $$\cK_{1,\mathrm{semi}}(t,x,\tilde x):=\frac{1}{\sqrt{4\pi t}}\left(e^{-\frac{(x-\tilde x)^2}{4t}}-e^{-\frac{(x+\tilde x)^2}{4t}}\right)$$
according to \cite{carslaw_heatconduction}, Section 18, p.33. The second integral of \eqref{eq:vrkern} converges to zero similar to \lemref{lem:uminuslim}. The first integral of \eqref{eq:vrkern} can be further computed as
\[
    \int_0^t\mathrm{erf}\left(\frac{x}{\sqrt{4(t-\tau)}}\right)u'(\tau,1)\df\tau,
\]
where the \textit{error function} is defined as $$\mathrm{erf}(x):=\frac{2}{\sqrt{\pi}}\int_0^xe^{-y^2}\df y.$$ All conditions of \lemref{lem:convolutionlimit} are satisfied. Then we can show that $\lim_{t\apc\infty}\tilde u_R(t,x)=0$ pointwise. Eventually, $\lim_{t\apc+\infty}u(t,x)=\lim_{t\apc+\infty} u(t,1)$ pointwise.

A similar result is obtained on $(-\infty,-1]$: $\lim_{t\apc+\infty}u(t,x)=\lim_{t\apc+\infty} u(t,-1)$ pointwise.
\end{proof}

We shall now establish the integrability conditions for $u'_+$ and $u'_-$ that are assumed in \propref{prop:steadyderive}. First, we prove this for $u_-$.

\begin{lem}\label{lem:uminuslim}
$\lim_{t\apc+\infty}u_-(t)=0$, and $u_-'\in L^1(\IR_0^+)$.
\end{lem}
\begin{proof}
Apply H\"older's inequality to \eqref{eq:uminus}. One gets
$$|u_-(t)|\leqslant\Vert u_0\Vert_{L^2}\cdot\sqrt[4]{\frac{2}{\pi t}},$$
which approaches to $0$ when $t\apc+\infty$.

For the absolute integrability of $u_-'(t)$, we just need to show that the following integral is finite
$$\int_0^{+\infty}\left|\int_{\IR}\frac{\partial}{\partial t}\left(\frac{e^{-\frac{|1-\tilde{x}|^2}{4t}}}{\sqrt{4\pi t}}-\frac{e^{-\frac{|1+\tilde{x}|^2}{4t}}}{\sqrt{4\pi t}}\right)u_0(\tilde x)\df\tilde x\right|\df t.$$
We compute 
$$\frac{\partial}{\partial t}\left(\frac{e^{-\frac{|1\pm\tilde{x}|^2}{4t}}}{\sqrt{4\pi t}}\right)=\frac{e^{-\frac{|1\pm\tilde x|^2}{4t}}}{4\sqrt{4\pi t^5}}(|1\pm\tilde x|^2-2t).$$
As $u_0$ is in $H^1(\IR)$, then we can apply \lemref{lem:xnt3n1inft} to conclude the integral is finite on the interval $(1,+\infty)$.

By the property of the heat kernel, \begin{equation}\label{eq:heatproperty}
    \frac{\partial\cK_1(t,|x-y|)}{\partial t}=\frac{\partial^2\cK_1(t,|x-y|)}{\partial x^2},
\end{equation}
and Stokes' theorem (integration by parts), it is now enough to show that the following integral is finite
$$\int_0^1\int_{\IR}\left(\frac{e^{-\frac{|1- \tilde{x}|^2}{4t}}}{\sqrt{4\pi t^3}}\frac{|1-\tilde x|}{2}+\frac{e^{-\frac{|1+\tilde{x}|^2}{4t}}}{\sqrt{4\pi t^3}}\frac{|1+\tilde x|}{2}
\right)|u_0'(\tilde x)|\df\tilde x\df t.
$$
This is true due to \lemref{lem:xnt2n01}.
\end{proof}

Establishing a similar result for $u_+$ is more complicated. Naturally, we should require that the temporal limit of the secretion term, $\lim_{t\apc+\infty}\Phi(t)=\Phi_{\infty}\in\IR^+_0$, exists in order to expect the existence of the steady states. By this limit we mean, that there exists a version of $\Phi$ in its equivalence class in $L^\infty_\loc(\IR)$ that has the indicated limit. Otherwise, the cell would keep secreting more and more compound indefinitely, preventing the system from reaching equilibrium. However, we shall see that an additional regularity on $\Phi$ is needed.

\begin{prop}\label{lem:upluslim}
Assume $\lim_{t\apc+\infty}\Phi(t)=\Phi_{\infty}\in\IR^+_0$ exists. Then
$$\lim_{t\apc+\infty}u_+(t)=2\Phi_\infty\int_0^{+\infty}p_a(t,1)\df t.$$
\end{prop}
\begin{proof}
By \lemref{lem:convolutionlimit}, the temporal limit of $u_{+}^{\mathrm{tp}}$ in is \eqref{eq:uplusinpa}
$$\lim_{t\apc+\infty}u_{+}^{\mathrm{tp}}(t)=\lim_{t\apc+\infty}\Phi(t)\cdot\int_0^{+\infty}p_a(t,1)\df t$$
because $2\Phi\in L^\infty(\IR_0^+)$ with $\lim_{t\apc+\infty}2\Phi=2\Phi_{\infty}$, and $p_a(t,1)\in L^1(\IR_0^+)$ by \propref{prop:tildepexpdecay}. On the other hand,
$u_{+}^{\mathrm{sp}}$
in \eqref{eq:uplusinpa} has decaying rate $O(t^{-\frac{1}{4}})$, which can be shown exactly as in \lemref{lem:uminuslim}.
\end{proof}

\begin{prop}\label{prop:uplusprimeregular}
Assume $\lim_{t\apc+\infty}\Phi(t)=\Phi_{\infty}\in \IR_0^+$ exists, and $\Phi-\Phi_{\infty}\in L^1(\IR_0^+)$. Then $u_+'\in L^1(\IR_0^+)$.
\end{prop}
\begin{proof}
In order to show that $u_+'$ is absolutely integrable, we show this separately for $(u_+^{\mathrm{tp}})'$ and $(u_+^{\mathrm{sp}})'$. 
For $(u_+^{\mathrm{tp}})'$ use the basic properties of Laplace transform; see Table \cosref{tab:laplace}{A.1}, 
\begin{equation}\label{eq:uplutpprime}
(u_+^{\mathrm{tp}})'=\int_0^t2(p_a'(t-\tau,1)-p_a(0,1))\Phi(\tau)\df\tau.
\end{equation}
By \eqref{eq:paboundinf}, $p_a(0,1)=0$. In this case,
$(u_+^{\mathrm{tp}})'=2p_a'(\cdot,1)\ast\Phi=2p_a'(\cdot,1)\ast(\Phi-\Phi_\infty+\Phi_\infty)=2p_a'(\cdot,1)\ast\Phi=2p_a'(\cdot,1)\ast(\Phi-\Phi_\infty)+2\Phi_\infty\ast p_a'(\cdot,1)$. \propref{prop:tildepexpdecay} shows that $p_a'(\cdot,1)$ is in $L^1(\IR_0^+)$. By Young's Inequality (see\cite{folland-analysis} Proposition 8.9, p. 241), $2p_a'(\cdot,1)\ast(\Phi-\Phi_\infty)$ is in $L^1(\IR_0^+)$. Hence, $(u_+^{\mathrm{tp}})'$ is in $L^1(\IR_0^+)$.

It remains to show that $(u_+^{\mathrm{sp}})'$ is absolutely integrable. In other words, we need to show that the following integral is finite: \begin{equation}\label{eq:uplusspder}
\int_0^{+\infty}\left|\int_{\IR}\int_0^{+\infty}\frac{\partial}{\partial t}\left(\frac{\tau e^{-\frac{\tau^2}{4t}}}{\sqrt{4\pi t^3}}\left(\tilde p_a(\tau,|1-\tilde x|)+\tilde p_a(\tau,|1+\tilde x|)\right)\right)\df\tau u_0(\tilde x)\df\tilde x\right|\df t.
\end{equation}
The technique remains the same as in \propref{prop:paw11} for $t\in(1,+\infty)$: computing the time derivative directly and invoking \lemref{lem:xnt3n1inft}. On $t\in(0,1)$, the argument is trickier. First observe that
$$
\frac{\tau e^{-\frac{\tau^2}{4t}}}{\sqrt{4\pi t^3}}=-2\cdot\frac{\partial}{\partial\tau}\left(\frac{e^{-\frac{\tau^2}{4t}}}{\sqrt{4\pi t}}\right),$$
and $$\tilde p_a'(\tau,\beta)=-a\tilde p_a(\tau-1,\beta).$$ Together with Stokes' theorem, the finiteness of \eqref{eq:uplusspder} is reduced to verifying the finiteness of the following integrals,
\begin{align*}
\,&\int_0^1\int_{\IR}\left|\int_0^{+\infty}\frac{\partial}{\partial t}\frac{\partial}{\partial\tau}\left(\frac{e^{-\frac{\tau^2}{4t}}}{\sqrt{4\pi t}}\right)\tilde p_a(\tau,|1\pm\tilde x|)\df\tau u_0\df\tilde x\right|\df t\\
=\,&\int_0^1\left|\int_{\IR}\frac{\partial}{\partial t}\left(-\frac{e^{-\frac{|1\pm\tilde x|^2}{4t}}}{\sqrt{4\pi t}}\tilde p_a(|1\pm\tilde x|,|1\pm\tilde x|)-\int_0^{+\infty}\frac{e^{-\frac{\tau^2}{4t}}}{\sqrt{4\pi t}}\tilde p_a'(\tau,|1\pm\tilde x|)\right)\df\tau u_0\df x\right|\df t\\
\leqslant&\int_0^1\left|\int_{\IR}\frac{\partial}{\partial t}\left(\frac{e^{-\frac{|1\pm\tilde x|^2}{4t}}}{2\sqrt{4\pi t}}\right)u_0\df \tilde x\right|\df t+\int_0^1\left|\int_{\IR}\int_0^{+\infty}\frac{\partial}{\partial t}\left(\frac{e^{-\frac{\tau^2}{4t}}}{\sqrt{4\pi t}}a\tilde p_a(\tau-1,|1\pm\tilde x|)\right)\df\tau u_0\df\tilde x\right|\df t.
\end{align*}
The first integral is finite as this was already proven in \lemref{lem:uminuslim} by noticing the property \eqref{eq:heatproperty} of the heat kernel. Because $p_a(\tau-1,|1\pm\tilde x|)$ is identically $0$ on $[0,1+|1\pm\tilde x|)$, we may reset the lower limit of the innermost integral to be $|1\pm\tilde x|$. The second integral is now upper-bounded by
\begin{equation}\label{eq:3fold1plusx}
\int_0^1\int_{\IR}\int_{|1\pm\tilde x|}^{+\infty}\left(\frac{e^{-\frac{\tau^2}{4t}}}{2\sqrt{4\pi t^3}}+\frac{\tau^2e^{-\frac{\tau^2}{4t}}}{4\sqrt{4\pi t^5}}\right)Ce^{\omega\tau}\df\tau|u_0|\df\tilde x\df t.
\end{equation}
Use the estimation \eqref{eq:paexpboundnega0} and \eqref{eq:paexpboundnega2}, \eqref{eq:3fold1plusx} is further bounded from above by 
\begin{equation}\label{eq:extsquare}
\int_0^1\int_{\IR}\left(\frac{Ce^{-\frac{|1\pm\tilde x|}{2\sqrt{t}}}}{2\sqrt{\pi t^2}}+\frac{Ce^{-\frac{|1\pm\tilde x|}{2\sqrt{t}}}}{4\sqrt{\pi t^4}}(|1\pm\tilde x|^2+4\sqrt{t}|1+\tilde x|+8t)\right)|u_0|\df\tilde x\df t.
\end{equation}
Up to translation, each term in \eqref{eq:extsquare} satisfies the conditions in \lemref{lem:explinearbound}.
\end{proof}

Putting all preliminary results together, we reach our final result of this section,

\begin{proof}[Proof of \thmref{thm:steadystate}]
$\lim_{t\apc+\infty}u_+(t)$ exists by \lemref{lem:upluslim}. We apply \thmref{thm:fvt} to conclude that the limit is equal to $\lim_{s\apc0^+}sU_+(s)$. Given that $\Phi\in L^{\infty}_{\mathrm{loc}}(\IR_0^+)$ and $\lim_{t\apc+\infty}\Phi(t)=\Phi_{\infty}$. Then $\lim_{s\apc 0^+}s\cdot\Theta(s)=\Phi_{\infty}$. Thus, 
\[
    \lim_{s\apc0^+}sU_+(s) =\lim_{s\apc0^+}s(2P_a(s,1)\Theta(s)) 
    =\lim_{s\apc0^+}\frac{e^{-\sqrt{s}}}{\sqrt{s}+ae^{-\sqrt{s}}}\cdot\lim_{s\apc0^+}s\Theta(s)
    =\frac{\Phi_{\infty}}{a}.
\]
Moreover, $\lim_{t\apc+\infty}u_-(t)=0$ by \lemref{lem:uminuslim}, which shows that 
\[
    \lim_{t\apc+\infty}u(t,\pm1)=\frac{\Phi_{\infty}}{2a}.
\]
The regularity of the derivates is ensured by \lemref{lem:uminuslim} and \propref{prop:uplusprimeregular}. Finally, the steady state on $\IR\smallsetminus(-1,1)$ is obtained via \propref{prop:steadyderive}.
\end{proof}

\begin{rem}
It is crucial to have the all these extra $L^1$-regularity conditions on $u'(t,\pm1)$, $u_\pm'$, or $\Phi-\Phi_{\infty}$. Qualitatively, this just means that these functions approaching to their temporal limits with ``mild fluctuations". A key step is to apply Young's inequality and \lemref{lem:convolutionlimit} for convolutions. (The condition in ``$\Phi-\Phi_{\infty}\in L^1(\IR^+_0)$" in \propref{prop:uplusprimeregular} can also be substituted -- but not equivalently -- by ``$\Phi'\in L^1(\IR^+_0)$".) According to \lemref{lem:c00l1counter}, the convolution of an $L^1$-function and a $C_0^0$-function, a continuous function vanishing at infinity, is not necessarily an $L^1$-function. Namely, we can construct a $\Phi$ with temporal limit $0$ (without assuming that $\Phi-\Phi_{\infty}$ or $\Phi'$ belongs to $L^1(\IR_0^+)$), yet $u_+'\not\in L^1(\IR_0^+)$. Consequently, this leads non-existence of the steady states of the system.
\end{rem}

\section{Numerical Results}\label{Sec_numerical_results}
Although we have good control over $\tilde p_a(t,\beta)$ -- as illustrated in Fig. \ref{fig:tildepapos} and Fig. \ref{fig:tildepanega} -- thanks to its closed-form expression, this function is primarily introduced to study further theoretical insights of $p_a(t, \beta)$, which is more directly tied to the dynamics of the system. However, 
$p_a(t, \beta)$ has no such analytic formula. Therefore, in this section, we investigate the behaviours of $p_a(t, \beta)$ from a numerical perspective.

The function $p_a(t,\beta)$ is obtained as the inverse Laplace transform of $P_a(s,\beta)$ given in \eqref{eq:lappspell} by the Bromwich integral formula (\corref{cor:bromwich}),
\begin{equation}\label{eq:pabromwich}
p_a(t,\beta)=\frac{1}{2\pi i}\int_{\sigma-i\infty}^{\sigma+i\infty}\frac{e^{st-\beta\sqrt{s}}}{2\sqrt{s}+2ae^{-\sqrt{s}}}\df s.
\end{equation}
We explore its behaviour under various choices of $a$ and $\beta$, aiming to reveal its key qualitative and quantitative features. The improper integral \eqref{eq:pabromwich} is approximated by the following definite integral:
\begin{equation}
p_a(t,\beta)\approx\frac{e^{\sigma t}}{2\pi }\int_{-L}^L\Re\left(e^{it\xi}P_a(\sigma+it\xi ,\beta)\right)\df\xi.
\end{equation} 
After testing different values, we found that setting $\sigma = 0.1$ and the truncation limit $L = 50$ yields stable and reliable numerical results.

Fig. \ref{fig:asmall} shows the graphs of $p_a(t,\beta)$ for $a=0.25$ and $a=e^{-1}$, along with various values of $\beta$. These values of $a$ lie in the range $[0,e^{-1}]$, as derived from the key theoretical results in Section \ref{Sec:Pos}, and the numerical simulations confirm the theoretical results. For $a>e^{-1}$ but not yet very large, $p_a(t, \beta)$ can be negative and appears as a stable oscilation as time proceeds; see Figure \ref{fig:a2}. 

\begin{figure}[h!]
    \centering
    \subfigure[$a = 0.25$]{
    \includegraphics[width = 0.48\textwidth]{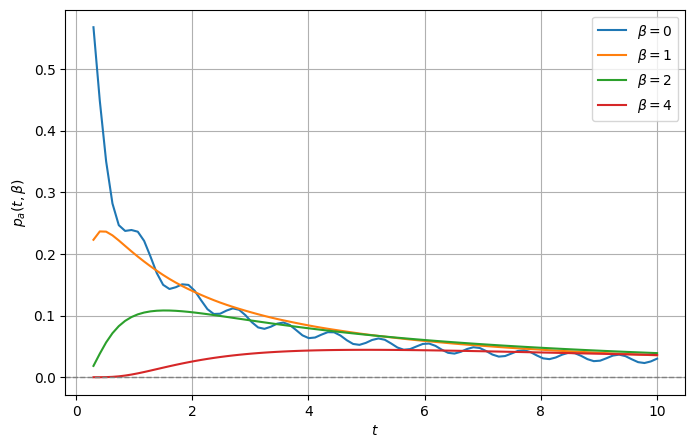}}
    \subfigure[$a=e^{-1}$]{
    \includegraphics[width = 0.48\textwidth]{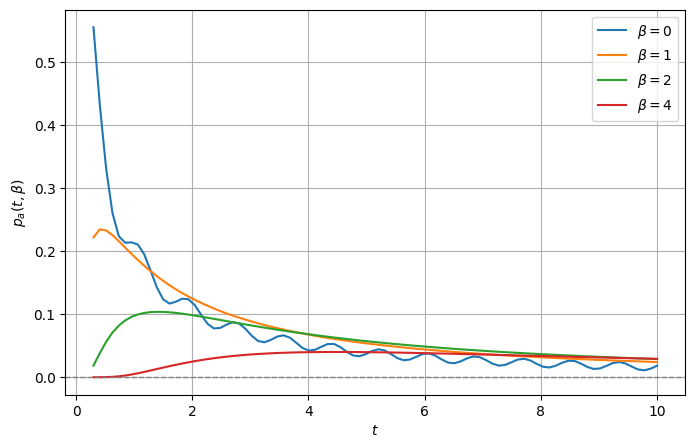}}
    \caption{The plot shows the values of $p_a(t,\beta)$ against $t$, with $a=0.25$ (Panel (a)) and $a=e^{-1}$ (Panel (b)). In each subfigure, we consider $\beta = \{0,1,2,4\}$.}
    \label{fig:asmall}
\end{figure}

\begin{figure}[h!]
    \centering
    \subfigure[$a=1.0$]{
    \includegraphics[width = 0.48\textwidth]{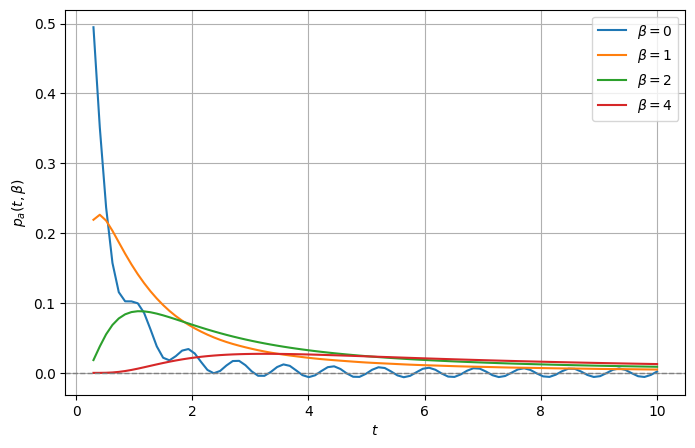}}
    \subfigure[$a=2.0$]{
    \includegraphics[width = 0.48\textwidth]{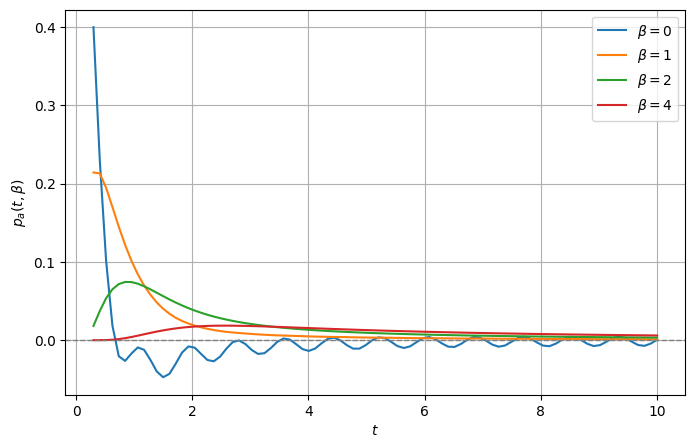}}
    \caption{The plot shows the values of $p_a(t,\beta)$ against $t$ and negative value starts appearing as stable oscilations, with $a=1.0$ (Panel (a)) and $a=2.0$ (Panel (b)). In each subfigure, we consider $\beta \in \{0,1,2,4\}$.}
    \label{fig:a2}
\end{figure}

For large values of $a$, say $a=50$, we expect that $p_a(t,\beta)$ to be an unstable oscillation (Fig. \ref{fig:tildepapos}) as the theoretical results suggest. Here one needs to be careful with selecting an eligible integral contour, i.e., to evaluate the right-hand side of Expression \eqref{eq:pabromwich}, there should be no poles on the right plane of the vertical line $\mathrm{Re} = \sigma$. In the example of $a = 50$, $P_a(s,\beta)$ has its principle pole
\begin{align*}
    \sqrt{s}+50e^{-\sqrt{s}}=0\;\;&\Leftrightarrow\;\; \sqrt{s}e^{-\sqrt{s}}=-50\\
    \;\;&\Leftrightarrow\;\; s=[W_k(-50)]^2.
\end{align*}
In this case, the principle pole is $[W_0(-50)]^2\approx 1.193+12.686i$. We need to choose $\sigma>\Re([W_0(-50)]^2)\approx1.193$. Therefore, our original choice of $\sigma=0.1$ fails to produce a trustworthy result (see Fig. \ref{fig:a50}(a)). Instead, in Figure \ref{fig:a50}(b), we used $\sigma = 1.2$. 
\begin{figure}[h!]
    \centering
    \subfigure[$\sigma = 0.1$]{
    \includegraphics[width = 0.48\textwidth]{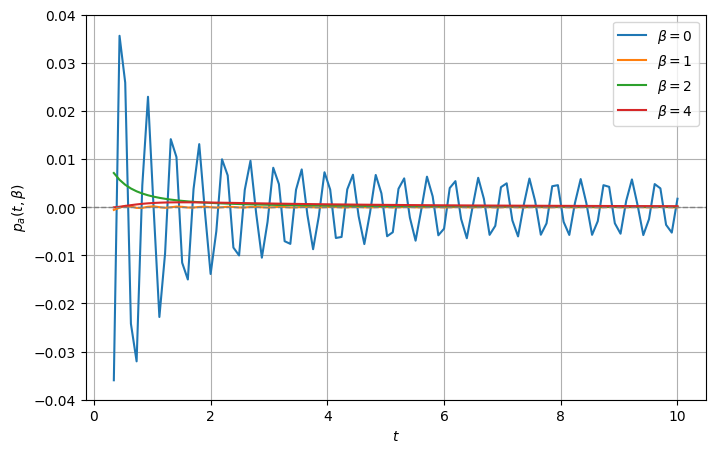}}
    \subfigure[$\sigma = 1.2$]{
    \includegraphics[width = 0.48\textwidth]{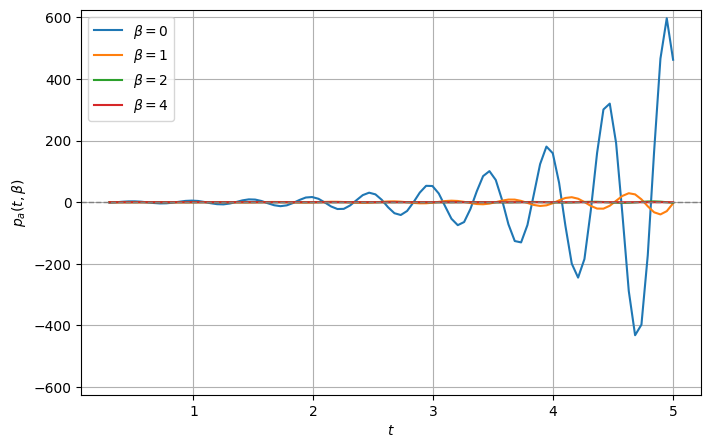}}
    \caption{The plot shows the values of $p_a(t, \beta)$ against $t$ when $a = 50$, with $\sigma = 0.1$ (Panel (a)) using an improper integral contour as defined in Expression \eqref{eq:pabromwich}and $\sigma = 1.2$ (Panel (b)) using a proper contour.}
    \label{fig:a50}
\end{figure}

The critical case happens when $[W_0(-\alpha_0)]^2$ is purely imaginary. Thus, $\arg(W_0(-\alpha_0))=\frac{\pi}{4}$ and $W_0(-\alpha_0)=x+ix$ for some $x>0$. By definition, $$(x+ix)e^{x+ix}=-\alpha_0, \qquad \alpha_0\in\IR$$ which implies that 
\begin{align*}
&\Im\left(e^x(x+ix)(\cos(x)+i\sin(x)\right)=0\\ \Rightarrow\;\; & \sin(x)+\cos(x)=0\\
\Rightarrow\;\; & x=\frac{3\pi}{4}+k\pi,\quad k\in\IZ.
\end{align*}
The smallest positive solution is then $\displaystyle x=\frac{3\pi}{4}$. Together with the relation 
$$-\alpha_0 = \Re\left(e^x(x+ix)(\cos(x)+i\sin(x)\right),$$ and hence, $$\alpha_0=\frac{3\pi\sqrt{2}}{4}e^{\frac{3\pi}{4}}\approx35.157.$$ In this case, we obtain a periodic undamped oscillation as shown in Fig. \ref{fig:atheta_0})
\begin{figure}[h!]
    \centering
    \includegraphics[width = 0.6\textwidth]{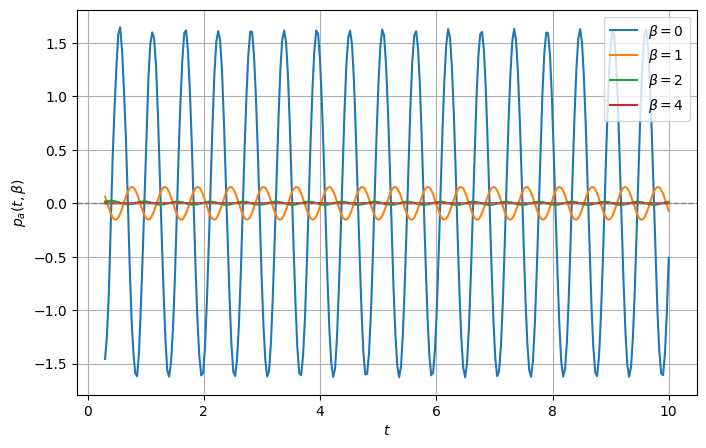}
    \caption{Undamped oscillation occurs when $[W_0(-\alpha_0)]^2$ only contains imaginary part. Here, we take $a = 35.157$ and $\beta\in\{0,1,2,4\}$.}
    \label{fig:atheta_0}
\end{figure}

Further using Monte-Carlo simulation, we can find the purple region where $p_a$ is nonnegative and the red region where $p_a$ is negative at some $t$. We compute $p_a(t,\beta)$ for 20,000 pairs of $(a,\beta)$ uniformly randomly chosen in the region $ [0,10]\times [0,3]$. However, due to the limitations in CPU performance and numerical precision, we can only test $p_a(t,\beta)>-\varepsilon$ for $t\in [0,4]$ with the tolerance to be $\varepsilon:=0.01$. The resulting distribution is shown in Fig. \ref{fig:monte-carlo}. 

Numerically, we run the simulation on the pairs of $(a, \beta)$, which range from $(0.08, 10)$ as grid points. The simulation is conducted between $t\in (0, 400)$, and negativity is defined by the ratio between the maximum and minimum value in the inverse Laplace transform: if the minimum value is negative and the ratio, $|\min(p_a(t, \beta))|/|\max(p_a(t, \beta))|>10^{-5}$ holds, then there exists negative solution.

\begin{figure}[h!]
    \centering
    \includegraphics[width = 0.9\textwidth]{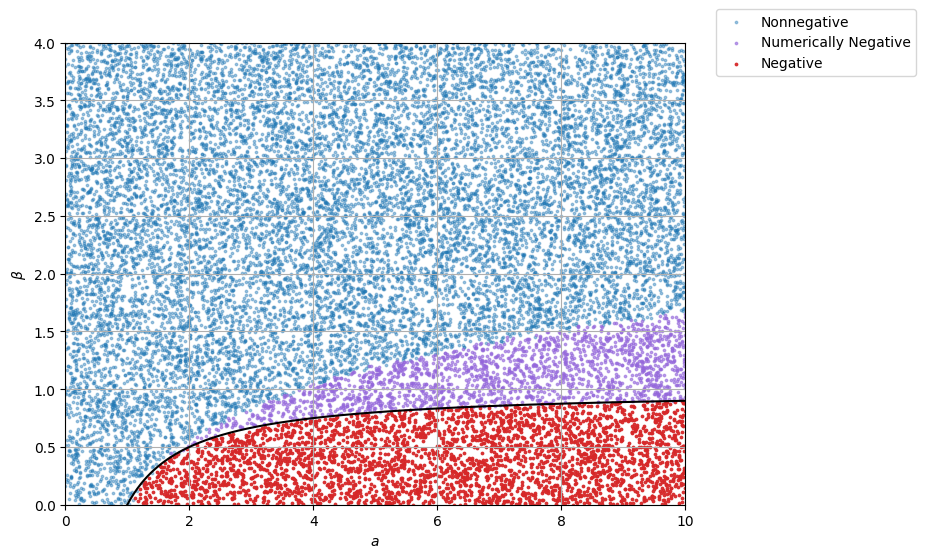}\caption{Monte-Carlo simulations on solving System \eqref{eq: system considered u} -- \eqref{def:Psi new} with various values of $(\alpha,\beta)$. Both purple and red points represent the existence of negative solutions: the red points are aligned with theoretical analysis in Section \ref{Sec:Pos} and the purple points are newly found via numerical simulations. The blue points represent the nonnegative solutions.}
    \label{fig:monte-carlo}
\end{figure}
The critical black curve in Fig. \ref{fig:monte-carlo} is  $a\beta-a+1=0$ as in Section \ref{Sec:Pos}. The purple points under black curve are falsely positive, likely due to numerical imprecision and the limited time window $t \in [0,4]$ used in the simulation. It is possible that $p_a(t,\beta)$ is already negative in theory but the value is too close to zero to comparing to the tolerance, or that it becomes negative for larger values of $t$ beyond the time window.

The implementation is carried out on CPU \texttt{AMD EPYC 9R45} using \texttt{Python 3.9} for data science with \texttt{NumPy} and \texttt{SciPy}.

\section{Conclusion}\label{Sec_conclusions}
This paper investigates a time-evolving biological system of compound-exchange cells modeled by a nonlocal PDE with negative feedback. The study is motivated by the two-dimensional models in \cite{Peng2023} and \cite{Yang2025}, which we reduce to a one-dimensional setting for analytical tractability. We focus on the positivity and steady-state behavior of the system. A novel aspect of our approach lies in the use of Laplace transform techniques, which convert the temporal evolution problem into an algebraic and complex-analytic one. Numerical evidence further confirms and supports our analytical results.

Some numerical simulation (see \secref{Sec_numerical_results}) show that for certain pairs $(a, \beta)$ lying just above the critical curve in Fig. \ref{fig:possible-reject}, the function $\tilde p_a(t, \beta)$ can still be negative at some $t$. The schematic Fig. \ref{fig:certify-numerical} shows these adjustments based on the theoretical results \thmref{thm:upluspos} and the numerical simulations. It also reveals that the critical curve is not sharp, very possibly because the curve is obtained by considering only the first derivative of $P_a(t,\beta)$ and there are many more higher derivatives thereof which contain more refined and delicate information.

\begin{figure}[H]
    \centering
    \includegraphics[scale=0.45]{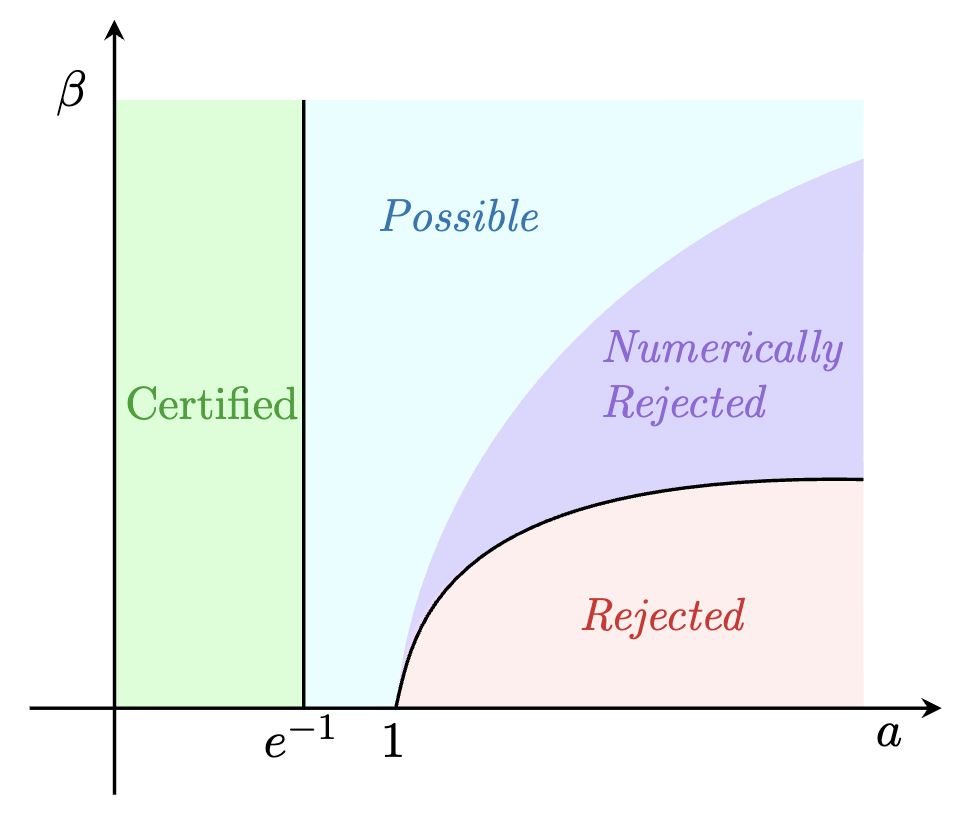}
    \caption{Adjusted schematic diagram with the theoretically certified region for positivity and the numerically rejected region}
    \label{fig:certify-numerical}
\end{figure}

At the heart of our analysis are the functions $p_a(t, \beta)$ and $\tilde p_a(t,\beta)$, which encapsulate key dynamical properties of the system. Through their behavior, we derive conditions on the feedback coefficient $a$, the secretion rate $\Phi$, and the initial condition $u_0$ that ensure positivity and the existence of steady states.

The two-dimensional models in \cite{Peng2023} and \cite{Yang2025} are considerably more complex, particularly due to the difficulty of applying Laplace transform techniques in higher dimensions. We leave a detailed exploration of those systems to future work.

\newpage

\renewcommand{\theHsection}{A\arabic{section}}

\appendix
\noindent{\huge\textbf{\hypertarget{app}{Appendix}}}

\section{Essentials of Laplace Transform Theory}\label{app:A}
In this Appendix, we collect and prove several relevant results about Laplace transform required in \secref{Sec:Pos} and \secref{Sec:Laplace}. See \secref{Sec:Laplace} for the basic definitions and properties.

\begin{defn}\label{defn:loclp}
Let $f:\IR_0^+\to\IR$. We say $f$ is \textit{locally}-$L^p$ for $p\in[1,\infty]$, namely $f\in L^p_{\mathrm{loc}}(\IR_0^+)$, if $f\in L^p([t_0,t_1])$ for any compact subinterval  $[t_0,t_1]\subseteq\IR_0^+$. We say $f\in L^p_{\mathrm{loc}}(\IR_0^+)$ is of \textit{$($eventually$)$ exponential order} $\omega\in\IR$ if there exists $t_0\geqslant0$ and zero-measure $U\subseteq[t_0,+\infty)$ such that $\sup_{t\in[t_0,+\infty)\smallsetminus U}|e^{-\omega t}f(t)|<+\infty$, namely, $\esssup_{t\geqslant t_0}|e^{-\omega t}f(t)|<+\infty$.
\end{defn}
Note that our definition of functions of (eventually) exponential order in \defnref{defn:loclp} is different from the standard one in literature, say Section 1.4 in \cite{Arendt-Batty}. E.g., $\frac{1}{\sqrt{t}}$ whose Laplace transform exists $\sqrt{\frac{\pi}{s}}$; see Table \cosref{tab:laplace}{A.1} below, but it is not square-integrable along any vertical line $\{\sigma+yi\mid y\in\IR\}$ with $\sigma>0$. However, this kind of functions is actually important for us. 

\begin{exa}
Extend any function in $L^1_\mathrm{loc}(\IR_0^+)$ by $0$ for $t<0$. Let $\theta$ be the \textit{Heaviside step function}
$$\theta(t)=\begin{cases}
    1, & t\geqslant0,\\
    0, & t<0.
\end{cases}$$
Some basic Laplace transform pairs are mentioned in the following table. More details can be found in \cite{oberhettinger-laplace}. 

\begin{center}\hypertarget{tab:laplace}{}
    \begin{tabular}{l|c}
    
        \multicolumn{1}{c|}{$f(t)$} & $F(s) = \mathcal{L}\{f\}$ \\ \hline\xrowht{15pt}
        $af_1(t)+bf_2(t)$ & $aF_1(s)+bF_2(s)$\\ \xrowht{15pt}
        $f(ct)$ & $\frac{1}{c} F \left( \frac{s}{c} \right)$\\ \xrowht{15pt}
        $t^p\theta(t)$, $p > -1$ & $\frac{\Gamma(p+1)}{s^{p+1}}$\\ \xrowht{15pt}
        $f(t-t_0)$ & $e^{-t_0s} F(s)$ \\ \xrowht{15pt}
        $ (f\ast g)(t)$ & $F(s) G(s)$\\ \xrowht{15pt}
        $f'(t)$ & $sF(s)-f(0)$
    \end{tabular}
    \vspace{1.5em}
    
    \footnotesize\textbf{\textit{Table A.1: Table of Laplace Transforms}}
\end{center}
\end{exa}

\begin{thm}\label{thm:lapexist}
Let $f\in L^1_{\mathrm{loc}}(\IR_0^+)$. If $\cL\{f\}$ exists for some $s_0\in\IC$, then $\cL\{f\}$ exists for all $s$ with $\Re(s)>\Re(s_0)$. In this case, $F:\{s\in\IC\mid\Re(s)>\Re(s_0)\}\to\IC$ is holomorphic. Assume moreover $f$ is of exponential order $\omega$. Then $\cL\{f\}$ exists for all $s$ with $\Re(s)>\omega$. 
\end{thm}
\begin{proof}
The proof can be found in Section 1.4 and Section 1.5 in \cite{Arendt-Batty}. Although their theorem applies to the broader generality of functions valued in certain Banach spaces, note that our definition of exponential order $\omega$ is weaker as we only require the \textit{essential limit superior} to be finite not just (\textit{essential}) \textit{supremum}. However, this will not be a problem. Take a big $T$ such that $|fe^{-st}|$ is essentially bounded on $[T,+\infty)$. 
$\left|\int_0^Tfe^{-st}\df t\right|\leqslant C\int_0^T|f|\df t$ as $e^{-st}$ is bounded on $[0,T]$. The rest becomes the same as in \cite{Arendt-Batty} when $t>T$. Later, we shall see that this is crucial to modify this definition.
\end{proof}

\begin{thm}[Post-Widder Inversion Formula]\label{thm:postwidder}
Let $f\in L^1_\loc(\IR_0^+)$. Assume that $\mathrm{abs}(f)<+\infty$. Let $F=\mathcal{L}\{f\}$. Then for every Lebesgue point $t>0$ of $f$, 
\begin{equation}\label{eq:Post-Widder formula}
    f(t)=\lim _{n\apc+\infty }{\frac {(-1)^{n}}{n!}}\left({\frac {n}{t}}\right)^{n+1}F^{(n)}\left({\frac {n}{t}}\right).
\end{equation}
\end{thm}
\begin{proof}
    See \cite{Arendt-Batty}, Theorem 1.7.7, p.43.
\end{proof}
If, for example, $f$ is a Lipschitz function with $f(0)=0$, then \eqref{eq:Post-Widder formula} holds for every $t>0$ (see \cite{Arendt-Batty}, Theorem 2.3.1, p.75).

\begin{cor}[Bromwich integral formula]\label{cor:bromwich}
Let $F:\{s\in\IC\mid\Re(s)>
\omega\}\to\IC$ be a function satisfying the properties in \thmref{thm:lapinvexist}. Then for any $\sigma>\omega$,
$$\cL^{-1}\{F\}(t)=\frac{1}{2\pi i}\int_{\sigma-\infty i}^{\sigma+\infty i}e^{st}f(s)\df s.$$
\end{cor}
\begin{proof}
See Theorem 2.3.4 of \cite{Arendt-Batty}.
\end{proof}

We conclude this Section with a version of the \textit{final value theorem} that characterizes the long-term behavior of $f(t)$ in terms of its Laplace transform. In the literature, proofs and even statements thereof are often incomplete, or additional assumptions on the derivative $f'(t)$ are imposed. Here, we provide a self-contained proof tailored to our conditions that avoids reliance on derivatives. A warning first: 
\begin{quote}
    {\it One really needs to assume a priori that {\bf $f(t)$ has a limit when $t\apc+\infty$},\\
    otherwise one may get erroneous results.}
\end{quote}

\begin{thm}[Final value theorem]\label{thm:fvt}
Let $f\in L^1_{\mathrm{loc}}(\IR_0^+)$ such that $\lim_{t\apc+\infty}f(t)$ exists. Then 
$$\lim_{t\apc+\infty}f(t)=\lim_{s\apc 0^+}sF(s).$$
\end{thm}
\begin{proof}
By \thmref{thm:lapexist}, $\cL\{f\}$ exists for all $s>0$. Let $L:=\lim_{t\apc+\infty}f(t)$. By Table \cosref{tab:laplace}{A.1}, $\theta(t)$ is transformed to $\frac{1}{s}$, namely $s\int_0^{+\infty}e^{-st}\df t=1$. Then
$$sF(s)-L=s\int_0^{+\infty}(f(t)-L)e^{-st}\df t,$$
which implies that
$$|sF(s)-L|\leqslant s\int_0^{+\infty}|f(t)-L|e^{-st}\df t.$$
For any $\varepsilon>0$, choose $T$ such $|f(x)-L|<\varepsilon$. Then
$$s\int_0^{+\infty}|f(t)-L|e^{-st}\df t\leqslant s\int_0^{T}|f(t)-L|e^{-st}\df t+s\int_T^{+\infty}|f(t)-L|e^{-st}\df t.$$
Because $e^{-st}\leqslant1$ on $\IR_0^+$, the first term is bounded by $s(TL+B)$ where $B=\int_0^{+\infty}|f(t)|\df t$ exists because $f\in L^1_{\mathrm{loc}}(\IR_0^+)$. The second term is bounded by $s\int_0^{+\infty}\varepsilon e^{-st}\df t=\varepsilon$. Now we can choose a small $s$ such that the sum of the two terms is bounded by $2\varepsilon$.
\end{proof}

\section{Analysis of a Delayed Differential Equation}\label{app:B}

The following linear delayed differential equation (DDE) with nonconstant coefficients \begin{equation}\label{eq:dde}
\left\{\begin{alignedat}{2}
y'(t) &= A(t)y(t-\tau),\quad & & t\in(t_0,+\infty), 
\\
y(t) &= h(t), & &t\in [t_0-\tau,t_0],
\end{alignedat}\right.
\end{equation}
where $h(t)$ provides the history condition for $t\in [t_0-\tau, t_0]$, has been extensively studied (see e.g., \cite{gyori1991oscillation, DIBLIK1998200}, and various references in the latter). It, or the even more general version with multiple and variable delays (see \eqref{eq:DDE with variable delay and coefficients}), occurs as an essential object of study in various questions, like it does (in much simplified form) in our question, about positivity of solutions of the System \eqref{eq: system considered u}--\eqref{def:Psi new}. Properties of the solution (like positivity or being oscillatory, i.e., having zeros at arbitrarily large times) then characterize properties of solutions to the original question.

For constant delay and continuous coefficient function $A(t)$ and history $h(t)$, the existence of a unique solution $y:[t_0-\tau,+\infty)\apc\IR$ is immediate by means of integration by steps. By shifting time when necessary and redefining coefficient and history functions accordingly, we may assume without loss of generality that $t_0=\tau$. For the convenience of conducting analysis, we non-dimensionalize Equation \eqref{eq:dde} by rescaling as follows: 
\begin{equation}\label{eq:rescaling DDE}
    t^* = \frac{t}{\tau}, \quad y^*(t^*) = \frac{y(t)}{\hat{y}}, \quad A^*(t^*) = \tau A(t), \quad h^*(t^*) = \frac{h(t)}{\hat{y}}.
\end{equation}
Then the dimensionless equation is given by
\begin{equation*}
    \left\{
        \begin{alignedat}{2}
            \frac{\df y^*}{\df t^*} &= A^*(t^*)y^*(t^*-1), \quad& &t^*\in(1, +\infty),\\
            y^*(t^*) &= h^*(t^*), & & t^*\in[0,1].
        \end{alignedat}
    \right.
\end{equation*}

For our question of interest, we need to obtain those $A\in\IR$ for which the solution to the specific version of Equation \eqref{eq:dde} with $\tau=1$, $t_0=1$,
\begin{equation}
\label{eq:dimensionless dde}
        \left\{
        \begin{alignedat}{2}
            y'(t) &= Ay(t-1), \quad &&t\in(1, +\infty),\\
            y(t) &= 1, & & t\in[0,1].
        \end{alignedat}
    \right.
\end{equation}
is positive on $\IR^+$. We denote this solution by $\eta_A(t)$. Note that in view of the rescaling \eqref{eq:rescaling DDE}, the case of Equation \eqref{eq:dde} with constant coefficient and  history function is then covered by \eqref{eq:dimensionless dde}.

\begin{lem}\label{lem:bound y}
    Suppose that the coefficient function $A(t)$ and the history function $h(t)$ are continuous. If $A(t)$ is bounded on $[t_0-\tau,+\infty)$, then the solution $y(t)$ to \eqref{eq:dde} satisfies
    \begin{equation}\label{eq:exp bound y}
        |y(t)|\leqslant Me^{\omega (t-t_0+\tau)}\quad \mbox{for all } t\geqslant t_0-\tau,
    \end{equation}
    where $M:= \sup_{t\in [t_0-\tau,t_0]} |h(t)|$ and $\omega := \sup_{t\geqslant t_0-\tau} |A(t)|$.
\end{lem}
\begin{proof}
    Put $t_{k}:=t_0+k\tau$ for $k\in\IZ$. We shall prove the statement by induction. For $t\in[t_{-1},t_0]$, by construction $|y(t)|=|h(t)|\leqslant M\leqslant Me^{\omega(t-t_{-1})}$.
    Suppose that \eqref{eq:exp bound y} has been proven for all $t\in [t_{k-1},t_k]$, for all $0\leqslant k\leqslant n$. Let $t\in [t_n,t_{n+1}]$. Then
    \begin{align*}
        |y(t)| &\leqslant |y(t_n)| + \bigl| \int_{t_n}^t y'(\xi)\df \xi\bigl|  
        \leqslant Me^{\omega(n+1)\tau} + \omega\int_{t_n}^t \bigl| y(\xi-\tau)\bigr|\df\xi\\
        &\leqslant Me^{\omega(n+1)\tau} + \omega\int_{t_n}^t Me^{\omega(\xi-\tau-t_{-1})}\df\xi\\
        &\leqslant Me^{\omega(n+1)\tau} + M e^{-\omega\tau} \bigl( e^{\omega(t-t_{-1})} - e^{\omega(n+1)\tau} \bigr)\\
        &\leqslant Me^{\omega(n+1)\tau} \bigl(1 - e^{-\omega\tau}\bigr) + Me^{\omega(t-\tau-t_{-1})}\\
        &\leqslant Me^{\omega(t-t_{-1})}\bigl(1-e^{-\omega\tau} + e^{-\omega\tau} \bigr) \\
        &=Me^{\omega(t-t_{-1})}. 
    \end{align*}
Hence, the result follows.
\end{proof}

It is clear from integration by steps, that $\eta_A>0$ for $A\geqslant 0$. The case $A<0$ requires a more subtle analysis. Because the coefficient function is constant, we can apply Laplace transform to gain insight into the positivity of $\eta_A$.

The common approach to prove positivity of the solution to a specific system of ordinary differential equations or partial differential equations is to exhibit the solution as a limit of an iteration procedure (like Picard's Iteration) and prove that positivity is preserved in each step of this procedure (see e.g., \cite{Sikic:1994,Canizo_ea:2012,Hille_ea:2025}). This strategy did not work well for the DDE \eqref{eq:dimensionless dde}. Instead we employ a positivity result by Dibl\`ik, \lemref{prop:ddeincrease}.

The solution $\eta_A(t)$ to \eqref{eq:dimensionless dde} may be obtained through integration by steps, yielding
\begin{equation}
\label{eq:etaA}
    \eta_A(t) = \sum_{m=0}^\infty A^m \frac{(t-m)^m}{m!} \theta(t-m).
\end{equation}
This expression may be obtained through Laplace transform as well, because the coefficient function, $A(t)=A$, is constant.

Integration by steps yields that the solution $y(t)$ will be piecewise-$C^1$. Lemma \ref{lem:bound y} implies that $y$ is exponentially bounded: 
\begin{equation}\label{eq:etaexpord}
    \mbox{For any } A\in\IR,\qquad |\eta_A(t)|\leqslant e^{|A|t}\qquad \mbox{for all } t\geqslant0.
\end{equation}
Thus, $y$ and $y'$ have Laplace transforms. Put $Y:=\cL\{y\}$. One can also apply the Laplace transform to \eqref{eq:dimensionless dde}. The Uniqueness Theorem (Theorem \ref{thm:lapuniq}) yields 
\begin{equation}\label{eq:lapdde}
sY(s)-y(0)=Ae^{-s}Y(s),
\end{equation}
where $y(0)=1$. Then $Y(s)$ can be expressed as 
\begin{equation}\label{eq:lapinvdde}
Y(s)=\frac{1}{s-Ae^{-s}} = \frac{\frac{1}{s}}{1-\frac{A}{s}e^{-s}}.
\end{equation}
Hence, for $s\in\mathbb{C}$ such that $\displaystyle\left|\frac{A}{s}e^{-s}\right|<1$, $Y(s)$ can be expressed as a geometric series:
\begin{equation}
    Y(s) = \lim_{n\apc+\infty} Y_n(s) := \lim_{n\apc+\infty} \sum_{m=0}^n \frac{A^me^{-ms}}{s^{m+1}}.
\end{equation}
In particular, the series converges for $s>|A|$. Define
\[
y_n(t) := \sum_{m=0}^n \frac{A^m(t-m)^m}{m!}\theta(t-m).
\]
Then $y_n$ is exponentially bounded of the same type for all $n$:
\begin{equation}
    |y_n(t)| \leqslant \sum_{m=0}^n |A|^m \frac{(t-m)^m}{m!}\theta(t-m) \leqslant \sum_{m=0}^n \frac{|A|^m}{m!} t^m  \leqslant e^{|A|t} \quad\mbox{for all } t\geqslant 0.
\end{equation}
Hence, $y_n$ has Laplace transform (see Table \cosref{tab:laplace}{A.1})
\[
    \cL\{y_n\}(s) = \sum_{m=0}^n A^m \frac{e^{-ms}}{s^{m+1}} = Y_n(s).
\]
Moreover, $y_n$ converges uniformly on compact subsets of $\IR_0^+$ to $\eta_A$ in \eqref{eq:etaA}. Because of the uniform exponential order bound for the $y_n$, $f$ satisfies the same growth bound. Moreover, because of the pointwise convergence of the geometric series for $s>|A|$, \cite{Arendt-Batty} Theorem 1.7.5 (Approximation) now implies that $\cL\{f\}(s) = \lim_{n\apc+\infty} Y_n(s)= Y(s)$. Thus, $y(t) = \eta_A(t)$.

\subsection{Positivity of Solution to System \eqref{eq:dimensionless dde}}\label{Sec_Positivity}
We are now concerned with finding a condition on $A$ that ensures the positivity of $\eta_A$ on $\IR_0^+$. 
Our approach works for the particular system \eqref{eq:dimensionless dde}, and not all DDE can be studied through Laplace transform. Furthermore, even though \eqref{eq:lapinvdde} is derived from the first equation in System \eqref{eq:dimensionless dde} which is defined for $t\in(1,+\infty)$, the inverse Laplace transform of \eqref{eq:lapinvdde}, expressed in \eqref{eq:etaA}, is the solution to System \eqref{eq:dimensionless dde} for the entire $\IR_0^+$, i.e., \eqref{eq:etaA} satisfies the history term as well, even though it is a computation result of the DDE in $(1, +\infty)$. 

The behaviour of $\eta_A(t)$ can be become oscillatory. For $A\geqslant 0$ this cannot happen; see Fig. \ref{fig:tildepapos}. When $A<0$, determining the positivity of the solution becomes more complicated. Roughly speaking, the solution can be categorised as follows: when $|A|$ is sufficiently small, $y(t)$ is exponentially decaying but will stay positive; With $|A|$ is slightly large, $y(t)$ behaves like a damped oscillation; When $|A|$ is sufficiently large, $y(t)$ behaves like an unstable oscillation; see Fig. \ref{fig:tildepanega} for the numerical examples of various values of $A$. The theoretical thresholds to categorize the above behaviours of $y(t)$ requires further study of spectra of System \eqref{eq:dimensionless dde}.

\begin{figure}[h!]
    \centering
\includegraphics[width = 0.85\textwidth]{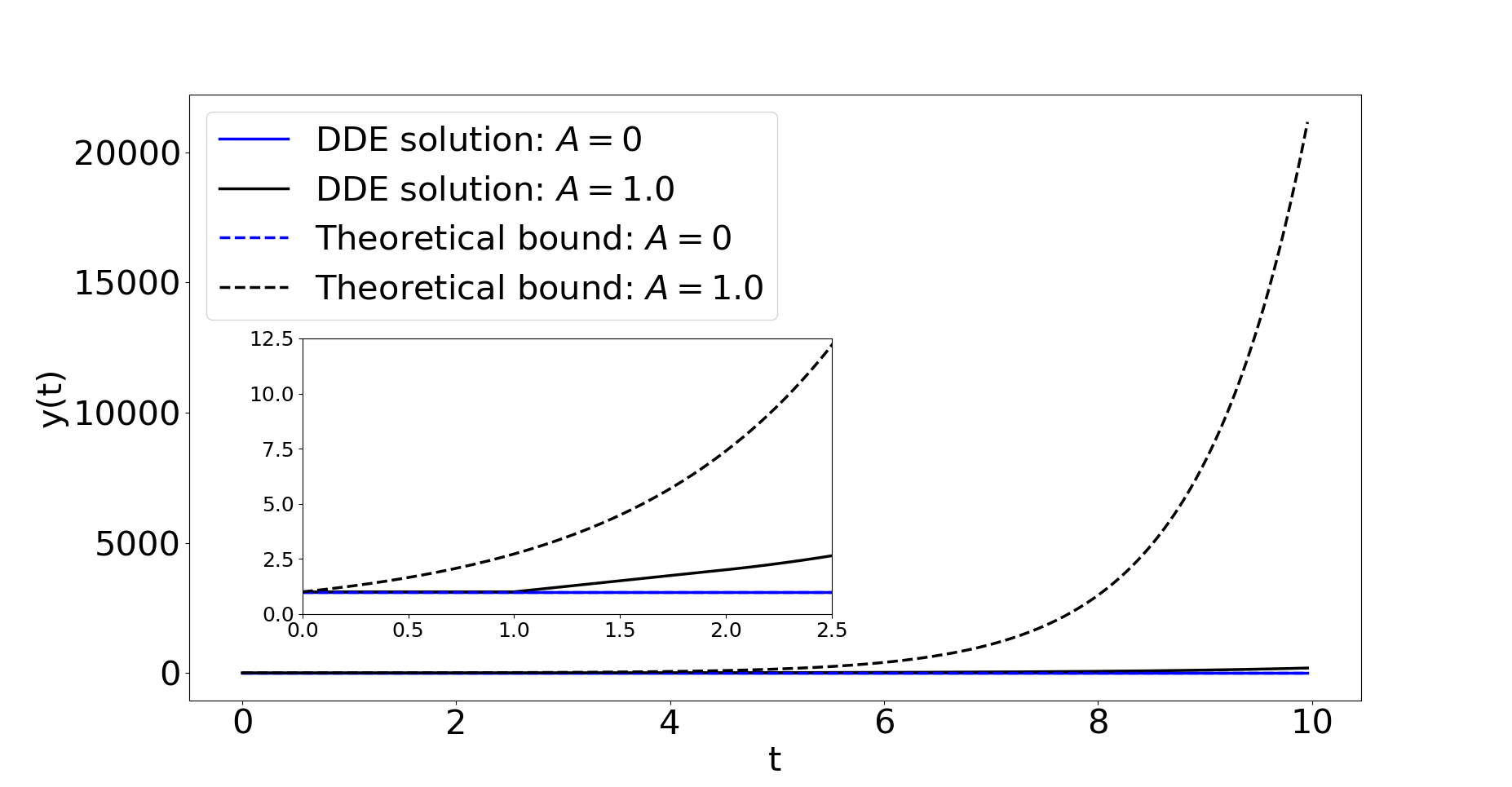}
    \caption{For $A\geqslant0$, the solution $y(t)$ to System \eqref{eq:dimensionless dde} stays positive and upper bounded by $e^{At}$. Here, the numerical solution $y(t)$ using forward Euler method (solid curves) and the corresponding upper bounds (dashed curves) are shown for $A=0$ (blue curves) and $A = 1.0$ (black curves), respectively.}
    \label{fig:tildepapos}
\end{figure}

\begin{figure}[h!]
    \centering
\includegraphics[width = 0.85\textwidth]{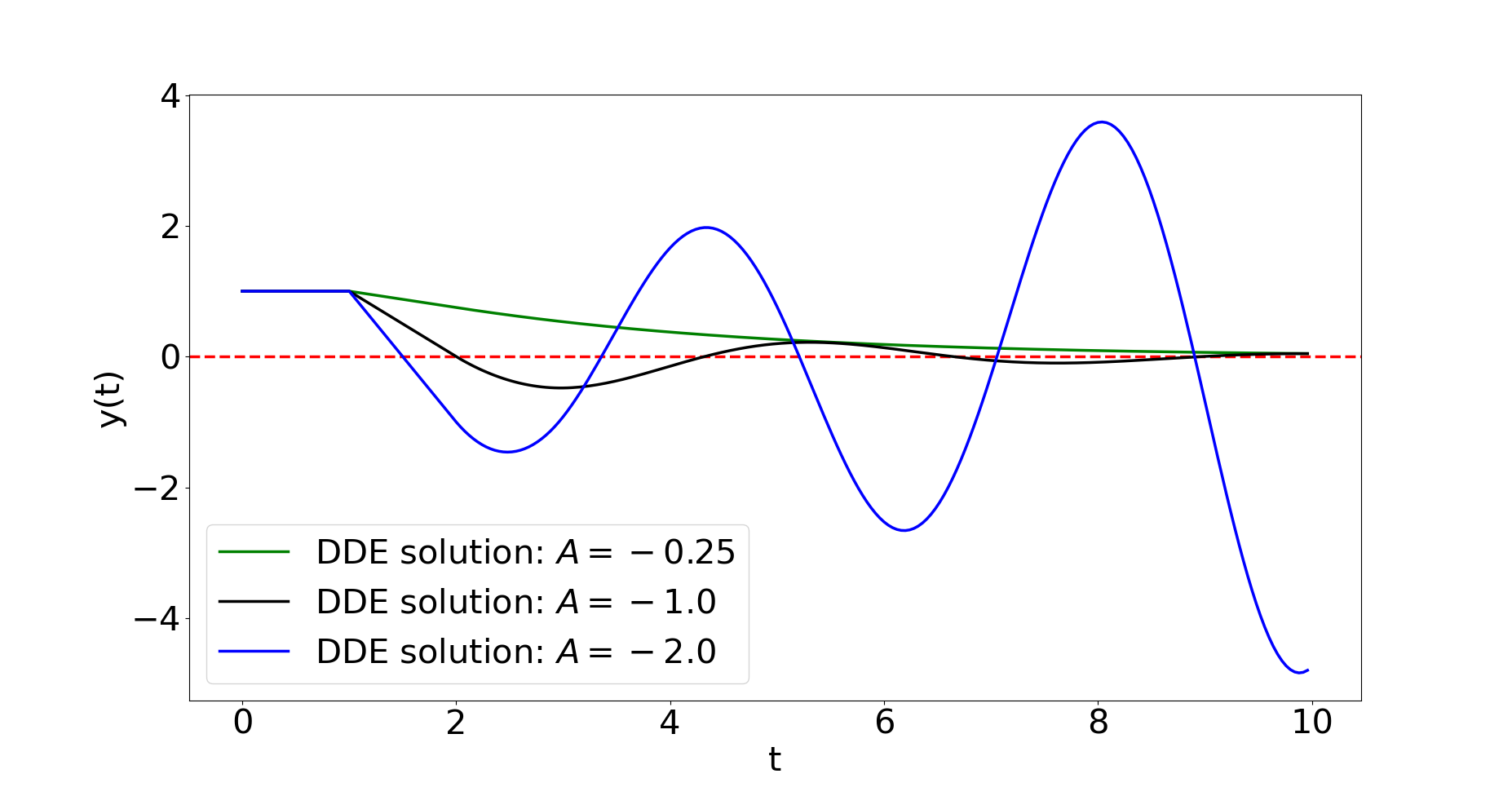}
    \caption{When $A\leqslant0$, the solution $y(t)$ to System \eqref{eq:dimensionless dde} behaves differently depending on the value of $|A|$. Here, we plot the solution for $A=-0.25$ (green), $A=-1.0$ (black) and $A=-2.0$ (blue), respectively.}
    \label{fig:tildepanega}
\end{figure}

\begin{lem}\label{lem:zerobrange}
    The spectral problem of System \eqref{eq:dimensionless dde}, given by $\lambda-Ae^{-\lambda}=0$, has two negative real roots when $A\in[-e^{-1},0)$: the larger root is in $[-1,A)$, and the smaller root is in $(A^{-1}, \ln(-A)]$.
\end{lem}
\begin{proof}

Denote the spectral equation as function $$f_A(\lambda)=\lambda-Ae^{-\lambda},$$
with $A\in[-e^{-1},0)$. Take its derivative
$$f_A'(\lambda)=1+Ae^{-\lambda}$$
which has a unique root $\lambda=\ln(-A)$. Thus, $f_A(\lambda)$ increases in $\lambda\in(\ln(-A),+\infty)$ and decreases in $\lambda\in(-\infty,\ln(-A))$. Notice that for $A$ in the given range, $A^{-1}<\ln(-A)\leqslant-1<A$. Denote these four points from left to right on the real axis by $\lambda_1,\lambda_2,\lambda_3,\lambda_4$ respectively. To obtain the desired statement, we just need to show that $f_A(\lambda_1)>0$, $f_A(\lambda_2)\leqslant0$, $f_A(\lambda_3)\leqslant0$, and $f_A(\lambda_4)>0$. 

We can determine the signs of $f_A$ at last three points easily:
$f_A(\lambda_2)=\ln(-A)+1\leqslant0$, $f_A(x_3)=-1-Ae^{-1}\leqslant0$,
$f_A(\lambda_4)=A(1-e^{-A})>0$. It remains to show that 
$
f_A(\lambda_1)=\frac{1}{A}-A e^{-\frac{1}{A}} > 0$. Set $\tilde A:=-\frac{1}{A}\in [e,+\infty)$. Then
$$f_A(\lambda_1)>0 \;\;\Leftrightarrow\;\;-\tilde A + \frac{e^{\tilde A}}{\tilde A} > 0\;\;\Leftrightarrow\;\;e^{\tilde A} > \tilde A^2.$$
Since
$$e^{\tilde A} + e^{-\tilde A} = 2\cosh(\tilde A)= 2 + \tilde A^2 + \text{(nonnegative h.o.t.)} \geqslant 2 + \tilde A^2,$$
we obtain
$$e^{\tilde A} \geqslant \tilde A^2 + (2 - e^{-\tilde A}) > \tilde A^2,$$
as $0<e^{-\tilde A}<1$ for $\tilde A>0$. Hence, the claim.
\end{proof}

The existence of the positive solution for the differential equation in System \eqref{eq:dimensionless dde} can be verified by the lemma below:
\begin{lem}\label{lem:ddepos}
When $A\in[-e^{-1},0)$,
there exists a strictly positive function satisfying \begin{equation}\label{eq:ddenocondition}
    y'(t)=Ay(t-1).
\end{equation}
\end{lem}
\begin{proof}
The proof can be analogously concluded from Lemma \ref{lem:zerobrange} due to the negativity of the eigenvalues of the corresponding spectral problem.
\end{proof}

Let $w(t)$ be any positive solution to Equation \eqref{eq:ddenocondition}. Define the variable transformation 
\[
    v(t):=\frac{y(t)}{w(t)},
\]
then we obtain 
\begin{align*}
v'(t)&=\frac{y'(t)w(t)-y(t)w'(t)}{w^2(t)} \\
&=\frac{y'(t)}{w(t)}-\frac{y(t)w'(t)}{w^2(t)}\\
&=\frac{Ay(t-1)}{w(t)}-\frac{Ay(t)w(t-1)}{w^2(t)}\\&
=\frac{-Aw(t-1)}{w(t)}\left(\frac{y(t)}{w(t)}-\frac{y(t-1)}{w(t-1)}\right)\\
&=\frac{-Aw(t-1)}{w(t)}(v(t)-v(t-1))\\
&:= \widetilde{A}(t)(v(t) - v(t-1)). 
\end{align*}
Hence, the new parameter $v(t)$ satisfies the following differential equation:
\begin{equation}\label{eq:vchavar}
v'(t)=\widetilde{A}(t)(v(t)-v(t-1)), \quad t>1.
\end{equation}
This transformation is important to prove the positivity of $y(t)$ in System \eqref{eq:dimensionless dde}. Prior to that, we investigate further of the properties of $v(t)$ defined in Equation \eqref{eq:vchavar} via Lemma 3 of \citet{DIBLIK1998200}.

\begin{lem}[\citep{DIBLIK1998200}, Lemma 3]\label{prop:ddeincrease}
Let $v(t)$ satisfy Equation \eqref{eq:vchavar} on $t\in (1,+\infty)$ and assume that the history condition $h(t)$ for $t\in[0,1]$ is continuous. If $\widetilde{A}(t)>0$ and $h(1)>h(t)$ on $t\in[0,1)$, then $v(t)$ is increasing on $[1,+\infty)$. 
\end{lem}

The key result in this appendix, i.e., the positivity of the solutions to System \eqref{eq:dimensionless dde} is stated below. 
\begin{prop}\label{thm:ypos}
For $A\in[-e^{-1},0)$,
the solution $y(t)$ to System \eqref{eq:dimensionless dde} is positive, and decreasing exponentially on $\IR_0^+$. 
\end{prop}
\begin{proof}
The proof we shall give now avoids the use of Theorem 3.3.1 of \citet{gyori1991oscillation}. Let $w(t) = e^{\lambda t}$ be the positive solution to Equation \eqref{eq:vchavar}, where $\lambda$ is the eigenvalues computed in Lemma \ref{lem:zerobrange}, then 
\[
    \widetilde{A} = \widetilde{A}(t) = -Ae^{-\lambda}>0,\qquad -1\leqslant\lambda< A <0.
\]
Hence, with $y(t)$ defined by System \eqref{eq:dimensionless dde}, $v(t)$ is then given by
\begin{equation*}
    \left\{
    \begin{alignedat}{2}
        v'(t)&=-Ae^{-\lambda}(v(t) - v(t-1)), \quad &&t\in(1, +\infty),\\
        v(t) &= e^{-\lambda t}, & &t\in[0,1].
    \end{alignedat}
    \right.
\end{equation*}
Notice $y(t)=1$ on $[0,1]$, so the history condition $h$ for $v$ on $[0,1]$ satisfies $h(1)=v(1)=e^{-\lambda}>e^{-\lambda t}=v(t)=h(t)$ for all $t\in[0,1)$. Then apply \lemref{prop:ddeincrease} to see that $v(t)$ is increasing on $[1,+\infty)$. Namely, $\frac{y(t)}{e^{\lambda t}}=v(t)\geqslant v(1)=e^{-\lambda}$. That is to say $y(t)\geqslant e^{\lambda (t-1)}>0$. This also shows that $y(t)$ is strictly decreasing on $[0,1)$ as $y'(t)=By(t-1)<0$, and hence, $y(t)$ is decreasing on $\IR_0^+$.
\end{proof}

\section{Miscellaneous Integrals}\label{app:C}
\begin{lem} For all $n\in\IN$,
\begin{equation}
\label{eq:gaussint}
\int_{\IR}x^{2n}e^{-ax^2}\df x =\frac{(2n-1)!!\sqrt{\pi}}{2^{n}\sqrt{a^{1+2n}}}.
\end{equation}
\end{lem}
\begin{proof}
By the well-known Gaussian integral
$$
\int_{\IR}e^{-ax^2}\df x=\sqrt{\frac{\pi}{a}}.
$$
Take partial derivative with respect to $a$ on both side,
\begin{align*}
    \frac{\partial^n}{\partial a^n}\left(\int_{\IR}e^{-ax^2}\df x\right)&=\frac{\partial^n}{\partial a^n}\left(\sqrt{\frac{\pi}{a}}\right)\\
    \int_{\IR}\frac{\partial^ne^{-ax^2}}{\partial a^n}\df x&=\frac{\partial^n}{\partial a^n}\left(\sqrt{\pi}a^{-\frac{1}{2}}\right)\\
    \int_{\IR}(-1)^nx^{2n}e^{-ax^2}\df x&=\left(-\frac{2n-1}{2}\right)\left(-\frac{2n-3}{2}\right)\cdots\left(-\frac{1}{2}\right)\sqrt{\frac{\pi}{a^{1+2n}}}.
\end{align*}
Note that $(2n-1)!!:=(2n-1)(2n-3)\cdots1$ by definition and $(2\cdot0-1)!!:=1$ by convention. We obtain the result.
\end{proof}

\begin{lem}\label{lem:lapkerncomputation}
Let $\beta\geqslant0$. The heat kernel $\cK_1(t,\beta)$ has Laplace transform
\begin{equation}\label{eq:lapheatkernone}
Q(s,\beta)=\frac{e^{-\sqrt{s}\beta}}{2\sqrt{s}}.
\end{equation}

The function $\cT(t,\beta):=\frac{\beta e^{-\frac{\beta^2}{4t}}}{\sqrt{4\pi t^3}}$ has Laplace transform
\begin{equation}\label{eq:lapsqrtkern}
T(s,\beta)=e^{-\sqrt{s}\beta}.
\end{equation}
\end{lem}

\begin{proof}
\begin{align*}
Q(s,\beta):=\,&\cL\{\cK_1(\cdot,\beta)\}(s)\\
=\,&\int_0^{+\infty} \frac{e^{-\frac{\beta^2}{4t}}}{\sqrt{4\pi t}}e^{-st}\df t\\
=\,&\frac{e^{-\sqrt{s}\beta}}{\sqrt{4\pi}}\int_0^{+\infty} \frac{e^{-\left(\sqrt{\frac{\beta^2}{4t}}-\sqrt{st}\right)^2}}{\sqrt{t}}\df t.
\end{align*}
Substitute $t=\frac{\beta^2}{4st'}$ and $\df t=-\frac{\beta^2}{4st'^2}\df t'$. Then
$$Q(s,\beta)=\frac{e^{-\sqrt{s}\beta}}{\sqrt{4\pi}}\int_0^{+\infty}\sqrt{\frac{\beta^2}{4st'^3}}e^{-\left(\sqrt{st'}-\sqrt{\frac{\beta^2}{4t'}}\right)^2}\df t'.$$
Add the two integral representations of $Q(s,\beta)$ up:
\begin{align*}
&2Q(s,\beta)=\frac{e^{-\sqrt{s}\beta}}{\sqrt{4\pi}}\int_0^{+\infty} \left(\sqrt{\frac{1}{t}}+\sqrt{\frac{\beta^2}{4st^3}} \right)e^{-\left(\sqrt{st}-\sqrt{\frac{\beta^2}{4t}}\right)^2}\df t\\
\Rightarrow\;\;&Q(s,\beta)=\frac{e^{-\sqrt{s}\beta}}{\sqrt{4\pi s}}\int_0^{+\infty}\frac{1}{2}\left(\sqrt{\frac{s}{t}}+\sqrt{\frac{\beta^2}{4t^3}}\right)e^{-\left(\sqrt{st}-\sqrt{\frac{\beta^2}{4t}}\right)^2}\df t
\end{align*}
where we divide the factor $2$ on the left and exact $\sqrt{s}$ on the denominator.
Substitute $v=\sqrt{st}-\sqrt{\frac{\beta^2}{4t}}$ which has range $\IR$ and $\df v=\frac{1}{2}\left(\sqrt{\frac{s}{t}}+\sqrt{\frac{\beta^2}{4t^3}}\right)\df t$. Using \eqref{eq:gaussint}, we eventually get
\begin{align*}
Q(s,\beta)&=\frac{e^{-\sqrt{s}\beta}}{\sqrt{4\pi s}}\int_0^{+\infty}e^{-v^2}\df t\\
&=\frac{e^{-\sqrt{s}\beta}}{2\sqrt{s}}.
\end{align*}

Moreover, we can compute the higher derivatives of Laplace transform of the heat kernel,
\begin{align*}
    Q^{(n)}(s,\beta)&=(-1)^n\int_0^{+\infty} t^n\frac{e^{-\frac{\beta^2}{4t}}}{\sqrt{4\pi t}}e^{-st}\df t\\
&=\frac{(-1)^n}{\sqrt{4\pi}}s^{-n-\frac{1}{2}}\int_0^{+\infty} t^{n-\frac{1}{2}}e^{-\frac{(\sqrt{s}\beta)^2}{4t}-t}\df t.
\end{align*}
Using Formula (15) in 6.22 of \cite{watson_bessel}, we can further express the result in terms of \textit{modified Bessel functions}:
\begin{equation}
    Q^{(n)}(s,\beta)=\frac{(-1)^n\beta^{n+\frac{1}{2}}}{\sqrt{2^{2n+5}\pi s^{n+\frac{1}{2}}}}K_{n-\frac{1}{2}}(\sqrt{s}\beta).
\end{equation}

For the second function, we compute
$$\int_0^{+\infty}\frac{\beta e^{-\frac{\beta^2}{4t}}}{\sqrt{4\pi t^3}}e^{-st}\df t=\frac{e^{-\sqrt{s}\beta}}{\sqrt{\pi}}\int_0^{+\infty}\frac{\beta e^{-\left(\sqrt{st}-\sqrt{\frac{\beta^2}{4t}}\right)^2}}{\sqrt{4t^3}}\df t.$$
Substitute $t=\frac{\beta^2}{4st'}$, add two integrals, divide by $2$, and substitute $v=\sqrt{st}-\sqrt{\frac{\beta^2}{4t}}$.
\end{proof}

\begin{lem}\label{lem:xnt2n01}
Let $f\in L^2(\IR)$, $n\in\IN$, and $p\geqslant0$. Then the integral
\begin{equation}\label{eq:xn01}
\int_0^1\int_{\IR}\frac{x^ne^{-\frac{x^2}{4t}}}{\sqrt{4\pi t^{2+n-p}}}f(x)\df x \df t
\end{equation}
converges absolutely.
\end{lem}

\begin{proof}
By H\"older's inequality and \eqref{eq:gaussint}, \eqref{eq:xn01} is
\begin{align*}
    &\leqslant\int_0^1\frac{1}{\sqrt{4\pi t^{2+n-p}}}\left(\frac{(2n-1)!\sqrt{\pi (2t)^{1+2n}}}{2^n}\right)^{\frac{1}{2}}\Vert f\Vert_{L^2} \df t\\
    &=\left(\frac{(2n-1)!!\sqrt{2^{1+2n}\pi}}{2^{n+2}\pi}\right)^{\frac{1}{2}}\Vert f\Vert_{L^2}\cdot\int_0^1t^{-\frac{3}{4}+\frac{p}{2}}\df t.,
\end{align*}
which is finite.
\end{proof}

\begin{lem}\label{lem:xnt3n1inft}
Let $f\in L^2(\IR)$, $n\in\IN$, and $p\geqslant0$. Then the integral
\begin{equation}\label{eq:xn1inf}
\int_1^{+\infty}\int_{\IR}\frac{x^ne^{-\frac{x^2}{4t}}}{\sqrt{4\pi t^{3+n+p}}}f(x)\df x \df t
\end{equation}
converges absolutely.
\end{lem}

\begin{proof}
By \eqref{eq:gaussint} and H\"older's inequality, \eqref{eq:xn1inf} is
\begin{align*}
    &\leqslant\int_1^{+\infty}\frac{1}{\sqrt{4\pi t^{3+n+p}}}\left(\frac{(2n-1)!!\sqrt{\pi (2t)^{1+2n}}}{2^{n}}\right)^{\frac{1}{2}}\Vert f\Vert_{L^2} \df t\\
    &=\left(\frac{(2n-1)!!\sqrt{2^{1+2n}\pi}}{2^{n+2}\pi}\right)^{\frac{1}{2}}\Vert f\Vert_{L^2}\cdot\int_1^{+\infty}t^{-\frac{5}{4}-\frac{p}{2}}\df t.
\end{align*}
which is finite.
\end{proof}

\begin{lem}\label{lem:pabound}
Let $a,\beta,C,p\geqslant0$, $\omega<0$ ,and $t>0$. Then
\begin{equation}\label{eq:paexpbound}
\int_\beta^{+\infty}\frac{\tau e^{-\frac{\tau^2}{4t}}}{\sqrt{4\pi t^p}}Ce^{a\tau}\df\tau\leqslant\frac{5Ce^{-\frac{\beta^2}{5t}+20a^2t}}{4\sqrt{\pi t^{p-2}}},
\end{equation}
\begin{equation}\label{eq:paexpboundnega1}
\int_\beta^{+\infty}\frac{\tau e^{-\frac{\tau^2}{4t}}}{\sqrt{4\pi t^p}}Ce^{\omega \tau}\df\tau\leqslant\frac{Ce^{-\frac{\beta^2}{4t}}}{\sqrt{\pi t^{p-2}}},
\end{equation}
\begin{equation}\label{eq:paexpboundnega0}
\int_\beta^{+\infty}\frac{ e^{-\frac{\tau^2}{4t}}}{\sqrt{4\pi t^p}}Ce^{\omega \tau}\df\tau\leqslant\frac{Ce^{-\frac{\beta}{2\sqrt{t}}}}{\sqrt{\pi t^{p-1}}},
\end{equation}
and
\begin{equation}\label{eq:paexpboundnega2}
\int_\beta^{+\infty}\frac{\tau^2 e^{-\frac{\tau^2}{4t}}}{\sqrt{4\pi t^p}}Ce^{\omega \tau}\df\tau\leqslant\frac{Ce^{-\frac{\beta}{2\sqrt{t}}}}{\sqrt{\pi t^{p-1}}}(\beta^2+4\sqrt{t}\beta+8t).
\end{equation}
\end{lem}
\begin{proof}
For \eqref{eq:paexpbound},
\begin{align*}
    \int_\beta^{+\infty}\frac{\tau e^{-\frac{\tau^2}{4t}}}{\sqrt{4\pi t^p}}Ce^{a\tau}\df\tau&\leqslant\int_\beta^{+\infty}\frac{\tau e^{-\frac{\tau^2}{5t}+20a^2t}}{\sqrt{4\pi t^{p}}}\df\tau\\
    &=\frac{5Ce^{-\frac{\beta^2}{5t}+20a^2t}}{4\sqrt{\pi t^{p-2}}},
\end{align*}
where the last step is because $e^{-\frac{\tau^2}{4t}+a\tau}\leqslant e^{-\frac{\tau^2}{5t}+20a^2t}$ for $\tau\geqslant0$ if and only if $-\frac{\tau^2}{4t}+a\tau\leqslant-\frac{\tau^2}{5t}+20a^2t$ for $\tau\geqslant0$. But this is trivial for $a\geqslant0$ and $t>0$ by looking at the two parabolas.

The estimate \eqref{eq:paexpboundnega1} follows from the same computation as in \eqref{eq:paexpbound}, upon noting that $Ce^{\omega\tau} \leqslant C$ under the given parameters.

For \eqref{eq:paexpboundnega0}, the estimation $-\frac{\tau^2}{4t}\leqslant-\frac{\tau}{2\sqrt{t}}+1$ shows that the integral is bounded above by
$$\int_\beta^{+\infty}\frac{Ce^{-\frac{\tau}{2\sqrt{t}}}}{\sqrt{4\pi t^{p}}}\df\tau=\frac{Ce^{-\frac{\beta}{2\sqrt{t}}}}{\sqrt{\pi t^{p-1}}}.$$

For \eqref{eq:paexpboundnega2},  we apply the same the estimation $-\frac{\tau^2}{4t}\leqslant-\frac{\tau}{2\sqrt{t}}+1$. Then the integral is bounded above by 
$$\int_\beta^{+\infty}\frac{C\tau^2e^{-\frac{\tau}{2\sqrt{t}}}}{\sqrt{4\pi t^{p}}}\df\tau=\frac{Ce^{-\frac{\beta}{2\sqrt{t}}}}{\sqrt{\pi t^{p-1}}}(\beta^2+4\sqrt{t}\beta+8t),$$
by integration by parts.
\end{proof}

\begin{lem}\label{lem:intderheatfinite}
For any $\beta>0$, $$\int_0^{+\infty}\left|\frac{\partial}{\partial t}\left(\frac{e^{-\frac{\beta^2}{4t}}}{\sqrt{4\pi t}}\right)\right|\df t$$ converges.
\end{lem}
\begin{proof}
The value function of $\frac{e^{-\frac{\beta^2}{4t}}}{\sqrt{4\pi t}}$ first increases from $0$ till its maximum, then decreases from the maximum to $0$ again. Then this integral is equal to twice of the maximum value.
\end{proof}

\begin{lem}\label{lem:explinearbound}
Let $f\in L^2(\IR)$, $n\in\IN$, and $q,p\in\IR$. Then
\begin{equation}\label{eq:explinear2n}
\int_0^1\int_{\IR}\frac{|x|^n\sqrt{t^q}e^{-\frac{|x|}{2\sqrt{t}}}}{\sqrt{t^p}}f(x)\df x\df t
\end{equation}
converges when $n+q\geqslant p-2$.
\end{lem}
\begin{proof}First by substitution, we can compute
$$
\int_0^{+\infty}x^ne^{-ax}\df x=\frac{n!}{a^{n+1}}.
$$

Again by H\"older's inequality, \eqref{eq:explinear2n} is 
\begin{align*}
&\leqslant\int_0^12\sqrt{t^{q-p}}\left(\int_0^{+\infty}(2n)!\sqrt{t^{2n+1}}\right)^{\frac{1}{2}}\Vert f\Vert_{L^2}\df x\df t\\
&=2\sqrt{(2n)!}\Vert f\Vert_{L^2}\cdot\int_0^1 t^{\frac{2n+2q-2p+1}{4}}\df t,
\end{align*}
which is finite as $\frac{2n+2q-2p+1}{4}\geqslant-\frac{3}{4}$ in the given range.
\end{proof}

\begin{lem}\label{lem:convolutionlimit}
Let $f\in L_{\mathrm{loc}}^{\infty}(\IR_0^+)$ and $g\in L^1(\IR_0^+)$ such that $\lim_{t\apc+\infty}f(t)=f_{\infty}$ exists. Then the limit of the convolution of $f$ and $g$ satisfies
$$\lim_{t\apc+\infty}\int_0^tf(t-\tau)g(\tau)\df\tau=f_{\infty}\cdot\int_0^{+\infty}g(t)\df t.$$
\end{lem}

\begin{proof}
Because the limit of $f$ exists when $t\apc+\infty$, $f$ in fact belongs to $L^{\infty}(\IR_0^+)$. For any $\varepsilon>0$, find $t_0$ such that $|f(t)-f_{\infty}|<\varepsilon$ for all $t>t_0$. Partition the convolution into two parts,
\begin{equation}\label{eq:convpartition}
\int_0^{t-t_0}f(t-\tau)g(\tau,1)\df\tau+\int_{t-t_0}^tf(t-\tau)g(\tau)\df\tau.
\end{equation}
The first integral of \eqref{eq:convpartition} can be bounded by
$$\left|\int_0^{t-t_0}f(t-\tau)g(\tau)\df\tau-\int_0^{t-t_0}f_{\infty}g(\tau)\df\tau\right|<\varepsilon\Vert g\Vert_{L^1}.$$
The second integral of \eqref{eq:convpartition} is bounded by $$\int_{t-t_0}^t \left|f(t-\tau)g(\tau)\df\tau\right|\leqslant\int_{t-t_0}^t|g(\tau)|\df\tau\cdot\Vert f\Vert_{L^\infty}.$$
As $g$ is absolutely integrable, then for all large $t$, we have $\int_{t-t_0}^t |g(\tau)|\df\tau<\varepsilon$. This shows that the limit exists. 
\end{proof}

\begin{lem}\label{lem:c00l1counter}
For any $f\in L^1$, there exists $g\in C_0^0$ such that $f\ast g\not\in L^1$.
\end{lem}
\begin{proof}
For simplicity, we just prove the statement with on $\IR$. With some modifications we can recover the desired results back on $\IR_0^+$.

\lemref{lem:convolutionlimit} shows that the convolution between an $L^1$-function and a $C_0^0$-function is in $C_0^0$. By general theory of Fourier transform, the Fourier transfomr of an $L^1$-function is a $C_0^0$-function (see Theorem 8.22 on Page 249 of \cite{folland-analysis}). Let $F$ be the Fourier transform of $f$. $F$ is $C_0^0$.

For example, take $g(t):=\frac{\sin(t)}{t}$, which has Fourier transform $G(\xi):=\theta(\xi+1)\theta(-\xi+1)$. The Fourier transform of $f\ast g$ is $F\cdot G$, which is not even continuous so it can not be the Fourier transform of any $L^1$-function.
\end{proof}

\section{Notations}
\begin{multicols}{2}
\noindent$\IN:=\{0,1,2,3,\cdots\}$\\
$\IZ:=\{0,\pm1,\pm2,\cdots\}$\\
$\IR^+:=\{r\in\IR\mid r>0\}$\\
$\IR^+_0:=\{r\in\IR\mid r\geqslant0\}$\\
$f\ast g(t):=\int_0^tf(t-\tau)g(\tau)\df \tau$\\
$C^0_b(\Omega):=L^\infty(\Omega)\cap C^0(\Omega)$\\
$C^0_0(\Omega):=\{f\in C^0(\Omega)\mid f \text{ vanishes at infinity}\}$
\end{multicols}

\newpage
\bibliographystyle{abbrvnat}
\bibliography{ref}

@article{Peng2023, 
    title={Quality of approximating a mass-emitting object by a point source in a diffusion model}, 
    volume={151}, 
    DOI={10.1016/j.camwa.2023.10.034}, 
    journal={Computers \& Mathematics with Applications}, 
    publisher={Elsevier BV}, author={Peng, Qiyao and Hille, Sander C.}, 
    year={2023}, 
    month=dec, 
    pages={491-507} 
}

@article{Peng2024-fv,
  title         = "Using multiple Dirac delta points to describe inhomogeneous
                   flux density over a cell boundary in a single-cell diffusion
                   model",
  author        = "Peng, Qiyao and Hille, Sander C",
  year          =  "2024",
  url     = "https://arxiv.org/abs/2401.16261",
}

@article{AMQ:04,
title={Semilinear parabolic equations involving measures and low regularity data},
  author={Amann, H and Quittner, P},
  journal={Transactions of the American Mathematical Society},
  volume={356},
  number={3},
  pages={1045--1119},
  year={2004}
}

@article{Amann2001,
title={Linear parabolic problems involving measures},
  author={Amann, Herbert},
  journal={Real Academia de Ciencias Exactas, Fisicas y Naturales. Revista. Serie A, Matematicas},
  volume={95},
  number={1},
  pages={85--119},
  year={2001},
  publisher={Real Academia de Ciencias Exactas Fisicas y Naturales}
}

@book{Salsa2022partial,
  title={Partial Differential Equations in Action; From Modelling to Theory},
  author={Salsa, Sandro and Verzini, Gianmaria},
  edition={2nd},
  year={2022},
  publisher={Springer}
}

@book{Amann1995,
  title={Linear and Quasilinear Parabolic Problems},
  author={Amann, Herbert},
  volume={1},
  year={1995},
  publisher={Birkh\"auser Verlag}
}

@book{Adams:2003,
    title={Sobolev Spaces},
    author={Adams, Robert A. and Fournier, John J.F.},
    year={2003},
    edition={2nd},
    publisher={Elsevier Science: Oxford}
}

@article{pena2024cellular,
  title={Cellular and molecular mechanisms of skin wound healing},
  author={Pe{\~n}a, Oscar A and Martin, Paul},
  journal={Nature Reviews Molecular Cell Biology},
  volume={25},
  number={8},
  pages={599--616},
  year={2024},
  publisher={Nature Publishing Group UK London}
}

@article{Klotz1984,
  title = {Mathematical models for ligand-receptor binding. Real sites,  ghost sites.},
  volume = {259},
  ISSN = {0021-9258},
  DOI = {10.1016/s0021-9258(18)90927-0},
  number = {16},
  journal = {Journal of Biological Chemistry},
  publisher = {Elsevier BV},
  author = {Klotz,  I M and Hunston,  D L},
  year = {1984},
  month = aug,
  pages = {10060-10062}
}

@article{Finlay2020,
  title = {100 years of modelling ligand-receptor binding and response: A focus on GPCRs},
  volume = {177},
  ISSN = {1476-5381},
  DOI = {10.1111/bph.14988},
  number = {7},
  journal = {British Journal of Pharmacology},
  publisher = {Wiley},
  author = {Finlay,  David B. and Duffull,  Stephen B. and Glass,  Michelle},
  year = {2020},
  month = feb,
  pages = {1472-1484}
}

@article{Yang2025,
  title = {Approximation of a compound-exchanging cell by a Dirac point},
  volume = {59},
  ISSN = {2405-8963},
  DOI = {10.1016/j.ifacol.2025.03.014},
  number = {1},
  journal = {IFAC-PapersOnLine},
  publisher = {Elsevier BV},
  author = {Yang,  Xiao and Peng,  Qiyao and Hille,  Sander C.},
  year = {2025},
  pages = {73-78}
}

@article{Volterra-Naito,
title = {Characterizations of linear Volterra integral equations with nonnegative kernels},
journal = {Journal of Mathematical Analysis and Applications},
volume = {335},
number = {1},
pages = {298-313},
year = {2007},
issn = {0022-247X},
doi = {10.1016/j.jmaa.2007.01.070},
author = {Toshiki Naito and Jong Son Shin and Satoru Murakami and Pham Huu Anh Ngoc}
}

@book{Arendt-Batty,
title = {Vector-valued Laplace Transforms and Cauchy Problems},
year = {2011},
issn = {1017-0480},
doi = {10.1007/978-3-0348-0087-7},
author = {Wolfgang Arendt and Charles J.K. Batty and Matthias Hieber and Frank Neubrander},
publisher = {Birkh\"auser Basel} 
}

@book{oberhettinger-laplace,
  author    = {Fritz Oberhettinger and Larry Badii},
  title     = {Tables of Laplace Transforms},
  publisher = {Springer-Verlag},
  year      = {1973},
  isbn      = {978-1-4419-7645-1},
  doi = {10.1007/978-3-642-65645-3}
}

@article{DIBLIK1998200,
title = {Asymptotic Representation of Solutions of Equation {$\dot y(t)=\beta(t)[y(t)-y(t-\tau(t))$]}},
journal = {Journal of Mathematical Analysis and Applications},
volume = {217},
number = {1},
pages = {200-215},
year = {1998},
issn = {0022-247X},
doi = {10.1006/jmaa.1997.5709},
author = {Josef Dibl\`ik}
}

@book{carslaw_heatconduction,
  author    = {H. S. Carslaw},
  title     = {Introduction to the mathematical theory of the conduction of heat in solids},
  publisher = {The Macmillan Co. of Canada, Ltd.},
  year      = {1921}
}

@book{watson_bessel,
  author    = {G. N. Watson},
  title     = {A Treatise on the Theory of Bessel Functions},
  edition   = {2nd},
  publisher = {Cambridge University Press},
  year      = {1966},
  address   = {Cambridge}
}

@book{folland-analysis,
  author    = {Gerald B. Folland},
  title     = {Real Analysis: Modern Techniques and Their Applications},
  edition   = {2nd},
  publisher = {John Wiley \& Sons},
  year      = {1999},
  isbn      = {0-471-31716-0}
}

@ARTICLE{Sznurkowska2022-ts,
  title     = "The gate to metastasis: key players in cancer cell intravasation",
  author    = "Sznurkowska, Magdalena K and Aceto, Nicola",
  journal   = "FEBS J.",
  publisher = "Wiley",
  volume    =  289,
  number    =  15,
  pages     = "4336--4354",
  month     =  aug,
  year      =  2022,
  copyright = "http://creativecommons.org/licenses/by-nc/4.0/",
  language  = "en"
}

@ARTICLE{Di_Russo2024-zr,
  title     = "Beyond the barrier: the immune-inspired pathways of tumor
               extravasation",
  author    = "Di Russo, Sara and Liberati, Francesca Romana and Riva, Agnese
               and Di Fonzo, Federica and Macone, Alberto and Giardina, Giorgio
               and Arese, Marzia and Rinaldo, Serena and Cutruzzol{\`a},
               Francesca and Paone, Alessio",
  journal   = "Cell Commun. Signal.",
  publisher = "Springer Science and Business Media LLC",
  volume    =  22,
  number    =  1,
  pages     = "104",
  month     =  feb,
  year      =  2024,
  copyright = "https://creativecommons.org/licenses/by/4.0",
  language  = "en"
}

@article{Qiu2025,
  title = {Traceability of Water Pollution: An Inversion Scheme via Dynamic Complex Geometrical Optics Solutions},
  volume = {57},
  ISSN = {1095-7154},
  DOI = {10.1137/23m1619605},
  number = {1},
  journal = {SIAM Journal on Mathematical Analysis},
  publisher = {Society for Industrial & Applied Mathematics (SIAM)},
  author = {Qiu,  Lingyun and Wang,  Zhongjing and Yu,  Hui and Yu,  Shenwen},
  year = {2025},
  month = jan,
  pages = {286-305}
}

@article{Gjerde2019,
  title = {A singularity removal method for coupled 1D-3D flow models},
  volume = {24},
  ISSN = {1573-1499},
  DOI = {10.1007/s10596-019-09899-4},
  number = {2},
  journal = {Computational Geosciences},
  publisher = {Springer Science and Business Media LLC},
  author = {Gjerde,  Ingeborg G. and Kumar,  Kundan and Nordbotten,  Jan M.},
  year = {2019},
  month = dec,
  pages = {443-457}
}

@ARTICLE{Asghar2024-ii,
  title         = "A multi-dimensional mathematical model for surface exerted
                   point forces in elastic media",
  author        = "Asghar, Sabia and Peng, Qiyao and Vermolen, Fred J",
  month         =  oct,
  year          =  2024,
  copyright     = "http://creativecommons.org/licenses/by/4.0/",
  archivePrefix = "arXiv",
  primaryClass  = "math.NA",
  eprint        = "2410.10436"
}

@book{gyori1991oscillation,
  title={Oscillation theory of delay differential equations: with applications},
  author={Gy{\"o}ri, Istv{\'a}n and Ladas, Gerasimos},
  year={1991},
  publisher={Oxford University Press}
}

@article{CGH96,
  title={On the Lambert {W} Function},
  author={R. M. Corless1 and G. H. Gonnet and D. E. G. Hare and D. J. Jeffrey1 and D. E. Knuth},
  year={1996},
  journal = {Advances in Computational Mathematics,
Vol},
  volume = {5}
}

@article{Burman2014,
  title = {CutFEM: Discretizing geometry and partial differential equations},
  volume = {104},
  ISSN = {1097-0207},
  DOI = {10.1002/nme.4823},
  number = {7},
  journal = {International Journal for Numerical Methods in Engineering},
  publisher = {Wiley},
  author = {Burman,  Erik and Claus,  Susanne and Hansbo,  Peter and Larson,  Mats G. and Massing,  André},
  year = {2014},
  month = dec,
  pages = {472-501}
}

@article{Xu2024,
  title = {A weighted shifted boundary method for immersed moving boundary simulations of Stokes' flow},
  volume = {510},
  ISSN = {0021-9991},
  DOI = {10.1016/j.jcp.2024.113095},
  journal = {Journal of Computational Physics},
  publisher = {Elsevier BV},
  author = {Xu,  Danjie and Colom\'es,  Oriol and Main,  Alex and Li,  Kangan and Atallah,  Nabil M. and Abboud,  Nabil and Scovazzi,  Guglielmo},
  year = {2024},
  month = aug,
  pages = {113095}
}

@article{Ashyraliyev2008,
  title = {On the numerical solution of diffusion-reaction equations with singular source terms},
  volume = {216},
  ISSN = {0377-0427},
  url = {http://dx.doi.org/10.1016/j.cam.2007.04.017},
  DOI = {10.1016/j.cam.2007.04.017},
  number = {1},
  journal = {Journal of Computational and Applied Mathematics},
  publisher = {Elsevier BV},
  author = {Ashyraliyev,  M. and Blom,  J.G. and Verwer,  J.G.},
  year = {2008},
  month = jun,
  pages = {20-38}
}

@article{Koeppl_2014,
 author = {K\"oppl, Tobias and Wohlmuth, Barbara},
 doi = {10.1137/130927619},
 issn = {1095-7170},
 journal = {SIAM J. Numer. Anal.},
 month = {January},
 number = {4},
 pages = {1753-1769},
 publisher = {Society for Industrial & Applied Mathematics (SIAM)},
 title = {Optimal A Priori Error Estimates for an Elliptic Problem with Dirac Right-Hand Side},
 url = {http://dx.doi.org/10.1137/130927619},
 volume = {52},
 year = {2014}
}

@article{Banerjee2022,
  title = {Nonlocal Reaction-Diffusion Equations in Biomedical Applications},
  volume = {70},
  ISSN = {1572-8358},
  DOI = {10.1007/s10441-022-09436-4},
  number = {2},
  journal = {Acta Biotheoretica},
  publisher = {Springer Science and Business Media LLC},
  author = {Banerjee,  M. and Kuznetsov,  M. and Udovenko,  O. and Volpert,  V.},
  year = {2022},
  month = mar 
}

@ARTICLE{Banerjee2020-bt,
  title     = "Prey-predator model with nonlocal and global consumption in the
               prey dynamics",
  author    = "Banerjee, Malay and {,Department of Mathematics \& Statistics,
               Indian Institute of Technology Kanpur, Kanpur - 208016, India}
               and Mukherjee, Nayana and Volpert, Vitaly and {,Institut Camille
               Jordan, UMR 5208 CNRS, University Lyon 1, 69622 Villeurbanne,
               France} and {,INRIA, Universit{\'e} de Lyon, Universit{\'e} Lyon
               1, Institut Camille Jordan, 43 Bd. du 11 Novembre 1918, 69200
               Villeurbanne Cedex, France} and {,Peoples Friendship University
               of Russia (RUDN University), 6 Miklukho-Maklaya St, Moscow,
               117198, Russian Federation}",
  journal   = "Discrete Contin. Dyn. Syst. Ser. S",
  publisher = "American Institute of Mathematical Sciences (AIMS)",
  volume    =  13,
  number    =  8,
  pages     = "2109--2120",
  year      =  2020,
  language  = "en"
}

@article{Sikic:1994,
	Author = {H.~\v{S}iki\'c},
	Journal = {Semigroup Forum},
	Pages = {273--302},
	Title = {Nonlinear perturbations of positive semigroups},
	Volume = {48},
	Year = {1994}
}

@article{Canizo_ea:2012,
    title={Measure solutions for some models in population dynamics},
    author={Ca{\~n}izo, J.A. and Carillo, J.A. and Cuadrado, S.},
    year={2013},
    journal={Acta Appl. Math.},
    volume={123},
    pages={141--156},
    doi={10.1007/s10440-012-9758-3},
    url={https://doi.org/10.1007/s10440-012-9758-3}
}

@article{Hille_ea:2025,
    title={Invariance properties of the solution for measure-valued semilinear transport equations},
    author={Hille, Sander C and Lyons, Rainey and Muntean, Adrian},
    year={2025},
    journal={Analysis and Mathematical Physics},
    volume={15},
    pages={93},
    doi={10.1007/s13324-025-01093-3},
    url={https://doi.org/10.1007/s13324-025-01093-3}
}

@book{Diekmann_ea:1995,
    title={Delay Equations; Functional-, Complex- and Nonlinear Analysis},
    author={Diekmann, Odo and van Gils, Stefan A and Verduyn Lunel, Sjoerd M and Walter, Hans {-O}},
    year={1995},
    publisher={Springer-Verlag}
}

@book{DUA01,
    title={Application of the electrical conductivity of concentrated electrolyte solutions to industrial process control and design: from experimental measurement towards prediction through modelling},
    journal = {TrAC Trends in Analytical Chemistry},
    volume = {20},
    number = {2},
    pages = {65-78},
    year = {2001},
    issn = {0165-9936},
    doi = {https://doi.org/10.1016/S0165-9936(00)00081-9},
    author = {Alberto {de Diego} and Aresatz Usobiaga and Luis Angel Fern\'andez and Juan Manuel Madariaga}
}
\end{document}